\documentclass{birkjour_t2}

\usepackage{amsfonts,amssymb}
\usepackage{amsmath}
\usepackage{graphicx}
\usepackage{cite}
\usepackage{enumerate}
\usepackage{appendix}

\usepackage[backref,colorlinks,linkcolor=blue,anchorcolor=green,citecolor=blue]{hyperref}
\numberwithin{equation}{section}

 \renewcommand{\theequation}{\Alph{section}.\arabic{equation}}

\renewcommand{\theequation}{\arabic{section}.\arabic{equation}}

 \newtheorem{theorem}{Theorem}[section]

 \newtheorem{definition}[theorem]{Definition}
 \newtheorem{remark}[theorem]{Remark}
  
  \newtheorem{lem}{\noindent Lemma}[section]
  \newtheorem{prop}{\noindent Proposition}[section]

  \newtheorem{assumption}{\noindent Assumption}[section]

\allowdisplaybreaks
\begin{document}

\title[Local null controllability of the  complete L.-B.]{Local null controllability of the  complete N-dimensional Ladyzhenskaya-Boussinesq model}

\author[J. C. Barreira]{Jo\~{a}o Carlos Barreira}\email{jcbarreira95@gmail.com}\thanks{Jo\~{a}o Carlos Barreira. Chair for Dynamics, Control, Machine Learning, and Numerics, Alexander von Humboldt-Professorship, Department
of Mathematics, Friedrich-Alexander-Universit\"at Erlangen-N\"urnberg, 91058 Erlangen, Germany and Instituto de Matem\'{a}tica e Estat\'{i}stica, UFF, Niter\'{o}i, 24210-201, RJ, Brasil, 0009-0009-4607-8966.}
\author[J. L\'{i}maco]{Juan L\'{i}maco}\email{
jlimaco@id.uff.br}\thanks{Juan  L\'{i}maco. Instituto de Matem\'{a}tica e Estat\'{i}stica, UFF, Niter\'{o}i, 24210-201, RJ, Brasil, 0000-0001-8695-8612.}

  \begin{abstract}
      This work investigates both local null controllability and large time null controllability for a class of complete Ladyzhenskaya--Boussinesq systems, where the controls are distributed and supported on small subsets of the domain. We incorporate in the temperature equation the nonlinear term $ (\nu_{0} + \nu_{1} \|\nabla y\|_{L^{p}}^{2})Dy : \nabla y$, which reflects more realistic physical behavior, but also introduces significant analytical and control-theoretic challenges. The analysis is carried out separately for the cases $p = 2$ and $3 < p \leq 6$. These equations model a differential turbulence system with temperature coupling and involve both local and nonlocal nonlinearities -- notably, transport terms and a turbulent viscosity depending on the spatially global energy dissipated by the mean flow. The proof of local null controllability relies on classical techniques, including Carleman estimates and Liusternik’s Inverse Mapping Theorem. Nevertheless, the presence of nonlinearities in both the velocity and temperature equations necessitates careful treatment. For the case \( p = 2 \), we also establish null controllability in large time by analyzing the uncontrolled system's asymptotic behavior and subsequently applying the previously obtained local controllability result.

\end{abstract}


\maketitle

\section{Introduction and main results}\label{intro}

In this work we are interested in studying a system that models viscous flows, where viscosity is in function of the velocity gradient, in which thermal effects are taken into account. We will consider $\Omega\subset\mathbb{R}^{N}$ ($N=2$ or $N=3$) be a non-empty bounded connected open set, with regular boundary $\partial\Omega$ and let $T>0$ be given. We will us denote by $Q$ the cylinder $\Omega\times (0,T)$ with side boundary $\Sigma=\partial\Omega\times (0,T)$.

Let $\omega\subset\Omega$ be a (small) non-empty open set. We denote by $(.,.)$ and $\Vert .\Vert$ respectively the $L^{2}$ scalar product and norm in $\Omega$. We will use $C$ to denote a generic positive constant. Thus, we will study the null controllability for the nonlinear systems:
\begin{equation}\label{lad.boussinesq}
    \left\{
    \begin{array}[c]{lll}
    y_{t}-\nabla\cdot\left( \nu(\nabla y)Dy\right) + (y\cdot\nabla)y + \nabla P = v\tilde{1}_{\omega}+\nu_{0}\theta e_{N}, \,\,\, \nabla\cdot y=0    & \text{in} & Q, \\
    \theta_{t}-\nabla\cdot\left( \nu(\nabla y)\nabla\theta\right) + y\cdot\nabla\theta = v_{0}\tilde{1}_{\omega} +  \nu(\nabla y)Dy:\nabla y &\text{in}& Q,\\
    y(x,t)=0, \theta(x,t)=0 &\text{on}& \Sigma,\\
    y(x,0)=y^{0}(x),\, \theta(x,0)=\theta^{0}(x) &\text{in}& \Omega,
    \end{array}\right.
\end{equation}
where
\begin{equation}\label{nunabla}
    \nu(\nabla y):=\nu_{0}+\nu_{1}\displaystyle\int_{\Omega}\vert\nabla y\vert^{2}dx
\end{equation}
and 
\begin{equation}\label{lad.boussinesq caso extra}
    \left\{
    \begin{array}[c]{lll}
    y_{t}-\nabla\cdot\left( \bar{\nu}(\nabla y)Dy\right) + (y\cdot\nabla)y + \nabla P = v\tilde{1}_{\omega}+\nu_{0}\theta e_{N}, \,\,\, \nabla\cdot y=0    & \text{in} & Q, \\
    \theta_{t}-\nabla\cdot\left( \bar{\nu}(\nabla\theta)\nabla\theta\right) + y\cdot\nabla\theta = v_{0}\tilde{1}_{\omega} +  \bar{\nu}(\nabla y)Dy:\nabla y &\text{in}& Q,\\
    y(x,t)=0, \theta(x,t)=0 &\text{on}& \Sigma,\\
    y(x,0)=y^{0}(x),\, \theta(x,0)=\theta^{0}(x) &\text{in}& \Omega,
    \end{array}\right.
\end{equation}
where $\bar{\nu}(\nabla\varsigma):= \nu_{0}+\nu_{1}\Vert\nabla\varsigma\Vert^{2}_{L^{p}}$, for  $ 3<p\leq 6$, and in both systems
\begin{equation*}
    e_{N}=\left\{ 
    \begin{array}{c}
     (0,1)\,\, \text{if}\,\, N=2, \\
     (0,0,1)\,\, \text{if}\,\, N=3.       
    \end{array}\right.
\end{equation*}
In \eqref{lad.boussinesq} and \eqref{lad.boussinesq caso extra}, $y=y(x,t)$ stands the ``averaged'' velocity field, $\theta=\theta(x,t)$ and $P=P(x,t)$ represent, respectively,  temperature and pressure of a fluid whose particles are in $\Omega$ during the time interval $(0,T)$; $\nu_{0}$ and $\nu_{1}$ are positive constants representing the kinematic viscosity and turbulent viscosity, respectively. $(y^{0},\theta^{0})$ are the initial states, that is to say, the states at time $t=0$; $\tilde{1}_{\omega}\in C^{\infty}_{0}(\Omega)$ such that $0<\tilde{1}_{\omega}\leq 1$ in $\omega$ and $\tilde{1}_{\omega}=0$ outside $\omega$; $Dy$ stands for the symmetrized gradient of $y$: $Dy=\frac{1}{2}(\nabla y +\nabla^{T}y)$ and
\begin{equation}\label{termo quadratico}
Dy:\nabla y := \displaystyle\sum_{i,j=1}^{N}\dfrac{1}{2}\left(\dfrac{\partial y_{j}}{\partial x_{i}} + \dfrac{\partial y_{i}}{\partial x_{j}} \right)\dfrac{\partial y_{i}}{\partial x_{j}}.
\end{equation}

Furthermore, $\omega\times (0,T)$ is the control domain and $v$ (force) and $v_{0}$ (heat sources) represent the controls acting on the system. 

As we are assuming on the right side of the heat equation the quadratic term $\nu(\nabla y)Dy:\nabla y$ or $\bar{\nu}(\nabla y)Dy:\nabla y$ which is related to the work done by viscous forces, the systems \eqref{lad.boussinesq} and \eqref{lad.boussinesq caso extra} can be considered generalizations of the complete Boussinesq model (which corresponds to several conservation laws involving momentum, mass and energy). Moreover, when $\nu_{0}=1$ and $\nu_{1}=0$ in \eqref{lad.boussinesq}, E. Fern\'{a}ndez - Cara et al \cite{CaraJuanDanytrue} proved that such system is locally null controllable. And, when we remove the entire term $\nu(\nabla y)Dy:\nabla y $ from the right side of equation \eqref{lad.boussinesq}, Huaman et al \cite{Huaman} demonstrated that such a system is locally null controllable by means of $N-1$ scalar controls for an arbitrary control domain. The two preceding works cited serve as the primary motivation for the development of this study.

Now notice that, in both \eqref{lad.boussinesq} and \eqref{lad.boussinesq caso extra}, when removing the temperature variable they become a particular case of
\begin{equation}\label{Ladyzhenskaya}
    \hspace{-0.1cm}\left\{\hspace{-0.1cm}
    \begin{array}[c]{lll}
    y_{t}-\nabla\cdot \textbf{T}(y,P) + (y\cdot\nabla)y  = f & \text{in} & Q, \\
    \nabla\cdot y=0    & \text{in} & Q,\\
    \text{etc.},
    \end{array}\right.
\end{equation}
where $f$ is an external force field, $\textbf{T}(y,P) := -P {I} + (\nu_{0} + \nu_{1}\vert Dy\vert^{r-2})Dy$ is the stress tensor with $r>2$ and
\begin{equation*}
    \vert Dy\vert := \left[ \sum_{i,j=1}^{N}\dfrac{1}{2}\left(\dfrac{\partial y_{j}}{\partial x_{i}}+\dfrac{\partial y_{i}}{\partial x_{j}}\right)^{2}\right]^{1/2}.
\end{equation*}

The first mathematical studies on this type of equations were introduced by O. Ladyzhenskaya in the 1960s and can be found in \cite{O.Lady66,O.Lady67, O.Lady68,O.Lady69}. Just as J.‑L. Lions considered in his relevant book \cite{Lions69} the case in which $Dy$ is replaced by $\nabla y$, that is, when  the tensor stress is of the form $\mathbf{{T}_{1}}(y,P) = -P { I} + (\nu_{0} + \nu_{1}\vert \nabla y \vert^{r-2})\nabla y$ and obtained important results of existence, uniqueness and regularity of solutions. For some regularity properties for the solutions of \eqref{Ladyzhenskaya}, see for instance \cite{Veiga}.

When $N=r=3$ the model \eqref{Ladyzhenskaya} is the classical turbulence model approached by  Smagorinsky in \cite{Smagorinsky}.

For additional investigations within the scope of control theory on variations of the \eqref{Ladyzhenskaya} model, we recommend: \cite{PitagorasJuanDenilson} in which analyzed the null controllability property when the stress tensor is the same as that considered by J.-L. Lions, that is, dependent on the state gradient; E. Fern\'{a}ndez - Cara et al \cite{CaraLimacoMenezes}, where the existence of local null controls was guaranteed for the case in which stress tensor is equal to $-PI + (\nu_{0}+\nu_{1}(\Vert Dy\Vert^{2}))Dy)$ with  $\nu_{1}$ being a continuously differentiable function, that is, $0\leq \nu_{1}\leq C$ and $\vert \nu_{1}^{\prime}\vert < C$. In this work, the authors also provided a numerical approximation and illustrated the behavior of the algorithm through examples; again, \cite{Huaman}, in which local null controllability was also established by means of $N-1$ scalar controls for an arbitrary control domain when \eqref{nunabla} is considered in $\mathbf{T}(y,P )$, more specifically, they replaced $\vert Dy\vert^{r-2}$ by $\nu(\nabla y)$. Furthermore, Guerrero and Takahashi \cite{GuerreroTakeo} addressed a similar controllability problem by considering the term $\|\text{curl}(y)\|^{2}$ instead of $|Dy|^{r-2}$, and demonstrated controllability via trajectories. To obtain this result, they established a Carleman estimate for the adjoint of a linear system incorporating a nonlocal spatial term.

We will now list some articles present in the literature that provided relevant controllability results for the Boussinesq system. 

On the exact local controllability of trajectories, \cite{Guerrero} dealt with the Boussinesq system with $N+1$ distributed scalar controls supported in small sets.  In this interesting work, firstly, a Carleman inequality was proved for a linearized version of the Boussinesq system, which leads to its null controllability at any time $T>0$. And from this, the result of exact controllability of trajectories was obtained. Still this context we mention, \cite{EnriqueGuerreroImanuvilovPuel} in which the authors proved that through some hypotheses were imposed on the control domain and the trajectories, the Boussinesq system is locally exactly controllable by $N-1$ scalar controls at a time $T>0$ to the trajectories. Moreover, removing the geometric conditions imposed by \cite{EnriqueGuerreroImanuvilovPuel}, \cite{Carreno2012} concludes the exact local controllability to a particular class of trajectories with internal controls having two vanishing components. Also \cite{FursikovImanuvilov98} and \cite{FursikovImanuvilov99} proved, respectively, the local exact boundary controllability to the trajectories of the Boussinesq system with $N+1$ scalar controls acting over the whole boundary in a bounded domain of $\mathbb{R}^{N}$ $(N=2\, \text{or}\, 3)$ with $C^{\infty}$-boundary and the local exact controllability to the same trajectories with $N+1$ scalar distributed controls, when the torus is the domain.

Considering a generalized Boussinesq equation in a periodic domain, a unit circle in the plane, \cite{Zhang} showed that depending on the location of the control, whether in the entire domain or in a subdomain, and the amplitude of the initial and terminal states it is possible conclude that the system is globally exactly controllable.

Finally, regarding the concept of insensitive controls, that is, controls that insensitize some functional that depends on the velocity field, see  \cite{Carreno2017} which showed for the Boussinesq system, without any control used on the temperature equation, the existence of the insensitizing controls such that  the control acting on the fluid equation can be chosen to have one vanishing component. And, for the system with controls acting on both equations, \cite{CarrenoGuerreroGueye} demonstrated the existence of insensitive controls with two vanishing components, when the case is three-dimensional  and  for the case two-dimensional the authors concluded that no is required control into the velocity equation.

The following vector spaces, frequently used in the context of incompressible fluids, which will be used throughout the article are:

\[
H:=\lbrace u\in L^{2}(\Omega)^{N}: \nabla\cdot u=0\, \text{in}\, \Omega, u\cdot\eta=0\, \text{on}\, \partial\Omega\rbrace
\]
and
\[
V^{p}:=\lbrace u\in W^{1,p}_{0}(\Omega)^{N}:\nabla\cdot u=0\, \text{in}\, \Omega\rbrace,
\]
where $\eta$ is the normal vector exterior to $\partial\Omega$ and $W^{1,p}_{0}(\Omega)$ is the closure of the space of test functions in $\Omega$, $\mathcal{D}(\Omega)$, in $W^{1,p}(\Omega)$(the standard Sobolev space). In particular, when $p=2$ we will denote $V=V^{p}$.

For $N=2$, $y^{0}\in V$, $\theta^{0}\in W_{0}^{1,3/2}(\Omega)$, and any $v\in L^{2}(\omega\times (0,T))^{N}$, $v_{0}\in L^{2}(\omega\times (0,T))$ sufficiently small in their respective spaces, (\ref{lad.boussinesq}) possesses exactly a strong solution $(y,p,\theta)$ with 
\begin{equation}\label{regularidade das sol.}
    \left\{
    \begin{array}{l}
     y\in L^{2}(0,T;H^{2}(\Omega)^{N}\cap V)\cap C^{0}([0,T];V),\, \, y_{t}\in L^{2}(0,T;H)    \vspace{0.1cm}  \\
     \theta\in L^{2}(0,T;W^{2,3/2}(\Omega)),\,\, \theta_{t}\in L^{2}(0,T;L^{3/2}(\Omega)).     
    \end{array}\right.
\end{equation}
For $N=3$, this is true if $y^{0}$, $\theta^{0}$, $v$ and $v_{0}$ are sufficiently small in their respective spaces. The proof of these statements can be seen later in the Appendix \ref{Appendix} and will be used opportunely to achieve a result of local null controllability  in a long time, as stated in Theorem \ref{Teo large time}.
\begin{definition}
    Let any non-empty open set $\omega\subset\Omega$. It will be said that \eqref{lad.boussinesq} (respectively \eqref{lad.boussinesq caso extra}) is locally null-controllable at time $T>0$ if there exists $\delta>0$ such that, for every $(y^{0},\theta^{0})\in V\times W_{0}^{1,3/2}(\Omega)$ (respectively $(y^{0},\theta^{0})\in V^{p}\times W_{0}^{1,p}(\Omega)$) with
\[
\Vert (y^{0},\theta^{0})\Vert_{V\times W_{0}^{1,3/2}(\Omega)}< \delta\, (\text{respectively}\, \Vert (y^{0},\theta^{0})\Vert_{V^{p}\times W_{0}^{1,p}(\Omega)}< \delta),
\]
there exists controls $v\in L^{2}(\omega\times (0,T))^{N}$, $v_{0}\in L^{2}(\omega\times (0,T))$ and associated solutions $(y,p,\theta)$ 
satisfying 
\begin{equation}\label{nulo}
    y(x,T)=0\,\,\, \text{and}\,\,\, \theta(x,T)=0\,\,\, \text{in}\,\,\, \Omega.
\end{equation}
\end{definition}

Thus, the main results of this paper are presented as follows:

\begin{theorem}\label{Teo principal control nulo}
The nonlinear system \eqref{lad.boussinesq} is locally null-controllable at any $T>0$.
\end{theorem}
\begin{theorem}\label{Teo extra control nulo}
The nonlinear system \eqref{lad.boussinesq caso extra} is locally null-controllable at any $T>0$.
\end{theorem}

In order to prove Theorems \ref{Teo principal control nulo} and \ref{Teo extra control nulo}, we will first see a result of null controllability for the linear system associated with (\ref{lad.boussinesq}) and \eqref{lad.boussinesq caso extra}
\begin{equation}
    \left\{
    \begin{array}{lll}
     \mathcal{L}_{1}y+ \nabla P = v\tilde{1}_{\omega} + \nu_{0}\theta e_{N} + F_{1},\,\,\, \nabla\cdot y = 0   & \text{in} & Q, \\
      \mathcal{L}_{2}\theta = v_{0}\tilde{1}_{\omega} + F_{2}    & \text{in} & Q,\\
      y(x,t)=0,\,\,\, \theta(x,t)=0 &\text{on}& \Sigma,\\
      y(x,0)=y^{0}(x),\,\,\, \theta(x,0)=\theta^{0}(x) &\text{in}& \Omega,
\end{array}\right.\label{lad.Bous.Linear}
\end{equation}
where $\mathcal{L}_{1}y:=y_{t} -\nu_{0}\Delta y$, $\mathcal{L}_{2}\theta:=\theta_{t} -\nu_{0}\Delta\theta$ and $F_1$ and $F_2$ are appropriate functions decaying exponentially when $t\to T^{-}$ (see Proposition \ref{proposição 2.1} below).

Once the null controllability of \eqref{lad.Bous.Linear} has been proven, we will define a Banach space that will contain a remodeling of the null controllability problem. In other words, we rewrite the null controllability property of \eqref{lad.boussinesq} and \eqref{lad.boussinesq caso extra}), separately, as abstract equations (see \eqref{Mapa F} and \eqref{Mapa F extra}) in well chosen spaces of ``admissible'' state-controls; see \eqref{espaço EN} and \eqref{espaço ZN} for \eqref{lad.boussinesq} and \eqref{espaço UN} and \eqref{espaço RN} for \eqref{lad.boussinesq caso extra}. In particular, through the definitions applied to the equations and ``admissible'' spaces, it is possible to show that such applications are well defined and $C^{1 }$ and, also its derivatives analyzed at zero are surjective. This will allow us to achieve the local null controllability of the systems in question.

Furthermore, when $N=2$ we also show that for certain conditions in the initial data it is possible to obtain a result of local null controllability in a large time for the solutions of the system \eqref{lad.boussinesq}. To do this, we will show that such solutions (for $v=v_0=0$)  have asymptotic behavior when $t\to\infty$. Therefore, we have the following theorem:

\begin{theorem}\label{Teo large time}[Large time Null-Controllability]
For $N=2$, let $(y^{0},\theta^{0})\in V\times H^{1}_{0}(\Omega)$ and $r>0$ a positive constant given by Theorem \ref{existence and uniq.} (see Appendix \ref{Appendix}) such that $\Vert(y^{0},\theta^{0})\Vert_{V\times H^{1}_{0}(\Omega)}<r$, then there exists a sufficiently large time $T > 0$ such that the
nonlinear system \eqref{lad.boussinesq} is null-controllable at $T$.
\end{theorem}

This paper is organized as follows. In Section \ref{Sec.1}, we recall some well-known results concerning parabolic problems, Stokes systems, and Carleman estimates, which play a crucial role in establishing the null controllability of system \eqref{lad.Bous.Linear}. Section \ref{Sec.2} is devoted to proving the null controllability of system \eqref{lad.Bous.Linear}, following the approach in \cite{Guerrero}. We also derive weighted Banach space estimates for the solutions of the linear system \eqref{lad.Bous.Linear} and for the corresponding controls \( v \) and \( v_0 \), which will be instrumental in Section \ref{Sec.3} for establishing the null controllability of systems \eqref{lad.boussinesq} and \eqref{lad.boussinesq caso extra}. In Section \ref{Sec.3}, we prove the null controllability of the aforementioned systems using Liusternik's Inverse Mapping Theorem. Section \ref{Sec.4} is dedicated to the proof of Theorem \ref{Teo large time}, which is based on a lemma ensuring that, under suitable conditions on the initial data, the solution of system \eqref{lad.boussinesq} (without controls $v$ and $v_0$) exhibits an asymptotic behavior as $t \to \infty$. Additionally, Section \ref{Comentarios Adicionais} contains further remarks and a discussion of some interesting open problems. Finally, Appendix \ref{Appendix} presents auxiliary results on the existence and uniqueness of solutions for system \eqref{lad.boussinesq}, along with the proof of the lemma stated in Section \ref{Sec.4}.

\section{Some previous results}\label{Sec.1}

Our goal in the present section is to present well-posedness results for parabolic problems and Stokes systems, as well as Carleman estimates for the adjoint system of \eqref{lad.Bous.Linear} which is given by
\begin{equation}
    \left\{
    \begin{array}{lll}
    \mathcal{L}_{1}^{\ast}\varphi + \nabla \pi = G_{1} ,\,\,\, \nabla\cdot\varphi = 0   & \text{in} & Q, \\
     \mathcal{L}_{2}^{\ast}\psi = \varphi e_{N} + G_{2}    & \text{in} & Q,\\
      \varphi(x,t)=0,\,\,\, \psi(x,t)=0 &\text{on}& \Sigma,\\
      \varphi(x,T)=\varphi^{T}(x),\,\,\, \psi(x,T)=\psi^{T}(x) &\text{in}& \Omega,
\end{array}\right.\label{adjunto de lad. Bous.}
\end{equation}
where $\mathcal{L}_{1}^{\ast}\varphi:= -\varphi_{t} -\nu_{0}\Delta \varphi$, $\mathcal{L}_{2}^{\ast}\psi:= -\psi_{t} -\nu_{0}\Delta \psi$, $\varphi^{T}\in H$, $\psi^{T}\in L^{2}(\Omega)$, $G_{1}\in L^{2}(Q)^{N}$ and $G_{2}\in L^{2}(Q)$.

\subsection{Well-posedness results} The results of this subsection will be applied when we study the null controllability of system \eqref{lad.Bous.Linear} (Section \ref{Sec.2}), since once we have the appropriate regularity for $\theta^{0}$ and $y^{0}$ the results described here can be applied to equation formed by $\eqref{lad.Bous.Linear}_{1}$ and $\eqref{lad.Bous.Linear}_{2}$.

The first lemma we mention here is applied to parabolic equations in $L^{p}-L^{q}$ spaces and its verification can be based on \cite{Robert}:
\begin{lem}\label{parabolic in LpLq} Let $1<r,s<\infty$ and suppose that $\phi^{0}\in W^{1,s}(\Omega)$ and $h\in L^{r}(0,T;L^{s}(\Omega))$. Then the problem
\begin{equation*}
    \left\{\begin{array}{lll}
        \phi_{t}-\Delta\phi = h &\text{in}& Q,\\
        \phi = 0 &\text{on}& \Sigma,\\
        \phi(0)=\phi^{0} &\text{in}& \Omega
    \end{array}\right.
\end{equation*}
admits a unique solution 
\begin{equation*}
    \phi\in W^{1,r}(0,T;L^{s}(\Omega))\cap L^{r}(0,T;W^{2,s}(\Omega)),
\end{equation*}
Furthermore, there exist a constant $C>0$ such that
\begin{equation}\label{estimate de Robert}
    \Vert\phi_{t}\Vert_{L^{r}(0,T;L^{s}(\Omega))} + \Vert\Delta\phi\Vert_{L^{r}(0,T;L^{s}(\Omega))}\leq C(\Vert\phi^{0}\Vert_{W^{1,s}(\Omega)} + \Vert h\Vert_{L^{r}(0,T;L^{s}(\Omega))}).
\end{equation}
\end{lem}
\begin{remark}
Remembering that the Sobolev space $W^{1,r}(0,T;X)$, where $X$ denote a real Banach space with norm $\Vert .\Vert_{X}$, consists of all functions $u\in L^{r}(0,T;X)$ such that $u^{\prime}=u_{t}$ exists in the weak sense and belongs to $L^{r}(0,T;X)$. Furthermore,
\begin{equation*}
   \Vert u\Vert_{W^{1,r}(0,T;X)}:=\left\{ \begin{array}{lll}
   \left[\int_{0}^{T}(\Vert u(t)\Vert_{X}^{r} + \Vert u^{\prime}(t)\Vert^{r}_{X})dt\right]^{1/r} &(1\leq r< \infty)&\vspace{0.1cm}\\
{\operatorname{ess}\sup}_{[0,T]}(\Vert u(t)\Vert_{X} + \Vert u^{\prime}(t)\Vert_{X})  &(r=\infty).&
    \end{array}\right.
\end{equation*}
More details about this space can be found at \cite{Evans}.
\end{remark}

The second result is valid for Stokes systems with homogeneous Dirichlet boundary conditions and can be found in \cite{Temam}:
\begin{lem}\label{Stokes em V}For every $T>0$, $u^{0}\in V$ and $f\in L^{2}(Q)^{N}$, there exists a unique solution $(u,q)\in \left(L^{2}(0,T;H^{2}(\Omega)^{N}\cap V)\cap L^{\infty}(0,T;V)\times L^{2}(0,T;H^{1}(\Omega)\right)$ to the Stokes system 
\begin{equation*}
    \left\{\begin{array}{lll}
      u_{t}-\Delta u + \nabla q = f,\,\,\, \nabla\cdot u=0 &\text{in}& Q,\\
      u=0 &\text{on}&\Sigma,\\
      u(0)= u^{0} &\text{in}&\Omega.
    \end{array}\right.
\end{equation*}
\end{lem}
The next result, proven in \cite{Giga}, concerns the regularity of the solutions of the Stokes system in $L^{p}-L^{q}$ spaces (see also \cite{Guerrero} for additional comments):
\begin{lem}\label{Stokes para Lp} Let $1<p_{1},p_{2}<\infty$ and suppose that $u^{0}\in W^{1,p_{2}}(\Omega)^{N}$ and $f\in L^{p_{1}}(0,T;L^{p_{2}}(\Omega))$. Then, the weak solution $u\in L^{2}(0,T;V)\cap L^{\infty}(0,T;H)$ of system 
\begin{equation*}
    \left\{\begin{array}{lll}
      u_{t}-\Delta u + \nabla q = f,\,\,\, \nabla\cdot u=0 &\text{in}& Q,\\
      u=0 &\text{on}&\Sigma,\\
      u(0) = u^{0} &\text{in}&\Omega
    \end{array}\right.
\end{equation*}
actually verifies, together with a pressure $q$, that
\begin{equation*}
    (u,\nabla q)\in \left(L^{p_{1}}(0,T;W^{2,p_{2}}(\Omega)^{N})\cap W^{1,p_{1}}(0,T;L^{p_{2}}(\Omega)^{N})\right)\times L^{p_{1}}(0,T;L^{p_{2}}(\Omega)^{N}).
\end{equation*}
Moreover, there exists a positive constant $C$ just depending on $\Omega$ such that
\begin{equation*}
    \begin{array}{l}
         \Vert u\Vert_{L^{p_{1}}(0,T;W^{2,p_{2}}(\Omega)^{N})\cap W^{1,p_{1}}(0,T;L^{p_{2}}(\Omega)^{N})}+\Vert\nabla q\Vert_{L^{p_{1}}(0,T;L^{p_{2}}(\Omega)^{N})}\\
         \leq C(\Vert f\Vert_{L^{p_{1}}(0,T;L^{p_{2}}(\Omega)^{N})}+\Vert u^{0}\Vert_{W^{1,p_{2}}(\Omega)^{N}}).
    \end{array}
\end{equation*}
\end{lem}

\subsection{Carleman estimates}
We dedicate this subsection to Carleman estimate, which will be fundamental to achieving the null controllability of \eqref{lad.Bous.Linear} for suitable $F_1$ and $F_2$ (Section \ref{Sec.2}).

Let's introduce a new non-empty open set $\omega_{0}\Subset\omega$. Due to Fursikov and Imanuvilov \cite{Fursikov} we have the following result:
\begin{lem}
There exists a function $\eta^{0}\in C^{2}(\overline{\Omega})$ satisfying
\begin{equation*}
    \left\{
    \begin{array}{ll}
     \eta^{0}(x)>0, & \forall x\in \Omega,  \\
      \eta^{0}(x)=0,  & \forall x\in\partial\Omega,\\
      \vert\nabla\eta^{0}(x)\vert>0, & \forall x\in\overline{\Omega}\setminus\omega_{0}.
    \end{array}\right.
\end{equation*}
\end{lem}

Let us introduce the function $\ell\in C^{\infty}([0,T])$ such that
\begin{equation*}
\ell(t) = \left\{
    \begin{array}{ll}
      \dfrac{T^{2}}{4} , & 0\leq t\leq T/2, \vspace{0.1cm} \\
       t(T-t), & T/2 < t\leq T.  
    \end{array}\right.
\end{equation*}
Thus, for all $\lambda>0$ and $m>4$, we consider the following weight functions:
\begin{equation*}
\begin{array}{c}  \alpha(x,t)=\dfrac{e^{5/4\lambda m\Vert\eta^{0}\Vert_{\infty}}-e^{\lambda(m\Vert\eta^{0}\Vert_{\infty}+\eta^{0}(x))}}{\ell(t)^{4}}, \,\,\,  \xi(x,t)=\dfrac{e^{\lambda(m\Vert\eta^{0}\Vert_{\infty}+\eta^{0}(x))}}{\ell(t)^{4}},\vspace{0.13cm}\\
   \alpha^{\ast}(t)=\displaystyle\max_{x\in\overline{\Omega}}\alpha(x,t),\,\,\,   \xi^{\ast}(t)=\displaystyle\min_{x\in\overline{\Omega}}\xi(x,t),\vspace{0.13cm}\\
   \hat{\alpha}(t)=\displaystyle\min_{x\in\overline{\Omega}}\alpha(x,t), \,\,\, \hat{\xi}(t)=\displaystyle\max_{x\in\overline{\Omega}}\xi(x,t).
\end{array}
\end{equation*}
The constant $m$ will be chosen large enough, in particular such that
\begin{equation}    \label{des.16>15}
{18\hat{\alpha} > 17\alpha^{\ast}}\,\,\, 
 \,\,\,\text{in}\,\,\, (0,T).
\end{equation}

We will present a Carleman estimate given by the following lemma:
\begin{lem}\label{Lema 2 de Guerrero}
    For any sufficiently large $s$ and $\lambda$, there exists a positive constant $C$ (depending on $T$, $s$ and $\lambda$) such that, for all $\varphi^{T}\in H$ and $\psi^{T}\in L^{2}(\Omega)$ and any $G_{1}\in L^{2}(Q)^{N}$ and $G_{2}\in L^{2}(Q)$, the solution to (\ref{adjunto de lad. Bous.}) verifies
    \begin{equation}\label{Carleman estimate}
        \begin{array}{c}
         \Vert\varphi(.,0)\Vert^{2} + \Vert\psi(.,0)\Vert^{2} + \displaystyle\iint\limits_{Q}e^{-2s\alpha}[\xi^{3}(\vert\varphi\vert^{2}+\vert\psi\vert^{2})+\xi(\vert\nabla\varphi\vert^{2}+\vert\nabla\psi\vert^{2})]dx\,dt\\
          \leq C\left(\hspace{0.18cm}\displaystyle\iint\limits_{\omega\times(0,T)}e^{-8s\hat{\alpha}+6s\alpha^{\ast}}\hat{\xi}^{16}(\vert\varphi\vert^{2}+\vert\psi\vert^{2})dx\,dt\right.\\
          +\left.\displaystyle\iint\limits_{Q}e^{-4s\hat{\alpha}+2s\alpha^{\ast}}\hat{\xi}^{15/2}(\vert G_{1}\vert^{2}+\vert G_{2}\vert^{2})dx\,dt\right).
        \end{array}
    \end{equation}
\end{lem}
\begin{proof}
    See, Lemma 2 in \cite{Guerrero}.
    \hfill
\end{proof}

\section{Null controllability of linear system \eqref{lad.Bous.Linear}}\label{Sec.2}

This section is dedicated to the null controllability of the linear system \eqref{lad.Bous.Linear}. We emphasize that two null controllability results will be obtained, since we will consider different cases for the initial data $y^{0}, \theta^{0}$ and the functions $F_{1}, F_{2}$. More precisely, in the first case we will consider more common spaces in control theory, such as $H^{1}_{0}(\Omega)$ and $L^{2}(Q)$ while in the second case we will work with spaces less usual ones, like $ W_{0}^{1,p}(\Omega)$ and $L^{q}(0,T;L^{p}(\Omega))$, for $3<p \leq 6$ and $p< q <\infty$.

Let us set
\begin{equation}\label{pesos control nulo}
    \left\{
    \begin{array}{l}
\rho = e^{s\alpha}\xi^{-3/2},\,\,     \rho_{1} = e^{2s\hat{\alpha}-s\alpha{\ast}}\hat{\xi}^{-15/4}\,\,\, \rho_{2} = e^{4s\hat{\alpha}-3s\alpha^{\ast}}\hat{\xi}^{-8},\,\,\,  

    \rho_{3}=e^{s\alpha^{\ast}}(\xi^{\ast})^{-1/2},\,\,\\
    \mu_{1}=e^{8s\hat{\alpha}-7s\alpha^{\ast}}\hat{\xi}^{-15},\,\,\,   
    \mu_{2}=e^{8s\hat{\alpha}-7s\alpha^{\ast}}\hat{\xi}^{-16},\,\,\,  
    \mu_{3}=e^{8s\hat{\alpha}-7s\alpha^{\ast}}\hat{\xi}^{-17}, \vspace{0.1cm}\\
    \kappa=e^{9s\hat{\alpha}-8s\alpha^{\ast}}\hat{\xi}^{-17},
   \end{array}\right.
\end{equation}
so that the values of $s$ and $\lambda$ satisfy the Lemma \ref{Lema 2 de Guerrero}. By inequality \eqref{des.16>15}, we can see that 
\begin{equation}\label{des pesos}
\left\{\begin{array}{l}
     \kappa\leq C\mu_{3} \leq C\mu_{2}\leq C\rho_{2}\leq C\rho_{3}\leq C\mu_{2}^{2},\vspace{0.1cm}\\
     \vert\mu_{2,t}\vert\leq C\rho_{1}, \vert \mu_{3}\mu_{3,t}\vert\leq C\mu_{2}^{2}\,\, \text{and}\,\, \kappa_{t}\leq C\mu_{3} \,\,\, \text{in}\,\,\, (0,T).
     \end{array}\right.
\end{equation}

With Lemma \ref{Lema 2 de Guerrero} we will be able to obtain a null controllability result for (\ref{lad.Bous.Linear}), in which the right-hand side $F_{1}$ and $F_{2}$ decay sufficiently fast to zero as $t\to T^-$. In other words, the following propositions are valid:

\begin{prop}\label{proposição 2.1}Let us assume that 
\begin{itemize}
 \item if $N=2$: $y^{0}\in H$, $\theta^{0}\in L^{2}(\Omega)$, $\rho_{3}F_{1}\in L^{2}(Q)^{2}$ and $\rho_{3}F_{2}\in L^{2}(0,T;L^{3/2}(\Omega))$.
    \item if $N=3$: $y^{0}\in H\cap L^{4}(\Omega)^{3}$, $\theta^{0}\in L^{2}(\Omega)$, $\rho_{3}F_{1}\in L^{2}(Q)^{3}$ and $\rho_{3}F_{2}\in L^{2}(0,T;L^{3/2}(\Omega))$.
\end{itemize}
Then, we can find state-controls $(y,P,\theta,v,v_{0})$ for (\ref{lad.Bous.Linear}) such that
\begin{equation}\label{2.8 de CaraLimacoNany}
\begin{array}{l}
\displaystyle\iint\limits_{Q}\rho_{1}^{2}(\vert y\vert^{2} + \vert \theta\vert^{2})dx\,dt + \displaystyle\iint\limits_{\omega\times (0,T)} \rho_{2}^{2}(\vert v\vert^{2} + \vert v_{0}\vert^{2})dx\,dt\vspace{0.1cm} \\
    \leq C \left(\Vert y^{0}\Vert^{2}_{H} + \Vert\theta^{0}\Vert^{2} + \Vert\rho_{3}F_{1}\Vert^{2}_{L^{2}(Q)^{N}} + \Vert\rho_{3}F_{2}\Vert^{2}_{L^{2}(0,T;L^{3/2}(\Omega))}\right).
    \end{array}
\end{equation}
In particular, one has $y(x,T)=0$ and $\theta(x,T)=0$. Moreover, if $(y^{0}, \theta^{0})\in V\times W_{0}^{1,3/2}(\Omega)$ then $y\in L^{2}(0,T;V)\cap C^{0}([0,T];H)$ and $\theta\in L^{2}(0,T;W^{2,3/2}(\Omega))\cap C^{0}([0,T];L^{3/2}(\Omega))$.
\end{prop}
\begin{proof}
    It is enough to observe that if $N=2$, by Sobolev embedding, we have $\rho_{3}F_{1}\in L^{2}(0,T;H^{-1}(\Omega)^{2})$, $\rho_{3}F_{2}\in L^{2}(0,T;H^{-1}(\Omega)^{2})$ and if $N=3$, we have  $\rho_{3}F_{1}\in L^{2}(0,T;W^{-1,6}(\Omega)^{3})$, $\rho_{3}F_{2}\in L^{2}(0,T;H^{-1}(\Omega)^{3})$. Hence, we can obtain the proof by following the ideas of Proposition 2 in \cite{Guerrero}.

    Indeed, let us introduce some notation:
\begin{itemize}
    \item [(i)] $P_{0}=\lbrace (\hat{\varphi},\hat{\pi},\hat{\psi})\in C^{\infty}(\overline{Q})^{N+2}; \nabla\cdot\hat{\varphi} = 0\, \text{in}\, Q, \hat{\varphi}\left|_{\Sigma}\right.=\hat{\psi}\left|_{\Sigma}\right.=0, \displaystyle\int_{\omega}\hat{\pi}(x,t)dx=0\rbrace;$
    
    \item[(ii)]$a((\varphi,{\pi},{\psi}),(\hat{\varphi},\hat{\pi},\hat{\psi}))=\displaystyle\iint\limits_{Q}\rho_{1}^{-2}(\mathcal{L}_{1}^{\ast}{\varphi}+\nabla{\pi})(\mathcal{L}_{1}^{\ast}\hat{\varphi}+\nabla\hat{\pi})dxdt$\\
    $+\displaystyle\iint\limits_{Q}\rho_{1}^{-2}(\mathcal{L}_{2}^{\ast}{\psi}-{\varphi}_{N})(\mathcal{L}_{2}^{\ast}{\hat{\psi}}-\hat{\varphi}_{N})dxdt + \displaystyle\iint\limits_{Q}\tilde{1}_{\omega}\,\rho_{2}^{-2}({\varphi}\hat{\varphi} + {\psi}\hat{\psi})dxdt, \, \forall\,(\hat{\varphi},\hat{\pi},\hat{\psi})\in P_{0};$
    \item[(iii)]$P$ is the completion of $P_{0}$ for the scalar product $a(.,.)$;
    \item[(iv)] $\langle l, (\hat{\varphi},\hat{\pi},\hat{\psi}) \rangle = \displaystyle\int_{0}^{T}\langle F_{1}(t),\hat{\varphi}(t)\rangle dt + \displaystyle\int_{0}^{T}\langle F_{2}(t),\hat{\psi}(t)\rangle dt + \displaystyle\int_{\Omega}y^{0}\hat{\varphi}(0)dx + \displaystyle\int_{\Omega}\theta^{0}\tilde{\psi}(0)dxdt$.
\end{itemize}
Thus, due to Carleman inequality \eqref{Carleman estimate}, we have
\begin{equation}\label{consequencia de Carleman}
     \begin{array}{c}
         \Vert\hat{\varphi}(.,0)\Vert^{2} + \Vert\hat{\psi}(.,0)\Vert^{2} + \displaystyle\iint\limits_{Q}\rho_{3}^{-2}(\vert\hat{\varphi}\vert^{2}+\vert\hat{\psi}\vert^{2})+\rho_{3}^{-2}\xi^{-2}(\vert\nabla\hat{\varphi}\vert^{2}+\vert\nabla\hat{\psi}\vert^{2})dx\,dt\\
          \leq C\,a((\hat{\varphi},\hat{\pi},\hat{\psi}),(\hat{\varphi},\hat{\pi},\hat{\psi})), \,\forall\, (\hat{\varphi},\hat{\pi},\hat{\psi})\in P_{0}
        \end{array}
\end{equation}
from which it is possible to conclude, using the density of $P_{0}$ in $P$, that
$l$ is a bounded linear form on $P$. Therefore, applying Lax-Milgram's lemma, there exists one and only one $(\varphi,\pi,\psi)\in P$ satisfying 
\begin{equation}\label{Solução por Lax Milgram}
a((\varphi,\pi,\psi),(\hat{\varphi},\hat{\pi},\hat{\psi}))=\langle l, (\hat{\varphi},\hat{\pi},\hat{\psi}) \rangle, \, \forall\, (\hat{\varphi},\hat{\pi},\hat{\psi})\in P.
\end{equation}
According, we can write 
\begin{equation}\label{definição de y, theta, v, v0}
    \left\{\begin{array}{lll}
    y  = \rho^{-2}_{1}(\mathcal{L}_{1}^{\ast}\varphi+\nabla\pi), \,\,\, \nabla\cdot \varphi=0 & in & Q, \vspace{0.1cm}\\
    \theta = \rho^{-2}_{1}(\mathcal{L}_{2}^{\ast}\psi-\varphi e_{N}) & in & Q,\vspace{0.1cm}\\
    v=-\rho^{-2}_{2}\varphi\tilde{1}_{\omega},\,\,\, v_{0}=-\rho^{-2}_{2}\psi\tilde{1}_{\omega} & in & Q,
    \end{array}\right.
\end{equation}
where $(\varphi,\pi,\psi)$ is the unique solution of \eqref{Solução por Lax Milgram}.

Next, just use the Sobolev embedding mentioned above and apply the arguments of \cite{Guerrero} to obtain the existence of controls $(v,v_{0})\in L^{2}(\omega\times (0,T))^{N+1}$ and associated solutions to \eqref{lad.Bous.Linear}
satisfying \eqref{2.8 de CaraLimacoNany} and consequently \eqref{nulo}. The regularity of the $\theta$ solution is justified by the maximum regularity for parabolic systems in spaces $L^{p}-L^{q}$, see Lemma \ref{parabolic in LpLq}, and consequently by the standard results for Stokes systems we obtain the regularity of $y$, see \cite{Temam}.
\hfill
\end{proof}

\begin{prop}\label{proposição controle nulo do sistema linear para W{1,p}} Consider $ 3<p\leq 6$ and $p<q<\infty$. Let us assume that the functions  $F_{1}, F_{2}$ in \eqref{lad.Bous.Linear} satisfy $\rho_{3}F_{1}\in L^{q}(0,T;L^{p}(\Omega)^{N})$, $\rho_{3}F_{2}\in L^{q}(0,T;L^{p}(\Omega))$ and $(y^{0},\theta^{0})\in V^{p}\times W^{1,p}_{0}(\Omega)$. Then \eqref{lad.Bous.Linear} is null-controllable, and its control-state satisfy $(v,v_{0})\in L^{2}(\omega\times(0,T))^{N+1}$, $y\in L^{q}(0,T;W^{2,p}(\Omega)^{N})\cap C^{0}([0,T];L^{p}(\Omega)^{N})$ and $\theta\in L^{q}(0,T;W^{2,p}(\Omega))\cap C^{0}([0,T];L^{p}(\Omega))$.
\end{prop}
\begin{proof}
 Indeed, since $V^{p}\times W^{1,p}_{0}(\Omega)\hookrightarrow H\cap L^{4}(\Omega)^{N}\times L^{2}(\Omega)$ and $L^{q}(0,T; L^{p}(\Omega))\hookrightarrow L^{2}(Q)\hookrightarrow L^{2}(0,T;L^{3/2}(\Omega))$ then the Proposition \ref{proposição 2.1} is verified and consequently \eqref{lad.Bous.Linear} is null-controllable. The regularities of $\theta$ and $y$ follow respectively from the Lemma \ref{parabolic in LpLq} and the Lemma \ref{Stokes para Lp}.
 \hfill
\end{proof}

\subsection{Estimates for the states solutions}\label{estimates for the states solutions}
In this subsection we will show estimates for the solutions associated with (\ref{lad.Bous.Linear}), that is, for both the velocity variable and the temperature variable. We will obtain estimates not only for $y$ and $\theta$, but also for $ \nabla y, \Delta y, \nabla\theta, \Delta\theta$ and the controls $v$ and $v_{0}$. The results obtained in this subsection will be fundamental to obtain the null controllability of the nonlinear systems \eqref{lad.boussinesq} and \eqref{lad.boussinesq caso extra}.

\begin{prop}\label{proposição 2.2}
   Let the assumptions in Proposition \ref{proposição 2.1} be satisfied. Let the state-control $(y,P,\theta,v,v_{0})$ satisfy (\ref{lad.Bous.Linear}) and (\ref{2.8 de CaraLimacoNany}). Then, the following estimate holds:
   \begin{equation}\label{2.9 CaraLimacoNayny}
       \begin{array}{l}
        \displaystyle\sup_{t\in [0,T]}\displaystyle\int_{\Omega}\mu_{1}^{2}\vert y\vert^{2}\,dx + \displaystyle\iint\limits_{Q}\mu_{1}^{2}\vert\nabla y\vert^{2}\,dx\,dt \vspace{0.1cm} \\
        \leq\, C\,\left( \Vert y^{0}\Vert^{2}_{H} + \displaystyle\iint\limits_{Q}[\rho_{3}^{2}\vert F_{1}\vert^{2}+\rho_{1}^{2}(\vert y\vert^{2}+\vert\theta\vert^{2})]\,dx\,dt + \displaystyle\iint\limits_{\omega\times (0,T)}\rho_{2}^{2}\vert v\vert^{2}\,dx\,dt \right).
       \end{array}
   \end{equation}
Furthermore, if $(y^{0},\theta^{0})\in V\times W_{0}^{1,3/2}(\Omega)$, one also has
 \begin{equation}
       \begin{array}{l}\label{2.10 CaraLimacoNany}
        \displaystyle\sup_{t\in [0,T]}\displaystyle\int_{\Omega}\mu_{2}^{2}\vert \nabla y\vert^{2}\,dx + \displaystyle\iint\limits_{Q}\mu_{2}^{2}(\vert y_{t}\vert^{2}+\vert\Delta y\vert^{2})\,dx\,dt \vspace{0.1cm}\\
     \leq\, C\,\left( \Vert y^{0}\Vert^{2}_{V} + \displaystyle\iint\limits_{Q}[\rho_{3}^{2}\vert F_{1}\vert^{2}+\rho_{1}^{2}(\vert y\vert^{2}+\vert\theta\vert^{2})]\,dx\,dt + \displaystyle\iint\limits_{\omega\times (0,T)}\rho_{2}^{2}\vert v\vert^{2}\,dx\,dt \right)    
       \end{array}  
   \end{equation}
and
\begin{equation}\label{2.11 CaraLimacoNany}
    \begin{array}{l}
      \displaystyle\int_{0}^{T}\mu_{2}^{2}\Vert \theta_{t}\Vert^{2}_{L^{3/2}(\Omega)}\,dt + \displaystyle\int_{0}^{T}\mu_{2}^{2}\Vert\Delta\theta\Vert^{2}_{L^{3/2}(\Omega)}\,dt\vspace{0.1cm}\\
         \leq \,C\, \left( \Vert\theta^{0}\Vert^{2}_{W_{0}^{1,3/2}(\Omega)}+\displaystyle\iint\limits_{Q}\rho_{1}^{2}\vert\theta\vert^{2}\,dx\,dt
         +\displaystyle\iint\limits_{\omega\times (0,T)}\rho_{2}^{2}\vert v_{0}\vert^{2}\,dx\,dt + \displaystyle\int_{0}^{T}\Vert\rho_{3}F_{2}\Vert^{2}_{L^{3/2}(\Omega)}\,dt\right).
    \end{array}
\end{equation}
\end{prop}

\begin{proof}
The proofs of (\ref{2.9 CaraLimacoNayny}) and (\ref{2.10 CaraLimacoNany}) can be obtained as in \cite{CaraLimacoMenezes} (just use the same arguments with the weights defined in (\ref{pesos control nulo})).

Let us prove (\ref{2.11 CaraLimacoNany}). Denote by $\tilde{\theta}=\mu_{2}\theta$. Then, by (\ref{lad.Bous.Linear}), we have 
\begin{equation}\label{mudança}
\left\{
    \begin{array}[l]{lll}
     \tilde{\theta}_{t}-\nu_{0}\Delta\tilde{\theta} = \tilde{h} &\text{in} & Q,  \\
      \tilde{\theta}(x,t)=0 & \text{on} & \Sigma  \\
      \tilde{\theta}(x,0)=\mu_{2}(0)\theta^{0}(x) & \text{in} & \Omega,
    \end{array}\right.
\end{equation}
where $\tilde{h} = \mu_{2}v_{0}\tilde{1}_{\omega}+\mu_{2}F_{2}+\mu_{2,t}\theta$.

Note that, by (\ref{des pesos}) and (\ref{2.8 de CaraLimacoNany}) we have $\mu_{2}v_{0}\tilde{1}_{\omega}, \mu_{2}F_{2}\in L^{2}(0,T;L^{3/2}(\Omega))$ and also
\begin{equation}\label{mu2t}
    \begin{array}{ccc}     \Vert\mu_{2,t}\theta\Vert^{2}_{L^{2}(0,T;L^{3/2}(\Omega))} &\leq  & C\displaystyle\int_{0}^{T}\left(\displaystyle\int_{\Omega}\vert\rho_{1}\theta\vert^{3/2}\,dx\right)^{4/3}dt  \vspace{0.2cm}\\
     &=& C\,\Vert\rho_{1}\theta\Vert^{2}_{L^{2}(0,T;L^{3/2}(\Omega))} < + \infty.
    \end{array}
\end{equation}
Then, $\tilde{h}\in L^{2}(0,T;L^{3/2}(\Omega))$. Therefore, from \eqref{estimate de Robert} we have 
\begin{equation*}
    \begin{array}{l}
      \displaystyle\int_{0}^{T}\mu_{2}^{2}\Vert \theta_{t}\Vert^{2}_{L^{3/2}(\Omega)}\,dt + \displaystyle\int_{0}^{T}\mu_{2}^{2}\Vert\Delta\theta\Vert^{2}_{L^{3/2}(\Omega)}\,dt  \vspace{0.1cm}\\

    \leq C \left(\Vert\tilde{h}(t)\Vert^{2}_{L^{2}(0,T;L^{3/2}(\Omega))} + \Vert\mu_{2}(0)\theta^{0}\Vert^{2}_{W^{1,3/2}(\Omega)}\right)
    \end{array}
\end{equation*}
and consequently
\begin{equation*}
    \begin{array}{l}
      \displaystyle\int_{0}^{T}\mu_{2}^{2}\Vert \theta_{t}\Vert^{2}_{L^{3/2}(\Omega)}\,dt + \displaystyle\int_{0}^{T}\mu_{2}^{2}\Vert\Delta\theta\Vert^{2}_{L^{3/2}(\Omega)}\,dt
      
         \vspace{0.1cm}\\
         \leq C\left( \Vert\theta^{0}\Vert^{2}_{W_{0}^{1,3/2}(\Omega)}         +\displaystyle\iint\limits_{Q}\rho_{1}^{2}\vert\theta\vert^{2}\,dx\,dt +\displaystyle\iint\limits_{\omega\times (0,T)}\rho_{2}^{2}\vert v_{0}\vert^{2}\,dx\,dt + \displaystyle\int_{0}^{T}\Vert\rho_{3}F_{2}\Vert^{2}_{L^{3/2}(\Omega)}\,dt\right),
    \end{array}
\end{equation*}
achieving the desired inequality.
\hfill
\end{proof}

\begin{prop}\label{proposição extra}  Let the assumptions in Proposition \ref{proposição controle nulo do sistema linear para W{1,p}} be satisfied. Then, the controls verifies
\begin{eqnarray}
\kappa v\in L^{2}(0,T;[H^{2}(\omega)\cap H^{1}_{0}(\omega)]^{N})\cap C^{0}([0,T];V), \,\, (\kappa v)_{t}\in L^{2}(\omega\times (0,T))^{N}.\label{regularidade do controle v} \vspace{0.1cm}   \\
\kappa v_{0}\in L^{2}(0,T;H^{2}(\omega))\cap C^{0}([0,T];H^{1}(\omega)), \,\, (\kappa v_{0})_{t}\in L^{2}(\omega\times (0,T)).\label{regularidade do controle v0}    
\end{eqnarray}
with the estimate 
\begin{equation*}
    \begin{array}{l}
\displaystyle\int_{0}^{T}\displaystyle\int_{\omega}\left[\vert (\kappa v)_{t}\vert^{2} + \vert (\kappa v_{0})_{t}\vert^{2} + \vert\kappa\Delta v\vert^{2} + \vert\kappa \Delta v_{0}\vert^{2}\right]dxdt + \sup\limits_{[0,T]}\Vert\kappa v\Vert^{2}_{V}\vspace{0.1cm}\\
 + \sup\limits_{[0,T]}\Vert\kappa v_{0}\Vert^{2}_{H^{1}(\omega)} \leq  C\left(\Vert y^{0}\Vert^{2}_{V^{p}}+\Vert\theta^{0}\Vert^{2}_{W^{1,p}_{0}(\Omega)}
 + \Vert\rho_{3}F_{1}\Vert^{2}_{L^{q}(0,T;L^{p}(\Omega)^{N})}\right.\\ 
\hspace{3cm}+ \left. \Vert\rho_{3} F_{2}\Vert^{2}_{L^{q}(0,T;L^{p}(\Omega))}\right).
    \end{array}
\end{equation*}
Furthermore, the associated states satisfy 
\begin{equation}\label{desig para y extra}
    \begin{array}{l}
\displaystyle\iint\limits_{Q}\mu_{3}^{2}\vert y_{t}\vert^{2}dx\,dt+\displaystyle\sup_{[0,T]}\displaystyle\int_{\Omega}\mu_{3}^{2}\vert\nabla y\vert^{2}dx +\displaystyle\iint\limits_{Q}\mu_{3}^{2}\vert \Delta y\vert^{2}dx\,dt + 
\displaystyle\sup_{[0,T]}\displaystyle\int_{\Omega}\mu_{2}^{2}\vert y\vert^{2}dx\vspace{0.1cm}\\

+ \displaystyle\iint\limits_{Q}\mu_{2}^{2}\vert\nabla y\vert^{2}dx\,dt 
\leq \,C\, \left( \Vert y^{0}\Vert^{2}_{V^{p}} + \displaystyle\iint\limits_{Q}\rho_{1}^{2}(\vert\theta\vert^{2} + \vert y\vert^{2})\,dx\,dt
 +\displaystyle\iint\limits_{\omega\times (0,T)}\rho_{2}^{2}\vert v\vert^{2}\,dx\,dt\right.\vspace{0.1cm}\\  \left.\hspace{3.5cm} + \Vert\rho_{3}F_{1}\Vert^{2}_{L^{q}(0,T;L^{p}(\Omega)^{N})}\right)
  \end{array}
\end{equation}
and
\begin{equation}\label{desig proposição extra}
     \begin{array}{l}
\displaystyle\iint\limits_{Q}\mu_{3}^{2}\vert \theta_{t}\vert^{2}dx\,dt + \displaystyle\sup_{[0,T]}\displaystyle\int_{\Omega}\mu_{3}^{2}\vert\nabla\theta\vert^{2}dx+\displaystyle\iint\limits_{Q}\mu_{3}^{2}\vert\Delta \theta\vert^{2}dx\,dt + \displaystyle\sup_{[0,T]}\displaystyle\int_{\Omega}\mu_{2}^{2}\vert\theta\vert^{2}dx\vspace{0.1cm}\\
+ \displaystyle\iint\limits_{Q}\mu_{2}^{2}\vert\nabla \theta\vert^{2}dx\,dt
 \leq \,C\, \left( \Vert\theta^{0}\Vert^{2}_{W^{1,p}_{0}(\Omega)}+\displaystyle\iint\limits_{Q}\rho_{1}^{2}\vert\theta\vert^{2}\,dx\,dt+\displaystyle\iint\limits_{\omega\times (0,T)}\rho_{2}^{2}\vert v_{0}\vert^{2}\,dx\,dt \right.\vspace{0.1cm}\\
\left.\hspace{3.5cm} + \Vert\rho_{3}F_{2}\Vert^{2}_{L^{q}(0,T;L^{p}(\Omega))}\right).
  \end{array}
\end{equation}
\end{prop}
\begin{proof}
The first part of the proof will be dedicated to concluding \eqref{regularidade do controle v} and \eqref{regularidade do controle v0}.

Let us set $u=\rho_{2}^{-2}\varphi$. Hence by the definition given in \eqref{definição de y, theta, v, v0} we have, after some computations, 
\begin{equation}\label{L* kappa u}
    \begin{array}{lll}
         \mathcal{L}_{1}^{\ast}(\kappa u)&=&(\kappa\rho^{-2}_{2})\mathcal{L}_{1}^{\ast}\varphi - (\kappa\rho^{-2}_{2})_{t}\varphi  \\
         &=& (\kappa\rho^{-2}_{2}\rho^{2}_{1})y -(\kappa\rho^{-2}_{2})\nabla\pi -(\kappa\rho^{-2}_{2})_{t}\varphi.
    \end{array}
\end{equation}
We notice that,
\begin{equation}\label{desigualdade kappau}
\vert \kappa\rho_{2}^{-2}\rho^{2}_{1}\vert \leq C\rho_{1},\,
\vert \kappa\rho_{2}^{-2}\vert\leq C,\,
\vert (\kappa\rho_{2}^{-2})_{t}\rho\rho^{-1}\vert\leq C\rho^{-1}.
\end{equation}
And, from the Carleman estimate \eqref{Carleman estimate} and again \eqref{definição de y, theta, v, v0}, we get:
\begin{equation}\label{varphi e psi com rho}
    \begin{array}{l}   
    \displaystyle\iint\limits_{Q}\rho^{-2}\vert\varphi\vert^{2}dxdt \leq \displaystyle\iint\limits_{Q}\rho_{1}^{2}(\vert y\vert^{2}+\vert\theta\vert^{2})dxdt +\displaystyle\iint\limits_{\omega\times (0,T)}\rho_{2}^{2}(\vert v\vert^{2}+\vert v_{0}\vert^{2})dxdt
    <+\infty.
   \end{array}
\end{equation}
Therefore, by \eqref{L* kappa u}, \eqref{desigualdade kappau} and \eqref{varphi e psi com rho} we obtain $(\tilde{u},\tilde{\pi}):=(\kappa u,\kappa\rho_{2}^{-2}\pi)$ solution of the Stokes system
\begin{equation*}
    \left\{\begin{array}{lll}
        \mathcal{L}_{1}^{\ast}\tilde{u} +\nabla\tilde{\pi} = \tilde{f},\,\,\, \nabla\cdot\tilde{u}=0  & \text{in} & Q,  \\
        \tilde{u}=0 & \text{on} & \Sigma,\\
       \tilde{u}(.,T)=0 &\text{in}& \Omega,
    \end{array}\right.
\end{equation*}
with $\tilde{f}=(\kappa\rho^{-2}_{2}\rho_{1})\rho_{1}y - (\kappa\rho^{-2}_{2})_{t}\rho\rho^{-1}\varphi\in L^{2}(Q)^{N}$.

By the standard regularity for solutions of Stokes systems, we can infer the regularity \eqref{regularidade do controle v} for $\kappa v = -\tilde{u}\tilde{1}_{\omega}$.

Similarly, define $w=-\rho_{2}^{-2}\psi$ and note that $v_{0}= w\tilde{1}_{\omega}$. Then,
\begin{equation}\label{L* rhow}
    \begin{array}{lll}
         \mathcal{L}_{2}^{\ast}(\kappa w)&=&-(\kappa\rho^{-2}_{2})\mathcal{L}_{2}^{\ast}\psi + (\kappa\rho^{-2}_{2})_{t}\psi  \\
         &=& -(\kappa\rho^{-2}_{2}\rho^{2}_{1})\theta -(\kappa\rho^{-2}_{2})\varphi e_{N} +(\kappa\rho^{-2}_{2})_{t}\psi\\
         &=& N_{1} + N_{2} + N_{3}.
    \end{array}
\end{equation}
Analyzing each $N_{i}$, $i=\{ 1, 2, 3\}$, we obtain 
\begin{equation*}
\begin{array}{l}
\vert N_{1}\vert \leq e^{3s\hat{\alpha}-3s\alpha^{\ast}}\hat{\xi}^{-19/4}\rho_{1}\theta,\,
\vert N_{2}\vert\leq C e^{s\hat{\alpha}-s\alpha^{\ast}}\hat{\xi}^{-1}\rho^{-1}\varphi e_{N},\,\
\vert N_{3}\vert\leq C\rho^{-1}\psi.
    \end{array}
\end{equation*}

Thus, from \eqref{varphi e psi com rho}, we deduce that $N_{1}+N_{2}+N_{3}\in L^{2}(Q)$. Therefore, taking into consideration the PDE satisfied by $\kappa w$ and the fact that
$(\kappa w)(.,T)=0$, we concluded that
\begin{equation*}
    \kappa w\in L^{2}(0,T;H^{2}(\Omega))\cap C^{0}([0,T];H^{1}(\Omega)), \,\, (\kappa w)_{t}\in L^{2}(Q).
\end{equation*}
In particular, \eqref{regularidade do controle v0} holds.

Now, let's establish the second part of the proposition, that is, \eqref{desig para y extra} and \eqref{desig proposição extra}:

 Firstly, multiplying the linear system $\eqref{lad.Bous.Linear}_{1}$ by $\mu_{2}^{2} y$ (as a test function), integrating in $\Omega$, we have that
\begin{equation*}
\begin{array}{lll}
       \dfrac{1}{2}\dfrac{d}{dt}\displaystyle\int_{\Omega}\mu_{2}^{2}\vert y\vert^{2}dx + \nu_{0}\displaystyle\int_{\Omega}\mu_{2}^{2}\vert\nabla y\vert^{2}dx &\leq & C\left( \displaystyle\int_{\Omega}\rho_{1}^{2}(\vert\theta\vert^{2}+\vert y\vert^{2})dx + \displaystyle\int_{\omega}\rho_{2}^{2}\vert v\vert^{2}dx\right.\vspace{0.1cm}\\
      &\quad{}& \left.+\displaystyle\int_{\Omega}\rho_{3}^{2}\vert F_{1}\vert^{2}dx \right)
    \end{array}
\end{equation*}
thanks to $\vert\mu_{2}\mu_{2,t}\vert\leq C\rho^{2}_{1}$, $\mu_{2}^{2}\leq C\rho_{1}^{2}$, $\mu_{2}^{2}\leq C\rho_{2}^{2}$ and $\mu_{2}^{2}\leq C\rho_{3}^{2}$. Then, integrating from $0$ to $t$ we find that
\begin{equation}\label{desig. para m2 y}
\begin{array}{lll}
\displaystyle\sup_{[0,T]}\displaystyle\int_{\Omega}\mu_{2}^{2}\vert y\vert^{2}dx + \displaystyle\iint\limits_{Q}\mu_{2}^{2}\vert\nabla y\vert^{2}dx\,dt 
&\leq& \,C\, \left( \Vert y^{0}\Vert_{H}^{2}+\displaystyle\iint\limits_{Q}\rho_{1}^{2}(\vert\theta\vert^{2}+\vert y\vert^{2})\,dx\,dt\right.\vspace{0.1cm}\\
&\quad{}&\left.+\displaystyle\iint\limits_{\omega\times (0,T)}\rho_{2}^{2}\vert v\vert^{2}\,dx\,dt + \displaystyle\iint\limits_{Q}\rho_{3}^{2}\vert F_{1}\vert^{2}dx\,dt\right).
\end{array}
\end{equation}

Next, using $\mu_{3}^{2} y_{t}$ as a test function in $\eqref{lad.Bous.Linear}_{1}$, integrating in $\Omega$ and taking into account that $ \mu_{3}^{2}\leq C\rho_{2}^{2}\leq C\rho_{3}^{2}$ and $\vert\mu_{3}\mu_{3,t} \vert\leq C\mu_{2}^{2}$, we obtain
    \begin{equation*}
        \begin{array}{lll}
\displaystyle\int_{\Omega}\mu_{3}^{2}\vert y_{t}\vert^{2}dx + \nu_{0}\dfrac{d}{dt}\displaystyle\int_{\Omega}\mu_{3}^{2}\vert\nabla y\vert^{2}dx
        & \leq & C\left(\displaystyle\int_{\omega}\rho_{2}^{2}\vert v\vert^{2} + \displaystyle\int_{\Omega}\rho_{3}^{2}\vert F_{1}\vert^{2}dx + \displaystyle\int_{\Omega}\rho_{1}^{2}\vert\theta\vert^{2}dx\right.\vspace{0.1cm}\\
        & \quad{}& \left. +  \displaystyle\int_{\Omega}\mu_{2}^{2}\vert\nabla y\vert^{2}dx \right).
        \end{array}
    \end{equation*}
Hence, integrating from $0$ to $t$ and making use of \eqref{desig. para m2 y}, we deduce that 
\begin{equation}\label{desig, mu3 y}
    \begin{array}{l}
\displaystyle\iint\limits_{Q}\mu_{3}^{2}\vert y_{t}\vert^{2}dx\,dt + \displaystyle\sup_{[0,T]}\displaystyle\int_{\Omega}\mu_{3}^{2}\vert\nabla y\vert^{2}dx
 \leq \,C\, \left( \Vert y^{0}\Vert^{2}_{V}+ \displaystyle\iint\limits_{\omega\times (0,T)}\rho_{2}^{2}\vert v\vert^{2}\,dx\,dt \right.\vspace{0.1cm}\\
\left.\hspace{4cm} + \displaystyle\iint\limits_{Q}\rho_{1}^{2}(\vert\theta\vert^{2} + \vert y\vert^{2})\,dx\,dt + \displaystyle\iint\limits_{Q}\rho_{3}^{2}\vert F_{1}\vert^{2}dx\,dt\right).
  \end{array}
\end{equation}

Finally, multiplying $\eqref{lad.Bous.Linear}_{1}$ by $-\mu_{3}^{2}\Delta y $ and followed in a similar way to the previous estimates, we arrive at
\begin{equation}\label{desig mu3 delta y}
    \begin{array}{l}
\displaystyle\sup_{[0,T]}\displaystyle\int_{\Omega}\mu_{3}^{2}\vert\nabla y\vert^{2}dx +\displaystyle\iint\limits_{Q}\mu_{3}^{2}\vert \Delta y\vert^{2}dx\,dt
 \leq \,C\, \left( \Vert y^{0}\Vert^{2}_{V}+\displaystyle\iint\limits_{Q}\rho_{1}^{2}(\vert\theta\vert^{2} + \vert y\vert^{2})\,dx\,dt\right.\vspace{0.1cm}\\
\left.\hspace{4.6cm} +\displaystyle\iint\limits_{\omega\times (0,T)}\rho_{2}^{2}\vert v\vert^{2}\,dx\,dt  + \displaystyle\iint\limits_{Q}\rho_{3}^{2}\vert F_{1}\vert^{2}dx\,dt\right).
  \end{array}
\end{equation}
Therefore, from \eqref{desig. para m2 y}-\eqref{desig mu3 delta y} and by Sobolev's immersions $V^{p}\hookrightarrow V$ and $L^{q}(0,T;L^{p}(\Omega))\hookrightarrow L^{2}(0,T;L^{2}(\Omega))$ we see that  \eqref{desig para y extra} holds.

The estimate \eqref{desig proposição extra} is obtained by multiplying $\mu_{2}^{2}\theta$, $\mu_{3}^{2}\theta_{t}$ and $-\mu_{3}^{2} \Delta\theta$ one after the other in $\eqref{lad.Bous.Linear}_{2}$ and using the same arguments as before.
\hfill
\end{proof}

The following result is a proposition by \cite{JuanJessica} and will play a crucial role in the conclusion of our main theorems, as it enables us to derive enhanced regularity for the solution \((y, \theta)\) in weighted function spaces.

\begin{prop}\label{Proposição 7.1 de Juan Jessica} If $u\in {L}^{q}(0,T;W^{2,p}(\Omega))$, $u_{t}\in L^{q}(0,T;L^{p}(\Omega))$ then $u\in C^{0}([0,T];W^{1,p}(\Omega))$, $p>2$ and $p<q<\infty$.
\end{prop}
\begin{proof} See Proposition 7.1 in \cite{JuanJessica}.
\end{proof}

The subsequent proposition will be fundamental in ensuring the null controllability of the system \eqref{lad.boussinesq caso extra}, and its proof will be derived from the preceding results in this section.
\begin{prop}\label{Proposicao com rhobarratheta}
    Let the assumptions in Proposition \ref{proposição extra} be satisfied. Then, the following estimates are valid
    \begin{equation}\label{estimate para theta W{1,p}}
    \begin{array}{l}
      \Vert (\kappa\theta)_{t}\Vert_{L^{q}(0,T;L^{p}(\Omega))} + \Vert\kappa\theta\Vert_{L^{q}(0,T;W^{2,p}(\Omega))} + \Vert\kappa\theta\Vert_{C^{0}(0,T;W^{1,p}(\Omega))}\vspace{0.1cm}\\
         \leq C\left(\Vert y^{0}\Vert^{}_{V^{p}}+\Vert\theta^{0}\Vert^{}_{W^{1,p}_{0}(\Omega)} +\Vert\rho_{3}F_{1}\Vert^{}_{L^{q}(0,T;L^{p}(\Omega)^{N})} + \Vert \rho_{3} F_{2}\Vert^{}_{L^{q}(0,T;L^{p}(\Omega))} \right)
    \end{array}
\end{equation}
and 
\begin{equation}\label{estimate para y em Vp}
    \begin{array}{l}
      \Vert (\kappa y)_{t}\Vert_{L^{q}(0,T;L^{p}(\Omega)^{N})} + \Vert\kappa y\Vert_{L^{q}(0,T;W^{2,p}(\Omega)^{N})} + \Vert\kappa y\Vert_{C^{0}(0,T;W^{1,p}(\Omega)^{N})}\vspace{0.1cm}\\
       \leq C\left(\Vert y^{0}\Vert^{}_{V^{p}}+\Vert\theta^{0}\Vert^{}_{W^{1,p}_{0}(\Omega)} + \Vert\rho_{3}F_{1}\Vert^{}_{L^{q}(0,T;L^{p}(\Omega)^{N})} + \Vert \rho_{3} F_{2}\Vert^{}_{L^{q}(0,T;L^{p}(\Omega))} \right)
    \end{array}
\end{equation}
    \end{prop}
\begin{proof}
Define $\tau=\kappa\theta$. Then, by (\ref{lad.Bous.Linear}), we have 
\begin{equation*}
\left\{
    \begin{array}[l]{lll}
     \tau_{t}-\nu_{0}\Delta\tau = b &\text{in} & Q,  \\
      \tau(x,t)=0 & \text{on} & \Sigma,  \\
      \tau(x,0)=\kappa(0)\theta^{0}(x) & \text{in} & \Omega,
    \end{array}\right.
\end{equation*}
where $b = \kappa v_{0}\tilde{1}_{\omega}+\kappa F_{2}+{\kappa_{t}}\theta$. Thus, as consequence of \eqref{des pesos}, \eqref{regularidade do controle v0} and \eqref{desig proposição extra} together with the fact that  $C^{0}(0,T;H^{1}(\Omega))\hookrightarrow L^{q}(0,T;L^{p}(\Omega))$ we have $b\in L^{q}(0,T;L^{p}(\Omega))$. Then, applying the Lemma \ref{parabolic in LpLq} and Proposition \ref{Proposição 7.1 de Juan Jessica}, we get \eqref{estimate para theta W{1,p}}.

Now, notice that, $(\gamma, \bar{P})=(\kappa y,\kappa P)$ solve the Stokes equation 
\begin{equation*}
    \left\{\begin{array}{lll}
       \gamma_{t}-\nu_{0}\Delta\gamma + \nabla\bar{P} = \bar{b},\,\,\, \nabla\cdot\gamma =0 &\text{in}& Q,\\
       \gamma(x,t)=0 &\text{on}& \Sigma,\\
       \gamma(x,0)=\kappa(0)y^{0}(x) &\text{in}& \Omega,
    \end{array}\right.
\end{equation*}
where $\bar{b} = \kappa v\tilde{1}_{\omega} + \kappa_{t}y + \nu_{0}\kappa\theta e_{N} + \kappa F_{1}$. Hence, applying \eqref{des pesos}, \eqref{regularidade do controle v}, \eqref{desig para y extra} and \eqref{estimate para theta W{1,p}} we have $\bar{b}\in L^{q}(0,T;L^{p}(\Omega)^{N})$. Then, from the Lemma \ref{Stokes para Lp} and again by Proposition \ref{Proposição 7.1 de Juan Jessica}, \eqref{estimate para y em Vp} is acquired.
\hfill
\end{proof}

\section{Proofs of the main theorems}\label{Sec.3}

In this section we will prove Theorems \ref{Teo principal control nulo} and \ref{Teo extra control nulo}.

\subsection{Development for the proof of the Theorem \ref{Teo principal control nulo}}

Here, we will proved the local null controllability for the system \eqref{lad.boussinesq}. Let us consider the Stokes operator $A:D(A)\rightarrow H$, where $D(A):=V\cap H^{2}(\Omega)^{N}$, $Aw=P(-\Delta w)$ for all $w\in D(A)$ and $P:L^{2}(\Omega)^{N}\rightarrow H$ is the orthogonal projector.

Let $ \mathcal{E}_{N}$ be (for $N=2$ or $N=3$) the following space:
\begin{equation}\label{espaço EN}
    \begin{array}{c}
   \mathcal{E}_{N}=\lbrace (y,P,\theta,v,v_{0}): \rho_{1}y, \rho_{2}v\tilde{1}_{\omega}\in L^{2}(Q)^{N}, 
   y\in L^{2}(0,T;D(A)), \nabla y\in L^{2}(Q)^{N\times N},\vspace{0.1cm}\\
   P\in L^{2}(0,T;H^{1}(\Omega)),
   \rho_{1}\theta, \rho_{2}v_{0}\tilde{1}_{\omega}\in L^{2}(Q), \theta\in L^{2}(0,T;W^{2,3/2}(\Omega)),\vspace{0.1cm}\\ 
  \text{for}\, F_{1}:=\mathcal{L}_{1}y+\nabla P - \nu_{0}\theta e_{N} - v\tilde{1}_{\omega}\,
  \text{and}\, F_{2}:=\mathcal{L}_{2}\theta -v_{0}\tilde{1}_{\omega},   
  \rho_{3}F_{1}\in L^{2}(Q)^{N},\vspace{0.1cm}\\ 
   \rho_{3}F_{2}\in L^{2}(0,T;L^{3/2}(\Omega)),  
  \nabla\cdot y\equiv 0, y(.,0)\in V, \theta(.,0)\in W_{0}^{1,3/2}(\Omega), \theta\mid_{\Sigma} =0 \rbrace,   
    \end{array}
\end{equation}
emphasizing that $\mathcal{L}_{1}y=y_{t} -\nu_{0}\Delta y$ and $\mathcal{L}_{2}\theta=\theta_{t} -\nu_{0}\Delta\theta$. Thus, it's clear that $\mathcal{E}_{N}$ is a Banach space for the norm $\Vert .\Vert_{ \mathcal{E}_{N}}$, where
\begin{equation*}
    \begin{array}{l}
     \Vert (y,P,\theta,v,v_{0})\Vert^{2}_{ \mathcal{E}_{N}}:= \Vert y\Vert^{2}_{L^{2}(0,T;D(A))} + \Vert\theta\Vert^{2}_{L^{2}(0,T;W^{2,3/2}(\Omega))}\vspace{0.1cm}  
    \, + \,\Vert \rho_{1}y\Vert^{2}_{L^{2}(Q)^{N}} \\ + \Vert\rho_{1}\theta\Vert^{2}_{L^{2}(Q)} + \Vert P\Vert^{2}_{L^{2}(0,T;H^{1}(\Omega))} + \Vert\rho_{2}v\Vert^{2}_{L^{2}(\omega\times (0,T))^{N}} 
    \, +\, \Vert\rho_{2}v_{0}\Vert^{2}_{L^{2}(\omega\times (0,T))}\vspace{0.1cm}\\ + \Vert\rho_{3}F_{1}\Vert^{2}_{L^{2}(Q)^{N}}
     \,+\, \Vert\rho_{3}F_{2}\Vert^{2}_{L^{2}(0,T;L^{3/2}(\Omega))} 
      + \Vert\theta(.,0)\Vert^{2}_{W_{0}^{1,3/2}(\Omega)}.
    \end{array}
\end{equation*}

Due to Proposition \ref{proposição 2.2} we get:
\begin{equation}
    \begin{array}{l}
    \Vert \mu_{1}y\Vert_{L^{\infty}(0,T;H)} + \Vert\mu_{1}y\Vert_{L^{2}(0,T;V)} 
    + \Vert\mu_{2}y\Vert_{L^{\infty}(0,T;V)} + \Vert\mu_{2}y\Vert_{L^{2}(0,T;D(A))}\vspace{0.1cm}\\ 
    \,\,+\,\Vert\mu_{2}y_{t}\Vert_{L^{2}(0,T;L^{2}(\Omega)^{N})} + \Vert\mu_{2}\theta_{t}\Vert_{L^{2}(0,T;L^{3/2}(\Omega))} 
    +\Vert\mu_{2}\theta\Vert_{L^{2}(0,T;W^{2,3/2}(\Omega))}\vspace{0.1cm}\\
    \,\,\leq C\Vert(y,P,\theta,v,v_{0})\Vert_{ \mathcal{E}_{N}}.    \end{array}\label{conseq. da prop}
\end{equation}
Furthermore, if $(y,P,\theta,v,v_{0})\in  \mathcal{E}_{N}$, then $y_{t}\in L^{2}(Q)^{N}$, 
whence $y:[0,T]\rightarrow V$ 
is continuous (see, \cite{Evans}) and we have $y(.,0)\in V$, with
\begin{equation}\label{des aplicados em 0}
    \Vert y(.,0)\Vert_{V}\leq C\Vert(y,P,\theta,v,v_{0})\Vert_{ \mathcal{E}_{N}},\,\, 
\end{equation}

Now, let us introduce the Banach space
\begin{equation}\label{espaço ZN}
    \mathcal{Z}_{N}=L^{2}(\rho_{3}^{2};Q)^{N}\times V\times L^{2}(\rho_{3}^{2}(0,T);L^{3/2}(\Omega))\times W^{1,3/2}_{0}(\Omega),
\end{equation}
where
$L^{2}({\rho}_{3}^{2}(0,T);L^{3/2}(\Omega))$ be the Banach space formed by the measurable functions $u =
u(x; t)$ such that ${\rho}_{3}u\in L^{2}(0,T; L^{3/2}(\Omega))$, i.e.,
\[
\Vert u\Vert^{2}_{L^{2}({\rho}_{3}^{2}(0,T);L^{3/2}(\Omega))}=\displaystyle\int_{0}^{T}{\rho}_{3}^{2}\Vert u(t)\Vert_{L^{3/2}(\Omega)}^{2}dt< +\infty.
\]
Replacing $L^{3/2}(\Omega)$ by $L^{2}(\Omega)$ in $L^{2}({\rho}_{3}^{2}(0,T);L^{3/2}(\Omega))$, we get $L^{2}(\rho_{3}^{2};Q)$.

Finally, consider also the mapping $\mathcal{F}:\mathcal{E}_{N}\rightarrow \mathcal{Z}_{N}$, such that
\begin{equation}\label{Mapa F}
    \mathcal{F}(y,P,\theta,v,v_{0})=(\mathcal{F}_{1},\mathcal{F}_{2}, \mathcal{F}_{3},\mathcal{F}_{4})(y,P,\theta,v,v_{0})
\end{equation}
where
\begin{equation}\label{Mapas Fi}
    \left\{
    \begin{array}{l}
      \mathcal{F}_{1}(y,P,\theta,v,v_{0}):=y_{t}-\nu(\nabla y)\Delta y + (y\cdot\nabla)y +\nabla P - \nu_{0}\theta e_{N} - v\tilde{1}_{\omega},\\
      \mathcal{F}_{2}(y,P,\theta,v,v_{0}):= y(.,0),\\
      \mathcal{F}_{3}(y,P,\theta,v,v_{0}):= \theta_{t} - \nu(\nabla y)\Delta\theta +y\cdot\nabla\theta 
      -\nu(\nabla y)Dy:\nabla y- v_{0}\tilde{1}_{\omega},
      \\
      \mathcal{F}_{4}(y,P,\theta,v,v_{0}):= \theta(.,0).
    \end{array}\right.
\end{equation}
Note that, in $(\ref{Mapas Fi})_{1}$ 
we used the definition of $\nabla\cdot\left(\nu(\nabla y)Dy\right) $ to rewrite in the form $ \nu(\nabla y)\Delta y$, since $\nabla\cdot y=0$. This is,
\begin{equation*}
\begin{array}{ll}
    \nabla\cdot\left(\nu(\nabla y) Dy\right)&=\nabla\cdot\left((\nu_{0}+\nu_{1}\Vert\nabla y\Vert^{2})(\nabla y +\nabla^{T}y)\right)\\
    &=(\nu_{0}+\nu_{1}\Vert\nabla y\Vert^{2})\nabla\cdot(\nabla y) + (\nu_{0}+\nu_{1}\Vert\nabla y\Vert^{2}) \nabla\cdot(\nabla^{T}y)\\
    &=(\nu_{0}+\nu_{1}\Vert\nabla y\Vert^{2})\Delta y + (\nu_{0}+\nu_{1}\Vert\nabla y\Vert^{2})\nabla(\nabla\cdot y)\\
    &=(\nu_{0}+\nu_{1}\Vert\nabla y\Vert^{2})\Delta y =\nu(\nabla y)\Delta y.
    \end{array}
\end{equation*}

We are interested in apply the Mapping Inverse Theorem in infinite dimensional spaces, that can be found in \cite{Alekseev}, and is given below, where $B_{r}(0)$ and $B_{\delta}(\zeta_{0})$ are open ball, respectively of radius $r$ and $\delta$.

\begin{theorem}[Mapping Inverse Theorem]\label{Liusternik}
Let $ \mathcal{E}$ and $ \mathcal{Z}$ be Banach spaces and let $\mathcal{F}:B_{r}(0)\subset  \mathcal{E}\rightarrow  \mathcal{Z}$ be a $\mathcal{C}^{1}$ mapping. Let as assume that $\mathcal{F}^{\prime}(0)$ is onto and let us set $\mathcal{F}(0)=\zeta_{0}$. Then, there exist $\delta >0$, a mapping $W: B_{\delta}(\zeta_{0})\subset  \mathcal{Z}\rightarrow  \mathcal{E}$ and a constant $K>0$ such that
\begin{equation*}
    W(z)\in B_{r}(0),\,\, \mathcal{F}(W(z))=z\,\, \text{and}\,\, \Vert W(z)\Vert_{ \mathcal{E}}\leq K\Vert z-\mathcal{F}(0)\Vert_{ \mathcal{Z}}\, \, \forall\, z\in B_{\delta}(\zeta_{0}).
\end{equation*}
In particular, $W$ is a local inverse-to-the-right of $\mathcal{F}$.
\end{theorem}

Thus, we will prove that we can apply this Theorem \ref{Liusternik} to the mapping $\mathcal{F}$ in (\ref{Mapa F})-(\ref{Mapas Fi}), through the following three lemmas:
\begin{lem}\label{F bem definido}
    Let $\mathcal{F}: \mathcal{E}_{N}\rightarrow  \mathcal{Z}_{N}$ be given by (\ref{Mapa F})-(\ref{Mapas Fi}). Then, $\mathcal{F}$ is well defined and continuous around the origin. 
\end{lem}
\begin{proof}
We will do the proof for the $N=3$ case, the $N=2$ case is similar. 

We want to show that $\mathcal{F}(y,P,\theta,v,v_{0})$ belongs to $ \mathcal{Z}_{3}$, for every $(y,P,\theta,v,v_{0})\in  \mathcal{E}_{3}$. To do this, we will show that each $\mathcal{F}_{i}(y,P,\theta,v,v_{0})$, with $i=\lbrace 1, 2, 3, 4\rbrace$, defined in \eqref{Mapas Fi} belongs to its respective space. Note that, 
\begin{equation*}
    \begin{array}{l}
      \Vert\mathcal{F}_{1}(y,P,\theta,v,v_{0})\Vert^{2}_{L^{2}(\rho^{2}_{3};Q)^{3}} \leq  3\Vert \rho_{3}(\mathcal{L}_{1}y + \nabla P - \nu_{0}\theta e_{N} - v\tilde{1}_{\omega})\Vert^{2}_{L^{2}(Q)^{3}}\vspace{0.1cm} \\
        \hspace{4.3cm}  
      +\, 3\Vert \rho_{3}(y.\nabla )y \Vert^{2}_{L^{2}(Q)^{3}}+\, 3\Vert\rho_{3}\nu_{1}\Vert\nabla y\Vert^{2}\Delta y\Vert^{2}_{L^{2}(Q)^{3}}.  
    \end{array}
\end{equation*}
By the definition of $ \mathcal{E}_{3}$ we have that
\begin{equation}\label{des primeiro termo de F1}
    \Vert \rho_{3}(\mathcal{L}_{1}y + \nabla P - \nu_{0}\theta e_{N} - v\tilde{1}_{\omega})\Vert^{2}_{L^{2}(Q)^{3}}=\Vert\rho_{3}F_{1}\Vert^{2}_{L^{2}(Q)^{3}}\leq C\,\Vert (y,P,\theta,v,v_{0})\Vert^{2}_{ \mathcal{E}_{3}}.
\end{equation}
 Also, from (\ref{des pesos}), in view of (\ref{conseq. da prop}) and by continuous immersion $H^{2}(\Omega)\hookrightarrow L^{\infty}(\Omega)$ we have
\begin{equation}\label{des com y e nabla y}
\begin{array}{l}
\Vert \rho_{3}(y.\nabla )y \Vert^{2}_{L^{2}(Q)^{3}} \leq C\,\displaystyle\iint\limits_{Q}\mu_{2}^{4}\vert y\vert^{2}\vert\nabla y\vert^{2}dx\,dt\vspace{0.1cm}\\
\leq C\left(\displaystyle\sup_{[0,T]}\displaystyle\int_{\Omega} \mu_{2}^{2}\vert\nabla y\vert^{2}\,dx\right)\left(\displaystyle\int_{0}^{T}\mu_{2}^{2}\Vert y\Vert^{2}_{L^{\infty}(\Omega)^{3}}dt\right)\vspace{0.1cm}\\
\leq C\,\Vert \mu_{2} y\Vert^{2}_{L^{\infty}(0,T;V)}\Vert \mu_{2} y\Vert^{2}_{L^{2}(0,T;D(A))} 
\leq  C\,\Vert (y,P,\theta,v,v_{0})\Vert^{4}_{ \mathcal{E}_{3}}.
\end{array}
\end{equation}    
And, since $\mu_{2}^{-1}\leq C$,
\begin{equation}\label{des com nabla e delta}
    \begin{array}{l}
\Vert\rho_{3}\nu_{1}\Vert\nabla y\Vert^{2}\Delta y\Vert^{2}_{L^{2}(Q)^{3}}  \leq C\,\displaystyle\int_{0}^{T} \mu_{2}^{-2}\mu_{2}^{4}\Vert \nabla y\Vert^{4}\displaystyle\int_{\Omega}\vert\mu_{2}\Delta y\vert^{2}dx\,dt  \vspace{0.1cm} \\
       \leq  C\,\Vert\mu_{2}y\Vert^{4}_{L^{\infty}(0,T;V)} \Vert \mu_{2} y\Vert^{2}_{L^{2}(0,T;D(A))}
      \leq C\,  \Vert (y,P,\theta,v,v_{0})\Vert^{6}_{ \mathcal{E}_{3}}. 
    \end{array}
\end{equation}

Therefore, by \eqref{des primeiro termo de F1}, \eqref{des com y e nabla y} and \eqref{des com nabla e delta} we get $\mathcal{F}_{1}(y,P,\theta,v,v_{0})\in L^{2}(\rho_{3}^{2};Q)^{3}$. From the inequality (\ref{des aplicados em 0}) it follows that $\mathcal{F}_{2}(y,P,\theta,v,v_{0})\in V$.

Now, for $\mathcal{F}_{3}(y,P,\theta,v,v_{0})$ note the following:
\begin{equation}\label{F3 limitado}
\begin{array}{l}
    \Vert\mathcal{F}_{3}(y,P,\theta,v,v_{0})\Vert^{2}_{L^{2}(\rho_{3}^{2}(0,T);L^{3/2}(\Omega))} \leq  C\, \left( \Vert\rho_{3}(\mathcal{L}_{2}\theta-v_{0}\tilde{1}_{\omega})\Vert^{2}_{L^{2}(0,T;L^{3/2}(\Omega))}\right.\vspace{0.1cm} \\
    \,\,+\,\Vert\rho_{3}y\cdot\nabla\theta\Vert^{2}_{L^{2}(0,T;L^{3/2}(\Omega))} + \Vert\rho_{3}\nu_{1}\Vert\nabla y\Vert^{2}\Delta\theta\Vert^{2}_{L^{2}(0,T;L^{3/2}(\Omega))}\vspace{0.1cm}\\

    \left.\,\, + \, \Vert\rho_{3}(\nu_{0}+\nu_{1}\Vert\nabla y\Vert^{2})Dy:\nabla y\Vert^{2}_{L^{2}(0,T;L^{3/2}(\Omega))}\right) = C\sum_{s=1}^{4}X_{s}.
    \end{array}
\end{equation}
Let's analyze each $X_{s}, s=\lbrace 1, 2, 3, 4\rbrace$:
\begin{equation*}
\begin{array}{l}
\hspace{-4.2cm}\bullet\,\, X_{1}=\Vert\rho_{3}F_{2}\Vert^{2}_{L^{2}(0,T;L^{3/2}(\Omega))}\leq C\,\Vert(y,P,\theta,v,v_{0})\Vert^{2}_{\mathcal{E}_{3}};
\end{array}
\end{equation*}
Using the continuous embedding $W^{1,3/2}(\Omega)\hookrightarrow L^{3}(\Omega)$,
\begin{equation*}
\begin{array}{l}
\bullet\,\, X_{2}\hspace{0.1cm}\leq \displaystyle\int_{0}^{T}\rho_{3}^{2}\left[\left(\displaystyle\int_{\Omega}\vert y\vert^{3}\,dx\right)^{1/2}\left( \displaystyle\int_{\Omega}\vert \nabla\theta\vert^{3}\,dx\right)^{1/2}\,  \right]^{4/3}dt\vspace{0.1cm}\\

= \displaystyle\int^{T}_{0}\rho_{3}^{2}\Vert y\Vert^{2}_{L^{3}(\Omega)^{3}}\Vert\nabla\theta\Vert^{2}_{L^{3}(\Omega)}\,dt

\leq C\displaystyle\int^{T}_{0}\mu_{2}^{4}\Vert \nabla y\Vert^{2}\Vert\nabla \theta\Vert^{2}_{W^{1,3/2}(\Omega)}\,dt\vspace{0.1cm}\\

\leq C\displaystyle\int^{T}_{0}\mu_{2}^{4}\Vert \nabla y\Vert^{2}\Vert\theta\Vert^{2}_{W^{2,3/2}(\Omega)}\,dt

\leq C\Vert \mu_{2 }y\Vert^{2}_{L^{\infty}(0,T;V)}\Vert\mu_{2}\theta\Vert^{2}_{L^{2}(0,T;W^{2,3/2}(\Omega))}
\vspace{0.1cm}\\
\leq C\, \Vert (y,P,\theta,v,v_{0})\Vert^{4}_{ \mathcal{E}_{3}};
\end{array}
\end{equation*}
\begin{equation*}
\hspace{-1.5cm}   \begin{array}{l}
 \bullet\,\, X_{3} \leq  C\displaystyle\int_{0}^{T}\left( \displaystyle\int_{\Omega}\mu_{2}^{3}\Vert\nabla y\Vert^{3}\vert\Delta\theta\vert^{3/2}dx \right)^{4/3}dt\vspace{0.1cm}\\
 
\leq  C \displaystyle\int_{0}^{T}\mu_{2}^{4}\mu_{2}^{-2}\Vert\nabla y\Vert^{4}\left(\displaystyle\int_{\Omega}\mu_{2}^{3/2}\vert\Delta\theta\vert^{3/2}dx \right)^{4/3}dt  \vspace{0.2cm}\\
\leq  C\left(\displaystyle\sup_{[0,T]}\displaystyle\int_{\Omega}\mu_{2}^{2}\vert\nabla y\vert^{2}\,dx\right)^{2}\displaystyle\int_{0}^{T}\left(\displaystyle\int_{\Omega}\mu_{2}^{3/2}\vert\Delta\theta\vert^{3/2}\,dx\right)^{4/3}dt\vspace{0.1cm}\\

\leq  C\, \Vert\mu_{2}y\Vert^{4}_{L^{\infty}(0,T;V)}\Vert\mu_{2}\theta\Vert^{2}_{L^{2}(0,T;W^{2,3/2}(\Omega))}
\leq C\,\Vert(y,P,\theta,v,v_{0})\Vert^{6}_{ \mathcal{E}_{3}};
    \end{array}
\end{equation*}
\begin{equation*}
\hspace{0.2cm}    \begin{array}{l}
\bullet\,\, X_{4}\leq C\,\displaystyle\int_{0}^{T}\left(\displaystyle\int_{\Omega}\mu_{2}^{3}(\nu_{0}+\nu_{1}\Vert\nabla y\Vert^{2})^{3/2}\vert\nabla y\vert^{3}\,dx\right)^{4/3}dt\vspace{0.1cm} \\
\leq C\displaystyle\int_{0}^{T}(\nu_{0}+\nu_{1}\Vert\nabla y\Vert^{2})^{2}\left(\displaystyle\int_{\Omega}\mu_{2}^{3}\vert\nabla y\vert^{3}\,dx\right)^{4/3}dt\vspace{0.1cm}\\
\leq C\, \displaystyle\int_{0}^{T}\left(\displaystyle\int_{\Omega}\mu_{2}^{3}\vert\nabla y\vert^{3}\,dx\right)^{4/3}dt + C\displaystyle\int_{0}^{T}\mu_{2}^{-4}\mu_{2}^{4}\Vert\nabla y\Vert^{4}\left(\displaystyle\int_{\Omega}\mu_{2}^{3}\vert\nabla y\vert^{3}\,dx \right)^{4/3}dt\vspace{0.1cm}\\
    = C (\tilde{K}_{1}+\tilde{K}_{2}). 
    \end{array}
\end{equation*}
First let's analyze $\tilde{K}_{1}$. Since $\nabla(\mu_{2}y)$ belongs to $L^{\infty}(0,T;L^{2}(\Omega)^{3\times 3})\cap L^{2}(0,T;H^{1}(\Omega)^{3\times 3})$ then using the Lemma 6.7 from \cite{Lions69}, which ensures continuous embedding 
\begin{equation}\label{imersao cap}
L^{\infty}(0,T;L^{2}(\Omega)^{3})\cap L^{2}(0,T;H^{1}(\Omega)^{3})\hookrightarrow L^{4}(0,T;L^{3}(\Omega)^{3}),
\end{equation}
we have $\nabla(\mu_{2}y)\in L^{4}(0,T;L^{3}(\Omega)^{3\times 3})$. Even more, $$\Vert \nabla (\mu_{2}y)\Vert_{L^{4}(0,T;L^{3}(\Omega)^{3})}\leq C\Vert\nabla (\mu_{2}y)\Vert_{L^{\infty}(0,T;L^{2}(\Omega)^{3})}^{1/2}\Vert\nabla (\mu_{2} y)\Vert_{L^{2}(0,T;H^{1}(\Omega)^{3})}^{1/2}.$$ 
Thereat,
\begin{equation*}
    \begin{array}{ll}
    \tilde{K}_{1} \leq   C\,\Vert\nabla (\mu_{2}y)\Vert_{L^{\infty}(0,T;L^{2}(\Omega)^{3})}^{2}\Vert\nabla (\mu_{2} y)\Vert_{L^{2}(0,T;H^{1}(\Omega)^{3})}^{2}.\vspace{0.1cm}\\

      \leq  C\,\Vert\mu_{2}y\Vert_{L^{\infty}(0,T;V)}^{2}\Vert\mu_{2} y\Vert_{L^{2}(0,T;D(A))}^{2} 
      \leq C\Vert (y,P,\theta,v,v_{0})\Vert^{4}_{ \mathcal{E}_{3}}.
    \end{array}
\end{equation*}
And again using the fact that $\mu_{2}^{-1}\leq C$, we get
\begin{equation*}
    \begin{array}{lll}
        \tilde{K}_{2}&\leq & C\left(\displaystyle\sup_{[0,T]}\displaystyle\int_{\Omega}\mu_{2}^{2}\vert\nabla y\vert^{2}\,dx\right)^{2}\displaystyle\int_{0}^{T}\left(\displaystyle\int_{\Omega}\mu_{2}^{3}\vert\nabla y\vert^{3}\,dx\right)^{4/3}dt 
        \leq  C\,\Vert (y,P,\theta,v,v_{0})\Vert^{8}_{ \mathcal{E}_{3}}.
    \end{array}
\end{equation*}
Thus, 
\begin{equation}\label{X4}
\begin{array}{l}
\bullet \, X_{4} \leq C(1+\Vert (y,P,\theta,v,v_{0})\Vert^{4}_{ \mathcal{E}_{3}})\Vert (y,P,\theta,v,v_{0})\Vert^{4}_{ \mathcal{E}_{3}}
\end{array}
\end{equation}
and consequently, from what we saw for each $X_{s}, s=\lbrace 1, 2, 3, 4\rbrace$, we obtain by \eqref{F3 limitado} that
\[
\mathcal{F}_{3}(y,P,\theta,v,v_{0})\in L^{2}(\rho_{3}^{2}(0,T);L^{3/2}(\Omega)).
\]

Finally, without difficulties, we can see that $\mathcal{F}_{4}(y,P,\theta,v,v_{0})\in W^{1,3/2}_{0}(\Omega)$. This prove that $\mathcal{F}$ is well defined.

The check that $\mathcal{F}$ is continuous around the origin is done in a similar way. With this, we have the proof of the lemma.
\hfill
\end{proof}

\begin{lem}\label{DF continuo}
    The mapping $\mathcal{F}: \mathcal{E}_{N}\longrightarrow  \mathcal{Z}_{N}$ is continuously differentiable.
\end{lem}
\begin{proof}
We will the proof for $N=3$ (the case $N=2$ is similar). Let us first prove that $\mathcal{F}$ is Gâteaux-differentiable at any $(y,P,\theta,v,v_{0})\in  \mathcal{E}_{3}$ and let us
compute the \textit{G-derivative} $\mathcal{F}^{\prime}(y,P,\theta,v,v_{0})$. 

Let us fix $(y,P,\theta,v,v_{0})\in  \mathcal{E}_{3}$ and let us take $(y^{\prime},P^{\prime},\theta^{\prime},v^{\prime},v_{0}^{\prime})\in  \mathcal{E}_{3}$ and $\sigma >0$. Also, by the decomposition made in \eqref{Mapas Fi}, we introduce the linear mapping $\mathcal{DF}: \mathcal{E}_{3}\longrightarrow  \mathcal{Z}_{3}$ with $\mathcal{DF}(y,P,\theta,v,v_{0})=\mathcal{DF}=(\mathcal{DF}_{1},\mathcal{DF}_{2},\mathcal{DF}_{3},\mathcal{DF}_{4})$ where
\begin{equation}\label{Df i}
 \left\{ \begin{array}{l}
 \mathcal{DF}_{1}(y^{\prime},P^{\prime},\theta^{\prime},v^{\prime},v_{0}^{\prime}):= y^{\prime}_{t}-\nu(\nabla y)\Delta y^{\prime} - 2\nu_{1}(\nabla y,\nabla y^{\prime})\Delta y + \nabla P^{\prime}\vspace{0.1cm}\\
 \hspace{3.5cm}+\, (y^{\prime}\cdot\nabla)y + (y\cdot\nabla)y^{\prime}-{\nu_{0}}\theta^{\prime}e_{3}-v^{\prime}\tilde{1}_{\omega},\vspace{0.15cm}  \\
  \mathcal{DF}_{2}(y^{\prime},P^{\prime},\theta^{\prime},v^{\prime},v_{0}^{\prime}):= y^{\prime}(.,0), \vspace{0.15cm}\\
  \mathcal{DF}_{3}(y^{\prime},P^{\prime},\theta^{\prime},v^{\prime},v_{0}^{\prime}):= \theta^{\prime}_{t}-\nu(\nabla y)\Delta \theta^{\prime} - 2\nu_{1}(\nabla y,\nabla y^{\prime})\Delta\theta \vspace{0.1cm} \\
\hspace{3.5cm} +\, y^{\prime}\cdot\nabla\theta + y\cdot\nabla\theta^{\prime}-v_{0}^{\prime}\tilde{1}_{\omega} -\nu(\nabla y) D y:\nabla y^{\prime}    \vspace{0.1cm}\\
  \hspace{3.5cm}-\,\left[\nu(\nabla y) D y^{\prime} + 2\nu_{1}(\nabla y,\nabla y^{\prime})D y \right]:\nabla y,\vspace{0.15cm}\\

  \mathcal{DF}_{4}(y^{\prime},P^{\prime},\theta^{\prime},v^{\prime},v_{0}^{\prime}):= \theta^{\prime}(.,0).
    \end{array}\right.
\end{equation}

From the definition of the spaces $\mathcal{E}_{3}, \mathcal{Z}_{3}$ and \eqref{Df i}, it
becomes clear that $\mathcal{DF}\in \mathcal{L}(\mathcal{E}_{3}, \mathcal{Z}_{3})$. Furthermore, for each $j=\lbrace 1, 2, 3, 4\rbrace$ we have
\begin{equation}\label{Converg de Fi}
\begin{array}{l}
    \dfrac{1}{\sigma}[\mathcal{F}_{j}\left((y,P,\theta,v,v_{0})+\sigma(y^{\prime},P^{\prime},\theta^{\prime},v^{\prime},v_{0}^{\prime})\right) -\mathcal{F}_{j}(y,P,\theta,v,v_{0})]\,\vspace{0.1cm}\\
\hspace{-1.5cm}\text{converges to}\, \mathcal{DF}_{j}(y^{\prime},P^{\prime},\theta^{\prime},v^{\prime},v_{0}^{\prime})\, \text{strong in}\, \mathcal{Z}_{3},\, \text{as}\, \sigma\longrightarrow 0. 
\end{array}
\end{equation}

Let us show that \eqref{Converg de Fi} is true. Firstly, we state that,
\begin{equation}\label{Converg de F1}
\begin{array}{l}
    \dfrac{1}{\sigma}[\mathcal{F}_{1}\left((y,P,\theta,v,v_{0})+\sigma(y^{\prime},P^{\prime},\theta^{\prime},v^{\prime},v_{0}^{\prime})\right) -\mathcal{F}_{1}(y,P,\theta,v,v_{0})]\,\vspace{0.1cm}\\
\hspace{-1.5cm}\text{converges to}\, \mathcal{DF}_{1}(y^{\prime},P^{\prime},\theta^{\prime},v^{\prime},v_{0}^{\prime})\, \text{strong in}\, L^{2}(\rho_{3}^{2};Q)^{3},\, \text{as}\, \sigma\longrightarrow 0. 
\end{array}
\end{equation}
Indeed,
\begin{equation*}
\begin{array}{l}
    \Vert\dfrac{1}{\sigma}[\mathcal{F}_{1}\left((y,P,\theta,v,v_{0})+\sigma(y^{\prime},P^{\prime},\theta^{\prime},v^{\prime},v_{0}^{\prime})\right) -\mathcal{F}_{1}(y,P,\theta,v,v_{0})]\vspace{0.1cm}\\
    \,\,\,- \mathcal{DF}_{1}(y^{\prime},P^{\prime},\theta^{\prime},v^{\prime},v_{0}^{\prime})\Vert_{L^{2}(\rho_{3}^{2};Q)^{3}}\leq \sigma\Vert (y^{\prime}\cdot\nabla)y^{\prime}\Vert_{L^{2}(\rho_{3}^{2};Q)^{3}}\vspace{0.1cm}\\
\,\,\,+\,\Vert\dfrac{\nu_{1}}{\sigma}\left(\Vert\nabla (y + \sigma y^{\prime})\Vert^{2}-\Vert\nabla y\Vert^{2}\right)\Delta y - 2\nu_{1}(\nabla y,\nabla y^{\prime})\Delta y\Vert_{L^{2}(\rho_{3}^{2};Q)^{3}}\vspace{0.1cm}\\     
     \,\,\,+\,\Vert\nu_{1}(\Vert\nabla (y+\sigma y^{\prime})\Vert^{2}-\Vert\nabla y\Vert^{2})\Delta y^{\prime}\Vert_{L^{2}(\rho_{3}^{2};Q)^{3}}= H_{1}+H_{2}+H_{3}.
    \end{array}
\end{equation*}
Note that, as proved in the Lemma \ref{F bem definido},
\begin{equation*}
    H^{2}_{1}\leq C\sigma\Vert(y^{\prime},P^{\prime},\theta^{\prime},v^{\prime},v_{0}^{\prime})\Vert^{4}_{ \mathcal{E}_{3}}
\end{equation*}
and we see that, $H_{1}\longrightarrow 0$, as $\sigma\longrightarrow 0$.

Also, as $\sigma\longrightarrow 0$,
\begin{equation*}
\begin{array}{l}
H^{2}_{2}=\nu_{1}^{2}\displaystyle\iint\limits_{Q}\rho_{3}^{2}\vert \dfrac{1}{\sigma}\left(\Vert\nabla ( y + \sigma y^{\prime})\Vert^{2}-\Vert\nabla y\Vert^{2}\right)\Delta y - 2(\nabla y,\nabla y^{\prime})\Delta y\vert^{2}dxdt\longrightarrow 0 
\end{array}
\end{equation*}
and 
\begin{equation*}
H^{2}_{3}=\nu_{1}^{2}\displaystyle\iint\limits_{Q}\rho_{3}^{2}\vert\,\Vert\nabla (y+\sigma y^{\prime})\Vert^{2}-\Vert\nabla y\Vert^{2}\,\vert^{2}\vert\Delta y^{\prime}\vert^{2}dxdt \longrightarrow 0.
\end{equation*}
Thus, (\ref{Converg de F1}) holds.

Now,  that the difference quotient
\begin{equation}\label{Converg de F2 e F4}
    \dfrac{1}{\sigma}[\mathcal{F}_{j}\left((y,P,\theta,v,v_{0})+\sigma(y^{\prime},P^{\prime},\theta^{\prime},v^{\prime},v_{0}^{\prime})\right) -\mathcal{F}_{j}(y,P,\theta,v,v_{0})]
\end{equation}
converges to $\mathcal{DF}_{j}(y^{\prime},P^{\prime},\theta^{\prime},v^{\prime},v_{0}^{\prime})$ strong for $j=2$ and $j=4$, respectively, in $V$ and $W_{0}^{1,3/2}(\Omega)$,  as $\sigma\longrightarrow 0$, is immediate. 

Finally, let's see that
\begin{equation}\label{Converg de F3}
\begin{array}{l}
    \dfrac{1}{\sigma}[\mathcal{F}_{3}\left((y,P,\theta,v,v_{0})+\sigma(y^{\prime},P^{\prime},\theta^{\prime},v^{\prime},v_{0}^{\prime})\right) -\mathcal{F}_{3}(y,P,\theta,v,v_{0})]\vspace{0.1cm}\\

\hspace{-1.6cm}\, \text{converges to} \,\mathcal{DF}_{3}(y^{\prime},P^{\prime},\theta^{\prime},v^{\prime},v_{0}^{\prime})\, \text{strong in}\, L^{2}(\rho_{3}^{2}(0,T);L^{3/2}(\Omega)), \, \text{as}\, \sigma\longrightarrow 0.
\end{array}
\end{equation}

For simplicity, we omit the notation of inequality norms below but let it be clear that they are all norms in $L^{2}(\rho^{2}_{3}(0,T);L^{3/2}(\Omega))$. More precisely,
\begin{equation*}
    \begin{array}{l}
     \Vert  \dfrac{1}{\sigma}[\mathcal{F}_{3}\left((y,P,\theta,v,v_{0})+\sigma(y^{\prime},P^{\prime},\theta^{\prime},v^{\prime},v_{0}^{\prime})\right) -\mathcal{F}_{3}(y,P,\theta,v,v_{0})]\vspace{0.1cm} \\
     
\,\,\,  -\mathcal{DF}_{3}(y^{\prime},P^{\prime},\theta^{\prime},v^{\prime},v_{0}^{\prime})\Vert\leq \sigma\Vert y^{\prime}\cdot\nabla\theta^{\prime}\Vert\vspace{0.1cm}\\

\,\,\,+ \sigma\Vert(\nu_{0}+\nu_{1}\Vert\nabla(y+\sigma y^{\prime})\Vert^{2})Dy^{\prime}:\nabla y^{\prime}\Vert\vspace{0.1cm}\\

\,\,\, +\Vert\dfrac{\nu_{1}}{\sigma}\left( \Vert\nabla(y+\sigma y^{\prime})\Vert^{2}-\Vert\nabla y\Vert^{2}\right)\Delta\theta-2\nu_{1}(\nabla y,\nabla y^{\prime})\Delta\theta\Vert\vspace{0.1cm}\\

\,\,\, +\Vert\nu_{1}(\Vert\nabla(y+\sigma y^{\prime})\Vert^{2}-\Vert\nabla y\Vert^{2})\Delta\theta^{\prime}\Vert\vspace{0.1cm}\\

\,\,\, +\Vert\dfrac{\nu_{1}}{\sigma}\left( \Vert\nabla(y+\sigma y^{\prime})\Vert^{2}-\Vert\nabla y\Vert^{2}\right)Dy :\nabla y-2\nu_{1}(\nabla y,\nabla y^{\prime})D y:\nabla y\Vert\vspace{0.1cm}  \\

\,\,\, +\Vert\nu_{1}(\Vert\nabla(y+\sigma y^{\prime})\Vert^{2}-\Vert\nabla y\Vert^{2})Dy:\nabla y^{\prime}\Vert\vspace{0.1cm}\\

\,\,\,+\Vert\nu_{1}(\Vert\nabla(y+\sigma y^{\prime})\Vert^{2}-\Vert\nabla y\Vert^{2})Dy^{\prime}:\nabla y\Vert=\displaystyle\sum_{j=1}^{7}I_{j}.
\end{array}
\end{equation*}

By the same arguments from the proof of Lemma \ref{F bem definido} (see, $X_2$, $X_3$ and $X_4$) together with those used in \eqref{Converg de F1}, we have
\[
\displaystyle\sum_{j=1}^{5}I_{j}\longrightarrow 0,\,\,\text{as}\,\, \sigma\longrightarrow 0.
\]
Also, since
\begin{equation*}
\begin{array}{cc}
    I^{2}_{6}\leq
C\displaystyle\int_{0}^{T}\sigma^{2}\left(2\Vert\nabla y\Vert\Vert\nabla y^{\prime}\Vert+\sigma^{2}\Vert\nabla y^{\prime}\Vert^{2}\right)^{2}\mu_{2}^{4}\left[\,\displaystyle\int_{\Omega}\vert\nabla y\vert^{3/2}\vert\nabla y^{\prime}\vert^{3/2}dx \right]^{4/3}dt\vspace{0.1cm}\\
\leq C\,\sigma^{2}\left(\Vert \mu_{2}y\Vert^{2}_{L^{\infty}(0,T;V)}\Vert \mu_{1}y^{\prime}\Vert^{2}_{L^{2}(0,T;V)} + \sigma^{2}\Vert \mu_{2}y^{\prime}\Vert^{4}_{L^{\infty}(0,T;V)}\right)\vspace{0.1cm}\\
\displaystyle\int_{0}^{T}\mu_{2}^{4}\left[\,\left(\displaystyle\int_{\Omega}\vert\nabla y\vert^{3}dx\right)^{1/2}\left(\displaystyle\int_{\Omega}\vert\nabla y^{\prime}\vert^{3}dx\right)^{1/2}\right]^{4/3}dt\vspace{0.1cm}\\
= C\,\sigma^{2}\left(\Vert \mu_{2}y\Vert^{2}_{L^{\infty}(0,T;V)}\Vert \mu_{1}y^{\prime}\Vert^{2}_{L^{2}(0,T;V)} + \sigma^{2}\Vert \mu_{2}y^{\prime}\Vert^{4}_{L^{\infty}(0,T;V)}\right)\vspace{0.1cm}\\
\left(\Vert\nabla(\mu_{2}y)\Vert^{4}_{L^{4}(0,T;L^{3}(\Omega)^{3})}+\Vert\nabla(\mu_{2}y^{\prime})\Vert^{4}_{L^{4}(0,T;L^{3}(\Omega)^{3})}\right)
    \end{array}
\end{equation*}
and from continuous embedding (\ref{imersao cap}), we have that $I_{6}$ is bounded and therefore $I_{6}\longrightarrow 0$, as $\sigma\longrightarrow 0$. By the same arguments we also have $I_{7}\longrightarrow 0$, as $\sigma\longrightarrow 0$. Consequently, \eqref{Converg de F3} is hold.

Therefore, from \eqref{Converg de F1}, \eqref{Converg de F2 e F4} and \eqref{Converg de F3} we have \eqref{Converg de Fi} and $\mathcal{F}$ is G\^ateaux-differentiable at any $(y,P,\theta,v,v_{0})\in  \mathcal{E}_{3}$, with \textit{G-derivative} $\mathcal{F}^{\prime}(y,P,\theta,v,v_{0})={\mathcal{DF}}(y,P,\theta,v,v_{0})$.

Now, let us prove that $(y,P,\theta,v,v_{0})\longmapsto\mathcal{F}^{\prime}(y,P,\theta,v,v_{0})$ is a continuous mapping. Thus, we will prove that $\mathcal{F}$ is not only G\^ateaux-differentiable,
but also Fr\'echet-differentiable and $\mathcal{C}^{1}$. For that, suppose that $$(y_m, P_m,\theta_m, v_m,{v_{0}}_{m})\longrightarrow (y, P, \theta, v, v_{0})\, \text{in}\,  \mathcal{E}_{3}$$ and we will prove the existence of $\varepsilon_{m}(y,P,\theta,v,v_{0})$ such that
\begin{equation}\label{convergencia da derivada}
    \begin{array}{c}
\Vert\left(\mathcal{F}^{\prime}(y_m, P_m,\theta_m, v_m,{v_{0}}_{m})- \mathcal{F}^{\prime}(y, P,\theta, v,v_{0})\right)(y^{\prime},P^{\prime},\theta^{\prime},v^{\prime},v_{0}^{\prime})\Vert^{2}_{ \mathcal{Z}_{3}}\vspace{0.1cm}\\

\leq\varepsilon_{m}\Vert(y^{\prime},P^{\prime},\theta^{\prime},v^{\prime},v_{0}^{\prime})\Vert^{2}_{\mathcal{E}_{3}},
    \end{array}
\end{equation}
for all $(y^{\prime},P^{\prime},\theta^{\prime},v^{\prime},v_{0}^{\prime})\in  \mathcal{E}_{3}$ and $\displaystyle\lim_{m\to\infty}\varepsilon_{m}=0$.

In order to simplify the notation, we will consider
\begin{equation*}
\mathcal{D}_{j,m}:=\mathcal{F}_{j}^{\prime}(y_m, P_m,\theta_m, v_m,{v_{0}}_{m})-\mathcal{F}_{j}^{\prime}(y, p,\theta, v,v_{0}).
\end{equation*}
So, notice that
\begin{equation*}
\begin{array}{l}
\bullet\,\,\,  \Vert\mathcal{D}_{1,m}(y^{\prime},P^{\prime},\theta^{\prime},v^{\prime},v_{0}^{\prime})\Vert^{2}_{L^{2}(\rho_{3}^{2};Q)^{3}}
\leq 3\Vert\nu_{1}\left(\Vert\nabla y\Vert^{2}-\Vert\nabla y_{m}\Vert^{2} 
 \right)\Delta y^{\prime}\Vert^{2}_{L^{2}(\rho_{3}^{2};Q)^{3}}\vspace{0.1cm}\\
 +\, 3\Vert 2\nu_{1}(\nabla y,\nabla y^{\prime})\Delta y - 2\nu_{1}(\nabla y_{m},\nabla y^{\prime})\Delta y_{m}\Vert^{2}_{L^{2}(\rho_{3}^{2};Q)^{3}}\vspace{0.1cm}\\
 +\,3\Vert(y^{\prime}\cdot\nabla)(y_{m}-y)+((y_{m}-y)\cdot\nabla)
y^{\prime}\Vert^{2}_{L^{2}(\rho_{3}^{2};Q)^{3}}\vspace{0.1cm}\\
=3K_{1}+12K_{2}+3K_{3}.
 \end{array}
\end{equation*}
Since,
\begin{equation*}
    \begin{array}{c}
    K_{1}\leq C\left(\Vert\Vert\nabla (y-y_{m})\Vert \Vert\nabla y \Vert\Delta y^{\prime}\Vert^{2}_{L^{2}(\rho_{3}^{2};Q)^{3}} +
     \Vert\Vert\nabla (y-y_{m})\Vert\Vert\nabla y_{m}\Vert\Delta y^{\prime}\Vert^{2}_{L^{2}(\rho_{3}^{2};Q)^{3}}\right)
\end{array}
\end{equation*}
then, using the same arguments as \eqref{des com nabla e delta}, we conclude that
\begin{equation*}
    \begin{array}{c}
  K_{1}\leq \varepsilon_{1,m}\Vert(y^{\prime},P^{\prime},\theta^{\prime},v^{\prime},v_{0}^{\prime})\Vert^{2}_{ \mathcal{E}_{3}}     
    \end{array}
\end{equation*}
where
\begin{equation*}
\begin{array}{c}
\varepsilon_{1,m}= C\Vert(y_m, P_m,\theta_m, v_m,{v_{0}}_{m})-(y, P,\theta, v,{v_{0}})\Vert^{2}_{\mathcal{E}_{3}}\left( \Vert(y, P,\theta, v,{v_{0}})\Vert^{2}_{\mathcal{E}_{3}}\right.\vspace{0.1cm}\\
\left.+\Vert(y_m, P_m,\theta_m, v_m,{v_{0}}_{m})\Vert^{2}_{\mathcal{E}_{3}}\right).
\end{array}
\end{equation*}
For $K_{2}$ let’s first see the following, adding and subtracing $\nu_{1}(\nabla y_{m},\nabla y^{\prime})\Delta y$ in $K_{2}$, we have 
\begin{equation*}
    \begin{array}{l}
        K_{2}=\Vert-\nu_{1}(\nabla(y_{m}-y),\nabla y^{\prime})\Delta y - \nu_{1}(\nabla y_{m},\nabla y^{\prime})\Delta(y_{m}-y)\Vert^{2}_{L^{2}(\rho_{3}^{2};Q)^{3}}\vspace{0.1cm}\\
        \leq C\,\left(\Vert\nu_{1} (\nabla(y_{m}-y),\nabla y^{\prime})\Delta y\Vert^{2}_{L^{2}(\rho_{3}^{2};Q)^{3}}+\Vert\nu_{1}(\nabla y_{m},\nabla y^{\prime})\Delta(y_{m}-y)\Vert^{2}_{L^{2}(\rho_{3}^{2};Q)^{3}}\right)\vspace{0.1cm}\\
        \leq\varepsilon_{2,m}\Vert(y^{\prime},P^{\prime},\theta^{\prime},v^{\prime},v_{0}^{\prime})\Vert^{2}_{\mathcal{E}_{3}},
    \end{array}
\end{equation*}
where
\begin{equation*}
    \begin{array}{c}
    \varepsilon_{2,m}= C\Vert(y_m, P_m,\theta_m, v_m,{v_{0}}_{m})-(y, P,\theta, v,{v_{0}})\Vert^{2}_{\mathcal{E}_{3}}\left( \Vert(y, P,\theta, v,{v_{0}})\Vert^{2}_{\mathcal{E}_{3}}\right.\vspace{0.1cm}\\
\left.+\, \Vert(y_m, P_m,\theta_m, v_m,{v_{0}}_{m})\Vert^{2}_{\mathcal{E}_{3}}\right).
    \end{array}
\end{equation*}
And, by the same reasoning as \eqref{des com y e nabla y},
\begin{equation*}
    \begin{array}{lll}
      K_{3}& \leq& C\left(\Vert(y^{\prime}\cdot\nabla)(y_{m}-y)\Vert^{2}_{L^{2}(\rho_{3}^{2};Q)^{3}}  + \Vert((y_{m}-y)\cdot\nabla)
y^{\prime}\Vert^{2}_{L^{2}(\rho_{3}^{2};Q)^{3}}\right)   \vspace{0.1cm}\\

&\leq &\varepsilon_{3,m}\Vert(y^{\prime},P^{\prime},\theta^{\prime},v^{\prime},v_{0}^{\prime})\Vert^{2}_{\mathcal{E}_{3}},
    \end{array}
\end{equation*}
with
\begin{equation*}
\varepsilon_{3,m}=C\Vert(y_m, P_m,\theta_m, v_m,{v_{0}}_{m})-(y, P,\theta, v,{v_{0}})\Vert^{2}_{\mathcal{E}_{3}}.
\end{equation*}

It is easy to check that $\mathcal{D}_{j,m}$ for $j=2$ and $j=4$ satisfy similar inequalities.

Again, all inequality norms below are norms in $L^{2}(\rho_{3}^{2}(0,T);L^{3/2}(\Omega))$, we will omit them for simplicity. For $\mathcal{D}_{3,m}$ after some calculations we get the following:
\begin{equation*}
    \begin{array}{l}
\bullet\,\,\, \Vert\mathcal{D}_{3,m}(y^{\prime},P^{\prime},\theta^{\prime},v^{\prime},v_{0}^{\prime})\Vert^{2}\leq C\left[ \Vert\, \Vert\nabla(y_{m}-y)\Vert\,\Vert\nabla y\Vert\, \Delta\theta^{\prime}\,\Vert^{2}\right.\vspace{0.1cm}\\ \,+ \Vert\,\Vert\nabla(y_{m}-y)\Vert\,\Vert\nabla y_{m}\Vert\Delta\theta^{\prime}\,\Vert^{2} + \Vert\nu_{1}(\nabla y_{m},\nabla y^{\prime})\Delta(\theta_{m}-\theta)\Vert^{2}\vspace{0.1cm}\\
       \,+\Vert\nu_{1}(\nabla(y_{m}-y),\nabla y^{\prime})\Delta\theta\Vert^{2} + \Vert(\nu_{0}+\nu_{1}\Vert\nabla y_{m}\Vert^{2})D(y_{m}-y):\nabla y^{\prime}\Vert^{2}\vspace{0.1cm}\\
       \,+ \Vert\nu_{1}(\Vert\nabla y_{m}\Vert^{2}-\Vert\nabla y\Vert^{2})Dy:\nabla y^{\prime}\Vert^{2} + \Vert(\nu_{0}+\nu_{1}\Vert \nabla y_{m}\Vert^{2})D y^{\prime}:\nabla (y_{m}-y)\Vert^{2}\vspace{0.1cm}\\
       \,+ \Vert \nu_{1}(\Vert\nabla y_{m}\Vert^{2}-\Vert\nabla y\Vert^{2})D y^{\prime}:\nabla y \Vert^{2} +\Vert\nu_{1}(\nabla y_{m},\nabla y^{\prime})D y_{m}:\nabla (y_{m}-y)\Vert^{2} \vspace{0.1cm}\\
      \, + \Vert\nu_{1}(\nabla (y_{m}-y),\nabla y^{\prime})D y_{m}:\nabla y \Vert^{2} + \Vert \nu_{1}(\nabla y, \nabla y^{\prime})D(y_{m}-y):\nabla y\Vert^{2}\vspace{0.1cm}\\
      \,\left. +\Vert y^{\prime}\cdot\nabla (\theta_{m}-\theta)\Vert^{2}+\Vert(y_{m}-y)\cdot\nabla\theta^{\prime}\Vert^{2}\right]  = C\displaystyle\sum_{s=4}^{16}{K}_{s}. 
       \end{array}
\end{equation*}
Let's check some terms,
\begin{equation*}
    \begin{array}{l}
      K_{4}\leq C \displaystyle\int_{0}^{T}\left(\displaystyle\int_{\Omega}\mu_{2}^{3}\Vert\nabla(y_{m}-y)\Vert^{3/2}\Vert\nabla y\Vert^{3/2}\vert\Delta\theta^{\prime}\vert^{3/2}dx\right)^{4/3}dt \vspace{0.1cm}\\
      \leq C\Vert\mu_{2}(y_{m}-y)\Vert^{2}_{L^{\infty}(0,T;H^{1}(\Omega)^{3})}\Vert\mu_{2} y\Vert^{2}_{L^{\infty}(0,T;H^{1}(\Omega)^{3})}\Vert\mu_{2}\theta^{\prime}\Vert^{2}_{L^{2}(0,T;W^{2,3/2}(\Omega))}\vspace{0.1cm}\\
      \leq \varepsilon_{4,m}\Vert(y^{\prime},P^{\prime},\theta^{\prime},v^{\prime},v_{0}^{\prime})\Vert^{2}_{\mathcal{E}_{3}},     
    \end{array}
\end{equation*}
where 
\[
\varepsilon_{4,m}=C\Vert(y_m, P_m,\theta_m, v_m,{v_{0}}_{m})-(y, P,\theta, v,{v_{0}})\Vert^{2}_{\mathcal{E}_{3}}\Vert(y,P,\theta,v,v_{0})\Vert^{2}_{\mathcal{E}_{3}}.
\]
Also,
\begin{equation*}
    \begin{array}{l}
        K_{8}\leq C\displaystyle\int_{0}^{T}\left(\displaystyle\int_{\Omega}\mu_{2}^{3}\vert(\nu_{0}+\nu_{1}\Vert\nabla y_{m}\Vert^{2})\vert^{3/2}\vert\nabla (y_{m}-y)\vert^{3/2}\vert\nabla y^{\prime}\vert^{2}dx \right)^{4/3}dt\vspace{0.1cm} \\
        \leq C\,\left(1+\Vert \mu_{2}y_{m}\Vert^{4}_{L^{\infty}(0,T;V)} \right)\Vert\vert\nabla (\mu_{2}y^{\prime})\vert^{2}\Vert_{L^{2}(0,T;L^{3/2}(\Omega)^{3})}\vspace{0.1cm}\\
        \,\,\,\,\,\Vert \vert\nabla(\mu_{2}(y_{m}-y))\vert^{2}\Vert_{L^{2}(0,T;L^{3/2}(\Omega)^{3})}
        \leq \varepsilon_{8,m}\Vert(y^{\prime},P^{\prime},\theta^{\prime},v^{\prime},v_{0}^{\prime})\Vert^{2}_{\mathcal{E}_{3}},         
    \end{array}
\end{equation*}
where
\[
\varepsilon_{8,m}=C(1+\Vert(y_{m},P_{m},\theta_{m},v_{m},{v_{0}}_{m})\Vert^{4}_{\mathcal{E}_{3}})\Vert(y_m, P_m,\theta_m, v_m,{v_{0}}_{m})-(y, P,\theta, v,{v_{0}})\Vert^{2}_{\mathcal{E}_{3}}.
\]
And,
\begin{equation*}
    \begin{array}{l}
       K_{16}\leq \varepsilon_{16,m} \Vert(y^{\prime},P^{\prime},\theta^{\prime},v^{\prime},v_{0}^{\prime})\Vert^{2}_{\mathcal{E}_{3}},           
    \end{array}
\end{equation*}
where 
\[
\varepsilon_{16,m}=C\Vert(y_m, P_m,\theta_m, v_m,{v_{0}}_{m})-(y, P,\theta, v,{v_{0}})\Vert^{2}_{\mathcal{E}_{3}}.
\]

The other terms follow analogously. 

Thus, we have $\lim\limits_{m\to\infty}\varepsilon_{s,m}=0$ for all $s\in\lbrace 1,...,16\rbrace$ and consequently follows \eqref{convergencia da derivada}. This ends the proof.
\hfill
\end{proof}

\begin{lem}\label{Mapa sobrejetivo}
Let $\mathcal{F}$ be the mapping in \eqref{Mapa F}-\eqref{Mapas Fi}. Then, $\mathcal{F}^{\prime}(0,0,0,0,0)$ is onto.
\end{lem}
\begin{proof}
Let $(F_{1},y^{0},F_{2},\theta^{0})\in \mathcal{Z}_{N}$. From Proposition \ref{proposição 2.1} we know there exists $(y,P,\theta,v,v_{0})$ satisfying \eqref{lad.Bous.Linear} and \eqref{2.8 de CaraLimacoNany}. Furthermore, from the usual regularity results for the Stokes system we have $(y,P)\in (L^{2}(0,T;D(A))\times L^{2}(0,T;H^{1}(\Omega)))$. Consequently, $(y,P,\theta,v,v_{0})\in \mathcal{E}_{N}$ and $$\mathcal{F}^{\prime}(0,0,0,0,0)(y,P,\theta,v,v_{0})=(F_{1},y^{0},F_{2},\theta^{0}).$$ 
\hfill
\end{proof}

\noindent\textbf{Proof of Theorem $\ref{Teo principal control nulo}$}
We conclude from Lemmas \ref{F bem definido}--\ref{Mapa sobrejetivo} that the Inverse Mapping Theorem (Theorem \ref{Liusternik}) can be applied to the spaces $\mathcal{E}_{N}$ and $\mathcal{Z}_{N}$ together with the mapping $\mathcal{F}$ introduced
at the beginning of this Section. Thus, there exists $\delta>0$ such that, for every $(y^{0},\theta^{0})\in V\times W_{0}^{1,3/2}(\Omega)$ satisfying
$\Vert (y^{0},\theta^{0})\Vert_{V\times W_{0}^{1,3/2}}< \delta,$
there exists controls $v\in L^{2}(\omega\times (0,T))^{N}$ and $v_{0}\in L^{2}(\omega\times (0,T))$ and associated solutions $(y,p,\theta)$ to (\ref{lad.boussinesq}) such that $y(x,T)=0$ and $\theta(x,T)=0$ in $\Omega$.

This proves that, the nonlinear system \eqref{lad.boussinesq} is locally null-controllable at time $T>0$.

\subsection{Development for the proof of the Theorem  \ref{Teo extra control nulo}}
Let 
\begin{equation}\label{espaço UN}
    \begin{array}{c}
   \mathcal{U}_{N}=(y,P,\theta,v,v_{0}): \rho_{1}y \in L^{2}(Q)^{N}, \rho_{2}v, (\kappa v)_{t}, \kappa\Delta v\in L^{2}(\omega\times (0,T))^{N},  \vspace{0.1cm}\\
   y\in L^{q}(0,T;W^{2,p}(\Omega)^{N}), \nabla y\in L^{2}(Q)^{N\times N},
   P\in L^{q}(0,T;L^{p}(\Omega)),
   \rho_{1}\theta\in L^{2}(Q), \vspace{0.1cm}\\ 
   \rho_{2}v_{0}, (\kappa v_{0})_{t}, \kappa \Delta v_{0}\in L^{2}(\omega\times (0,T)),
   \theta\in L^{q}(0,T;W^{2,p}(\Omega)),\vspace{0.1cm}\\ 
  \text{for}\, F_{1}:=\mathcal{L}_{1}y+\nabla P - \nu_{0}\theta e_{N} - v\tilde{1}_{\omega} \,
  \text{and}\, F_{2}:=\mathcal{L}_{2}\theta -v_{0}\tilde{1}_{\omega},   
  \rho_{3}F_{1}\in L^{q}(0,T;L^{p}(\Omega)^{N}),\vspace{0.1cm}\\
   \rho_{3}F_{2}\in L^{q}(0,T;L^{p}(\Omega)),   
  \nabla\cdot y\equiv 0,
  y(.,0)\in V^{p}, \theta(.,0)\in W_{0}^{1,p}(\Omega), 
\vspace{0.1cm}\\
y\mid_{\Sigma} =0, \theta\mid_{\Sigma} =0,\text{where}\,\,  3<p\leq 6 \,\, \text{and}\,\, p<q<\infty\rbrace,  
    \end{array}
\end{equation}
It's clear that ${\mathcal{U}}_{N}$ is a Banach space for the norm $\Vert .\Vert_{ {\mathcal{U}}_{N}}$, with
\begin{equation*}
    \begin{array}{l}
     \Vert (y,P,\theta,v,v_{0})\Vert^{q}_{ {\mathcal{U}}_{N}}:= \Vert y\Vert^{q}_{L^{q}(0,T;W^{2,p}(\Omega)^{N})} + \Vert\theta\Vert^{q}_{L^{q}(0,T;W^{2,p}(\Omega))}\vspace{0.1cm}  
    \, + \,\Vert \rho_{1}y\Vert^{q}_{L^{2}(Q)^{N}} \vspace{0.1cm}\\
    +\, \Vert\rho_{1}\theta\Vert^{q}_{L^{2}(Q)} + \Vert P\Vert^{q}_{L^{q}(0,T;L^{p}(\Omega))} + \Vert\rho_{2}v\Vert^{q}_{L^{2}(\omega\times (0,T))^{N}} 
    \, +\, \Vert\rho_{2}v_{0}\Vert^{q}_{L^{2}(\omega\times (0,T))} \vspace{0.1cm}\\
   +\Vert(\kappa v)_{t}\Vert^{q}_{L^{2}(\omega\times (0,T))^{N}} \, 
   
   +\Vert\kappa \Delta v\Vert^{q}_{L^{2}(\omega\times (0,T))^{N}} \,
    \, +\, \Vert(\kappa v_{0})_{t}\Vert^{q}_{L^{2}(\omega\times (0,T))} \vspace{0.1cm}\\
    
    + \Vert\kappa \Delta v_{0}\Vert^{q}_{L^{2}(\omega\times (0,T))}
      
    + \Vert\rho_{3}F_{1}\Vert^{q}_{L^{q}(0,T;L^{p}(\Omega)^{N})}
     
     +\,\Vert\rho_{3}F_{2}\Vert^{q}_{L^{q}(0,T;L^{p}(\Omega))}\vspace{0.1cm}\\

     +\, \Vert y(.,0)\Vert^{q}_{V^{p}}+\,\Vert \theta(.,0)\Vert^{q}_{W_{0}^{1,p}(\Omega)}.
    \end{array}
\end{equation*}

Now, let us introduce the Banach space
\begin{equation}\label{espaço RN}
    {\mathcal{R}}_{N}:=L^{q}(\rho_{3}^{q}(0,T);L^{p}(\Omega)^{N})\times V^{p}\times L^{q}(\rho_{3}^{q}(0,T);L^{p}(\Omega))\times W^{1,p}_{0}(\Omega),
\end{equation}
and the mapping ${\mathcal{I}}:{\mathcal{U}}_{N}\rightarrow {\mathcal{R}}_{N}$, such that
\begin{equation}\label{Mapa F extra}
   {\mathcal{I}}(y,P,\theta,v,v_{0})=({\mathcal{I}}_{1},{\mathcal{I}}_{2}, 
   {\mathcal{I}}_{3},{\mathcal{I}}_{4})(y,P,\theta,v,v_{0})
\end{equation}
where
\begin{equation}\label{Mapas Fi extras}
    \left\{
    \begin{array}{l}
      {\mathcal{I}}_{1}(y,P,\theta,v,v_{0}):=y_{t}-\bar{\nu}(\nabla y)\Delta y + (y\cdot\nabla)y +\nabla P - \nu_{0}\theta e_{N} - v\tilde{1}_{\omega},\\
      {\mathcal{I}}_{2}(y,P,\theta,v,v_{0}):= y(.,0),\\
      {\mathcal{I}}_{3}(y,P,\theta,v,v_{0}):= \theta_{t} - \bar{\nu}(\nabla\theta)\Delta\theta +y\cdot\nabla\theta
      -\bar{\nu}(\nabla y)Dy:\nabla y - v_{0}\tilde{1}_{\omega} ,
      \\
      {\mathcal{I}}_{4}(y,P,\theta,v,v_{0}):= \theta(.,0).
    \end{array}\right.
\end{equation}
To simplify the notation, in the norms of $L^{p}(\Omega)^{N}$ we will just write $L^{p}(\Omega)$. That said, we have the following results:
\begin{lem}\label{F bem definido extra}
    Let ${\mathcal{I}}: {\mathcal{U}}_{N}\rightarrow  {\mathcal{R}}_{N}$ be given by (\ref{Mapa F extra})-(\ref{Mapas Fi extras}). Then, ${\mathcal{I}}$ is well defined and continuous around the origin. 
\end{lem}
\begin{proof}
Let's prove that, for each $(y,P,\theta,v,v_{0})\in \mathcal{U}_{N}$ we have $\mathcal{I}(y,P,\theta,v,v_{0})\in \mathcal{R}_{N}$.

That ${\mathcal{I}}_{2}$ and ${\mathcal{I}}_{4}$ are well defined follows immediately from the definition of $\mathcal{U}_{N}$. So let's find out ${\mathcal{I}}_{1}$ and ${\mathcal{I}}_{3}$.

\textbf{Analysis of ${\mathcal{I}}_{1}$:}
\begin{equation*}
    \begin{array}{l}
    \bullet \,\Vert\rho_{3}F_{1}\Vert^{q}_{L^{q}(0,T;L^{p}(\Omega))}\leq C\Vert (y,P,\theta,v,v_{0})\Vert^{q}_{\mathcal{U}_{N}}.
    \end{array}
\end{equation*}
Taking into account \eqref{des.16>15} we have $\rho_{3}\kappa^{-2}\leq C$. Moreover, using the fact that $W^{1,p}(\Omega)\hookrightarrow L^{\infty}(\Omega)$ (since $p>N$) and the estimate \eqref{estimate para y em Vp} from the Proposition \ref{Proposicao com rhobarratheta}, we obtain
\begin{equation}\label{y grad y limitado em Lp}
    \begin{array}{l}
   \bullet\, \Vert\rho_{3}(y\cdot\nabla)y\Vert^{q}_{L^{q}(0,T;L^{p}(\Omega))} \leq \displaystyle\int_{0}^{T}\left(\displaystyle\int_{\Omega}\rho_{3}^{p}\vert y\vert^{p}\vert\nabla y\vert^{p}dx
   \right)^{q/p}dt\vspace{0.1cm}\\
   = \displaystyle\int_{0}^{T}\left(\displaystyle\int_{\Omega}\rho_{3}^{p}\kappa^{-2p}\kappa^{2p}\vert y\vert^{p}\vert\nabla y\vert^{p}dx
   \right)^{q/p}dt  \vspace{0.1cm}\\
   \leq C\displaystyle\int_{0}^{T}\Vert\kappa y\Vert_{L^{\infty}(\Omega)}^{q}\left(\displaystyle\int_{\Omega}\vert\kappa\nabla y\vert^{p}dx\right)^{q/p}dt\vspace{0.1cm}\\
   \leq C\Vert\kappa y\Vert^{q}_{L^{\infty}(0,T;W^{1,p}(\Omega))}\Vert\kappa y\Vert^{q}_{L^{q}(0,T;W^{1,p}(\Omega))}
   \leq C\Vert(y,P,\theta,v,v_{0})\Vert^{2q}_{\mathcal{U}_{N}}.
    \end{array}
\end{equation}
In a similar way
\begin{equation}\label{gradiente limitado em Lp}
    \begin{array}{l}
\bullet\,\Vert\rho_{3}\nu_{1}\Vert\nabla y\Vert_{L^{p}}^{2}\Delta y\Vert^{q}_{L^{q}(0,T;L^{p}(\Omega))} \leq C \displaystyle\int_{0}^{T}\rho_{3}^{q}\kappa^{-3q}\Vert\kappa\nabla y\Vert^{2q}_{L^{p}(\Omega)}\Vert \kappa\Delta y\Vert^{q}_{L^{p}(\Omega)}dt\vspace{0.1cm}\\
         \leq C\Vert\kappa y\Vert^{2q}_{L^{\infty}(0,T;W^{1,p}(\Omega))}\Vert\kappa y\Vert^{q}_{L^{q}(0,T;W^{2,p}(\Omega))} \leq C\Vert(y,P,\theta,v,v_{0})\Vert^{3q}_{\mathcal{U}_{N}}.
    \end{array}
\end{equation}
Hence, $\mathcal{I}_{1}(y,P,\theta,v,v_{0})\in L^{q}(\rho_{3}^{q}(0,T);L^{p}(\Omega)^{N})$.

\textbf{Analysis of $\mathcal{I}_{3}$:}
\begin{equation*}
   \bullet\, \Vert\rho_{3}F_{2}\Vert^{q}_{L^{q}(0,T;L^{p}(\Omega))}\leq C\Vert(y,P,\theta,v,v_{0})\Vert^{q}_{\mathcal{U}_{N}}.
\end{equation*}
Using the same previous arguments together with the estimates \eqref{estimate para theta W{1,p}} and \eqref{estimate para y em Vp} from the Proposition \ref{Proposicao com rhobarratheta}, we get 
\begin{equation*}
    \begin{array}{l}
      \bullet\,  \Vert\rho_{3} y\cdot\nabla\theta\Vert^{q}_{L^{q}(0,T;L^{p}(\Omega))}\leq C\displaystyle\int_{0}^{T}\Vert\kappa\nabla\theta\Vert^{q}_{L^{\infty}(\Omega)}\left(\displaystyle\int_{\Omega}\vert\kappa y\vert^{p}dx\right)^{q/p}dt\vspace{0.1cm}\\
        \leq C\displaystyle\int_{0}^{T}\Vert\kappa\nabla\theta\Vert^{q}_{W^{1,p}(\Omega)}\Vert\kappa y\Vert^{q}_{L^{p}(\Omega)}dt\leq C \Vert\kappa y\Vert^{q}_{L^{\infty}(0,T;W^{1,p}(\Omega))}\Vert\kappa\theta\Vert^{q}_{L^{q}(0,T;W^{2,p}(\Omega))}\vspace{0.1cm}\\
        \leq C\Vert(y,P,\theta,v,v_{0})\Vert^{2q}_{\mathcal{U}_{N}};
    \end{array}
\end{equation*}
    \begin{equation*}
        \begin{array}{l}
         \bullet\,   \Vert \rho_{3}\nu_{1}\Vert\nabla\theta\Vert^{2}_{L^{p}}\Delta\theta\Vert^{q}_{L^{q}(0,T;L^{p}(\Omega))}\leq C\displaystyle\int_{0}^{T}\rho^{q}_{3}\kappa^{-3q}\Vert\kappa\nabla\theta\Vert^{2q}_{L^{p}(\Omega)}\Vert\kappa\Delta\theta\Vert^{q}_{L^{p}(\Omega)}dt\\
            \leq C\Vert\kappa\theta\Vert^{2q}_{L^{\infty}(0,T;W^{1,p}(\Omega))}\Vert\kappa\theta\Vert^{q}_{L^{q}(0,T;W^{2,p}(\Omega))}\leq C\Vert(y,P,\theta,v,v_{0})\Vert^{3q}_{\mathcal{U}_{N}};
            \end{array}
    \end{equation*}
and, using that $\rho_{3}\kappa^{-4}\leq C$,
\begin{equation*}
    \begin{array}{l}
\bullet\,\Vert\rho_{3}\bar{\nu}(\nabla y)Dy:\nabla y\Vert^{q}_{L^{q}(0,T;L^{p}(\Omega))} \leq C\Vert\kappa y\Vert^{q}_{L^{\infty}(0,T;W^{1,p}(\Omega))}\Vert\kappa y\Vert^{q}_{L^{q}(0,T;W^{2,p}(\Omega))}\vspace{0.1cm}\\
\,+ \,C\Vert\kappa y\Vert^{3q}_{L^{\infty}(0,T;W^{1,p}(\Omega))}\Vert\kappa y\Vert^{q}_{L^{q}(0,T;W^{2,p}(\Omega))}\leq C\Vert(y,P,\theta,v,v_{0})\Vert^{4q}_{\mathcal{U}_{N}}.
    \end{array}
\end{equation*}
Consequently we have $\mathcal{I}_{3}(y,P,\theta,v,v_{0})\in L^{q}(\rho_{3}^{q}(0,T);L^{p}(\Omega))$. 

Using similar arguments it is easy to check the $\mathcal{I}$ is continuous around the origin. This proves the Lemma.
\hfill
\end{proof}

\begin{lem}\label{DF continuo extra}
    The mapping ${\mathcal{I}}: {\mathcal{U}}_{N}\longrightarrow  {\mathcal{R}}_{N}$ is continuously differentiable.
\end{lem}
\begin{proof}
     Let us first prove that $\mathcal{I}$ is Gâteaux-differentiable at any $(y,P,\theta,v,v_{0})\in  \mathcal{U}_{N}$ and let us
compute the \textit{G-derivative} $\mathcal{I}^{\prime}(y,P,\theta,v,v_{0})$. 

Let us fix $(y,P,\theta,v,v_{0})\in  \mathcal{U}_{N}$ and let us take $(y^{\prime},P^{\prime},\theta^{\prime},v^{\prime},v_{0}^{\prime})\in  \mathcal{U}_{N}$ and $\sigma >0$. Also, by the decomposition made in \eqref{Mapas Fi extras}, we introduce the linear mapping $\mathcal{I}: \mathcal{U}_{N}\longrightarrow  \mathcal{R}_{N}$ with $\mathcal{DI}(y,P,\theta,v,v_{0})=\mathcal{DI}=(\mathcal{DI}_{1},\mathcal{DI}_{2},\mathcal{DI}_{3},\mathcal{DI}_{4})$ where
\begin{equation}\label{DIi}
 \left\{ \begin{array}{l}
 \mathcal{DI}_{1}(y^{\prime},P^{\prime},\theta^{\prime},v^{\prime},v_{0}^{\prime}):= y^{\prime}_{t}-\bar{\nu}(\nabla y)\Delta y^{\prime} - \left(2\nu_{1}\Vert\nabla y\Vert_{L^{p}}^{2-p}\displaystyle\int_{\Omega}\vert\nabla y\vert^{p-2}\nabla y\nabla y^{\prime}dx \right)\Delta y\vspace{0.1cm}\\ 
 \hspace{3.5cm} +\, \nabla P^{\prime} +(y^{\prime}\cdot\nabla)y + (y\cdot\nabla)y^{\prime}-{\nu_{0}}\theta^{\prime}e_{3}-v^{\prime}\chi_{\omega},\vspace{0.15cm}  \\
  \mathcal{DI}_{2}(y^{\prime},P^{\prime},\theta^{\prime},v^{\prime},v_{0}^{\prime}):= y^{\prime}(.,0), \vspace{0.15cm}\\
  \mathcal{DI}_{3}(y^{\prime},P^{\prime},\theta^{\prime},v^{\prime},v_{0}^{\prime}):= \theta^{\prime}_{t}-\bar{\nu}(\nabla \theta)\Delta \theta^{\prime} - \left(2\nu_{1}\Vert\nabla \theta\Vert_{L^{p}}^{2-p}\displaystyle\int_{\Omega}\vert\nabla \theta\vert^{p-2}\nabla \theta\nabla \theta^{\prime}dx \right)\Delta\theta\vspace{0.1cm} \\
\hspace{3.5cm} +\, y^{\prime}\cdot\nabla\theta + y\cdot\nabla\theta^{\prime}-v_{0}^{\prime}\chi_{\omega} -\bar{\nu}(\nabla y) D y:\nabla y^{\prime}    \vspace{0.1cm}\\
  \hspace{3.5cm}-\,\left[\bar{\nu}(\nabla y) D y^{\prime} + \left(2\nu_{1}\Vert\nabla y\Vert_{L^{p}}^{2-p}\displaystyle\int_{\Omega}\vert\nabla y\vert^{p-2}\nabla y\nabla y^{\prime}dx \right) D y \right]:\nabla y,\vspace{0.15cm}\\

  \mathcal{DI}_{4}(y^{\prime},P^{\prime},\theta^{\prime},v^{\prime},v_{0}^{\prime}):= \theta^{\prime}(.,0).
    \end{array}\right.
\end{equation}

From the definition of the spaces $\mathcal{U}_{N}, \mathcal{R}_{N}$ and \eqref{DIi}, it
becomes clear that $\mathcal{DI}\in \mathcal{L}(\mathcal{U}_{N}, \mathcal{R}_{N})$. Furthermore, for each $j=\lbrace 1, 2, 3, 4\rbrace$ we have
\begin{equation}\label{Converg de Ii}
\begin{array}{l}
    \dfrac{1}{\sigma}[\mathcal{I}_{j}\left((y,P,\theta,v,v_{0})+\sigma(y^{\prime},P^{\prime},\theta^{\prime},v^{\prime},v_{0}^{\prime})\right) -\mathcal{I}_{j}(y,P,\theta,v,v_{0})]\,\vspace{0.1cm}\\
\hspace{-1.5cm}\text{converges to}\, \mathcal{DI}_{j}(y^{\prime},P^{\prime},\theta^{\prime},v^{\prime},v_{0}^{\prime})\, \text{strong in}\, \mathcal{R}_{N},\, \text{as}\, \sigma\longrightarrow 0. 
\end{array}
\end{equation}

Firstly, notice that, 
\begin{equation*}
\begin{array}{l}
    \Vert\dfrac{1}{\sigma}[\mathcal{I}_{1}\left((y,P,\theta,v,v_{0})+\sigma(y^{\prime},P^{\prime},\theta^{\prime},v^{\prime},v_{0}^{\prime})\right) -\mathcal{I}_{1}(y,P,\theta,v,v_{0})]\vspace{0.1cm}\\
    \,\,\,- \mathcal{DI}_{1}(y^{\prime},P^{\prime},\theta^{\prime},v^{\prime},v_{0}^{\prime})\Vert_{L^{q}(\rho_{3}^{q}(0,T);L^{p}(\Omega))}\leq \sigma\Vert (y^{\prime}\cdot\nabla)y^{\prime}\Vert_{L^{q}(\rho_{3}^{q}(0,T);L^{p}(\Omega))}\vspace{0.1cm}\\
\,\,\,+\,\Vert\dfrac{\nu_{1}}{\sigma}\left(\Vert\nabla (y + \sigma y^{\prime})\Vert_{L^{p}}^{2}-\Vert\nabla y\Vert_{L^{p}}^{2}\right)\Delta y \vspace{0.1cm}\\ 
-\left(2\nu_{1}\Vert\nabla y\Vert_{L^{p}}^{2-p}\displaystyle\int_{\Omega}\vert\nabla y\vert^{p-2}\nabla y\nabla y^{\prime}dx \right)\Delta y \Vert_{L^{q}(\rho_{3}^{q}(0,T);L^{p}(\Omega))}\vspace{0.1cm}\\     
     \,\,\,+\,\Vert\nu_{1}(\Vert\nabla (y+\sigma y^{\prime})\Vert_{L^{p}}^{2}-\Vert\nabla y\Vert_{L^{p}}^{2})\Delta y^{\prime}\Vert_{L^{q}(\rho_{3}^{q}(0,T);L^{p}(\Omega))}= \tilde{H}_{1}+\tilde{H}_{2}+\tilde{H}_{3}.
    \end{array}
\end{equation*}
That $\tilde{H}_1\rightarrow 0$, as $\sigma\rightarrow 0$, is immediate.

Let's analyze $\tilde{H}_3$, using first order Taylor expansion and discarding terms of higher order than $\sigma$, we have
\begin{equation*}
    \begin{array}{l}
         \left(\displaystyle\int_{\Omega}\vert\nabla(y+\sigma y^{\prime})\vert^{p}dx\right)^{2/p} = \left(\displaystyle\int_{\Omega}\vert\nabla y\vert^{p}dx + \sigma\displaystyle\int_{\Omega}p\vert\nabla y\vert^{p-2}\nabla y\nabla y^{\prime}dx\right)^{2/p}\vspace{0.1cm}\\
         =\left(\displaystyle\int_{\Omega}\vert\nabla y\vert^{p}dx\right)^{2/p} + \dfrac{2\sigma}{p}\left(\displaystyle\int_{\Omega}\vert\nabla y\vert^{p}dx\right)^{(2-p)/p}\displaystyle\int_{\Omega}p\vert\nabla y\vert^{p-2}\nabla y\nabla y^{\prime}dx \vspace{0.1cm}\\

         = \Vert \nabla y\Vert^{2}_{L^{p}} + 2\sigma\Vert\nabla y\Vert^{2-p}_{L^{p}}\displaystyle\int_{\Omega}\vert\nabla y\vert^{p-2}\nabla y\nabla y^{\prime}dx.
    \end{array}
\end{equation*}
Then,
\begin{equation*}
    \begin{array}{l}
         \displaystyle\lim_{\sigma\to 0}\left(\Vert\nabla (y+\sigma y^{\prime})\Vert^{2}_{L^{p}}-\Vert\nabla y\Vert^{2}_{L^{p}}\right)\Delta y^{\prime} \\
         = \displaystyle\lim_{\sigma\to 0}\left(\Vert\nabla y\Vert^{2}_{L^{p}} + 2 \sigma \Vert\nabla y\Vert_{L^{p}}^{2-p}\displaystyle\int_{\Omega}\vert\nabla y\vert^{p-2}\nabla y\nabla y^{\prime}dx - \Vert\nabla y\Vert^{2}_{L^{p}} \right)\Delta y^{\prime}\\
        = \displaystyle\lim_{\sigma\to 0}\, \sigma \left(2\Vert\nabla y\Vert_{L^{p}}^{2-p}\displaystyle\int_{\Omega}\vert\nabla y\vert^{p-2}\nabla y\nabla y^{\prime}dx\right)\Delta y^{\prime} = 0.
    \end{array}
\end{equation*}
Therefore,  by the arguments used in the analysis of $\mathcal{I}_{1}$ in Lemma \ref{F bem definido extra} and by Lebesgue's dominated convergence theorem we obtain $\tilde{H}_{3}\to 0$ as $\sigma \rightarrow 0$. In a similar way, $\tilde{H}_{2}\to 0$ as $\sigma \rightarrow 0$. 

For $j=2$ and $j=3$ the convergence \eqref{Converg de Ii} is prompt.

Finally, let's see that
\begin{equation}\label{Converg de I3}
\begin{array}{l}
    \dfrac{1}{\sigma}[\mathcal{I}_{3}\left((y,P,\theta,v,v_{0})+\sigma(y^{\prime},P^{\prime},\theta^{\prime},v^{\prime},v_{0}^{\prime})\right) -\mathcal{I}_{3}(y,P,\theta,v,v_{0})]\vspace{0.1cm}\\

\hspace{-1.6cm}\, \text{converges to} \,\mathcal{DI}_{3}(y^{\prime},P^{\prime},\theta^{\prime},v^{\prime},v_{0}^{\prime})\, \text{strong in}\, L^{q}(\rho_{3}^{q}(0,T);L^{p}(\Omega)), \, \text{as}\, \sigma\longrightarrow 0.
\end{array}
\end{equation}

Here, for simplicity, we will also omit the notation of norms but make it clear that they are all norms in $L^{q}(\rho^{q}_{3}(0,T);L^{p} (\Omega))$. Therefore,
\begin{equation*}
    \begin{array}{l}
     \Vert  \dfrac{1}{\sigma}[\mathcal{I}_{3}\left((y,P,\theta,v,v_{0})+\sigma(y^{\prime},P^{\prime},\theta^{\prime},v^{\prime},v_{0}^{\prime})\right) -\mathcal{I}_{3}(y,P,\theta,v,v_{0})]\vspace{0.1cm} \\
     
\,\,\,  -\mathcal{DI}_{3}(y^{\prime},P^{\prime},\theta^{\prime},v^{\prime},v_{0}^{\prime})\Vert\leq \sigma\Vert y^{\prime}\cdot\nabla\theta^{\prime}\Vert

+ \sigma\Vert(\nu_{0}+\nu_{1}\Vert\nabla(y+\sigma y^{\prime})\Vert_{L^{p}}^{2})Dy^{\prime}:\nabla y^{\prime}\Vert\vspace{0.1cm}\\

\,\,\, +\Vert\dfrac{\nu_{1}}{\sigma}\left( \Vert\nabla(\theta+\sigma \theta^{\prime})\Vert_{L^{p}}^{2}-\Vert\nabla \theta\Vert_{L^{p}}^{2}\right)\Delta\theta - \left(2\nu_{1}\Vert\nabla\theta\Vert^{2-p}_{L^{p}}\displaystyle\int_{\Omega}\vert\nabla \theta\vert^{p-2}\nabla\theta\nabla\theta^{\prime}dx\right)\Delta\theta\Vert\vspace{0.1cm}\\

\,\,\, +\Vert\nu_{1}(\Vert\nabla(\theta+\sigma \theta^{\prime})\Vert_{L^{p}}^{2}-\Vert\nabla \theta\Vert_{L^{p}}^{2})\Delta\theta^{\prime}\Vert

 +\Vert\dfrac{\nu_{1}}{\sigma}\left( \Vert\nabla(y+\sigma y^{\prime})\Vert_{L^{p}}^{2}-\Vert\nabla y\Vert_{L^{p}}^{2}\right)Dy :\nabla y \vspace{0.1cm}\\ 
\,\,- \left(2\nu_{1}\Vert\nabla y\Vert^{2-p}_{L^{p}}\displaystyle\int_{\Omega}\vert\nabla y\vert^{p-2}\nabla y\nabla y^{\prime}dx\right) D y:\nabla y\Vert

 +\Vert\nu_{1}(\Vert\nabla(y+\sigma y^{\prime})\Vert_{L^{p}}^{2}-\Vert\nabla y\Vert_{L^{p}}^{2})Dy:\nabla y^{\prime}\Vert\vspace{0.1cm}\\

\,\,\,+\Vert\nu_{1}(\Vert\nabla(y+\sigma y^{\prime})\Vert_{L^{p}}^{2}-\Vert\nabla y\Vert_{L^{p}}^{2})Dy^{\prime}:\nabla y\Vert=\displaystyle\sum_{l=1}^{7}\tilde{I}_{l}.
\end{array}
\end{equation*}

By the same arguments from the proof of Lemma \ref{F bem definido extra}  together with Lebesgue's dominated convergence theorem, we have
\[
\displaystyle\sum_{l=1}^{5}\tilde{I}_{l}\longrightarrow 0,\,\,\text{as}\,\, \sigma\longrightarrow 0.
\]
Also, using Hölder's inequality for $p$ and $\frac{p}{p-1}$ and $W^{1,p}(\Omega)\hookrightarrow L^{\infty}(\Omega)$,
\begin{equation}\label{I6 tilde}
\begin{array}{ll}
    \tilde{I}^{q}_{6}\leq C\displaystyle\int_{0}^{T}\left(\displaystyle\int_{\Omega}\rho_{3}^{p}(\Vert\nabla (y+\sigma y^{\prime})\Vert_{L^{p}}^{2}-\Vert\nabla y\Vert^{2}_{L^{p}})^{p}\vert\nabla y\vert^{p}\vert\nabla y^{\prime}\vert^{p}dx\right)^{q/p}dt\vspace{0.1cm}\\
    
\leq C\displaystyle\int_{0}^{T}\sigma^{q}\left(\Vert\nabla y\Vert_{L^{p}}^{2-p}\displaystyle\int_{\Omega}\vert\nabla y\vert^{p-2}\nabla y\nabla y^{\prime}dx\right)^{q}\left(\displaystyle\int_{\Omega}\rho_{3}^{p}\vert\nabla y\vert^{p}\vert\nabla y^{\prime}\vert^{p}dx\right)^{q/p}dt\vspace{0.1cm}\\

\leq C\,\sigma^{q}\displaystyle\int_{0}^{T}\Vert \nabla y\Vert^{q(2-p)}_{L^{p}(\Omega)}\left(\displaystyle\int_{\Omega}\vert\nabla y\vert^{p-1}\vert\nabla y^{\prime}\vert dx\right)^{q}\left(\displaystyle\int_{\Omega}\rho_{3}^{p}\vert\nabla y\vert^{p}\vert\nabla y^{\prime}\vert^{p}dx\right)^{q/p}dt \vspace{0.1cm}\\

\leq C\sigma^{q}\left(\displaystyle\sup_{[0,T]}\Vert \nabla y\Vert^{q}_{L^{p}(\Omega)}\right)^{\hspace{-0.1cm}(2-p)} \displaystyle\int_{0}^{T}\left(\Vert \nabla y\Vert_{L^{p}(\Omega)}^{q(p-1)}\Vert\nabla y^{\prime}\Vert_{L^{p}(\Omega)}^{q}\right)\left(\displaystyle\int_{\Omega}\rho_{3}^{p}\vert\nabla y\vert^{p}\vert\nabla y^{\prime}\vert^{p}dx\right)^{q/p}\hspace{-0.4cm}dt 
\vspace{0.1cm}\\

\leq C\sigma^{q} \Vert \kappa y \Vert_{L^{\infty}(0,T;W^{1,p}(\Omega))}^{q(2-p) + q(p-1)}\Vert \kappa y^{\prime}\Vert^{q}_{L^{\infty}(0,T;W^{1,p}(\Omega))}\displaystyle\int_{0}^{T}\Vert\kappa\nabla y^{\prime}\Vert^{q}_{L^{\infty}(\Omega)}\left(\displaystyle\int_{\Omega}\vert\kappa\nabla y\vert^{p}dx\right)^{q/p}\hspace{-0.4cm}dt\vspace{0.1cm}\\

\leq C\sigma^{q} \Vert \kappa y \Vert_{L^{\infty}(0,T;W^{1,p}(\Omega))}^{2q}\Vert \kappa y^{\prime}\Vert^{q}_{L^{\infty}(0,T;W^{1,p}(\Omega))}\Vert\kappa y^{\prime}\Vert^{q}_{L^{q}(0,T;W^{2,p}(\Omega))}.
    \end{array}
\end{equation}
From this, we can deduce that $\tilde{I}_{6}\longrightarrow 0$, as $\sigma\longrightarrow 0$. By the same arguments we also have $\tilde{I}_{7}\longrightarrow 0$, as $\sigma\longrightarrow 0$. Consequently, \eqref{Converg de I3} is true.

Then we can conclude that \eqref{Converg de Ii} holds and $\mathcal{I}$ is G\^ateaux-differentiable at any $(y,p,\theta,v,v_{0})\in  \mathcal{U}_{N}$, with \textit{G-derivative} $\mathcal{I}^{\prime}(y,p,\theta,v,v_{0})={\mathcal{DI}}(y,p,\theta,v,v_{0})$.

Now, we will show that $(y,P,\theta,v,v_{0})\longmapsto\mathcal{I}^{\prime}(y,P,\theta,v,v_{0})$ is continuous from $\mathcal{U}_{N}$ into $\mathcal{L}(\mathcal{U}_{N},\mathcal{R}_{N})$ and as consequently, in view of classical results, we will
have that $\mathcal{I}$ is Fr\'echet-differentiable and $\mathcal{C}^{1}$. Thus, suppose that $$(y_m, P_m,\theta_m, v_m,{v_{0}}_{m})\longrightarrow (y, P, \theta, v, v_{0})\,\, \text{in}\,\,  \mathcal{U}_{N}$$ and let us check that
\begin{equation}\label{convergencia da derivada de I}
    \begin{array}{c}
\Vert\left(\mathcal{I}^{\prime}(y_m, P_m,\theta_m, v_m,{v_{0}}_{m})- \mathcal{I}^{\prime}(y, P,\theta, v,v_{0})\right)(y^{\prime},P^{\prime},\theta^{\prime},v^{\prime},v_{0}^{\prime})\Vert_{ \mathcal{R}_{N}}\vspace{0.1cm}\\

\leq\chi_{m}\Vert(y^{\prime},P^{\prime},\theta^{\prime},v^{\prime},v_{0}^{\prime})\Vert_{\mathcal{U}_{N}},
    \end{array}
\end{equation}
for all $(y^{\prime},P^{\prime},\theta^{\prime},v^{\prime},v_{0}^{\prime})\in  \mathcal{U}_{N}$, for some $\displaystyle\lim_{m\to\infty}\chi_{m}=0$.

In order to simplify the notation, we will consider
\begin{equation*}
\mathbb{D}_{j,m}:=\mathcal{I}_{j}^{\prime}(y_m, P_m,\theta_m, v_m,{v_{0}}_{m})-\mathcal{I}_{j}^{\prime}(y, p,\theta, v,v_{0}).
\end{equation*}
So, notice that
\begin{equation}\label{D1,m em Lq}
\begin{array}{l}
\bullet\,\,\,  \Vert\mathbb{D}_{1,m}(y^{\prime},P^{\prime},\theta^{\prime},v^{\prime},v_{0}^{\prime})\Vert_{L^{q}(\rho_{3}^{q}(0,T);L^{p}(\Omega))}\vspace{0.1cm}\\
\leq C\left(\Vert\nu_{1}(\Vert\nabla y\Vert_{L^{p}(\Omega)}^{2}-\Vert\nabla y_{m}\Vert_{L^{p}(\Omega)}^{2} 
 )\Delta y^{\prime}\Vert_{L^{q}(\rho_{3}^{q}(0,T);L^{p}(\Omega))}\right.\vspace{0.1cm}\\
 +\, \Vert \left(2\nu_{1}\Vert\nabla y\Vert_{L^{p}}^{2-p}\displaystyle\int_{\Omega}\vert\nabla y\vert^{p-2}\nabla y\nabla y^{\prime}dx \right)\Delta y\vspace{0.1cm}\\
 - \left(2\nu_{1}\Vert\nabla y_{m}\Vert_{L^{p}}^{2-p}\displaystyle\int_{\Omega}\vert\nabla y_{m}\vert^{p-2}\nabla y_{m}\nabla y^{\prime}dx \right)\Delta y_{m}\Vert_{L^{q}(\rho_{3}^{q}(0,T);L^{p}(\Omega))}\vspace{0.1cm}\\
 \left.+\,\Vert(y^{\prime}\cdot\nabla)(y_{m}-y)+((y_{m}-y)\cdot\nabla)
y^{\prime}\Vert_{L^{q}(\rho_{3}^{q}(0,T);L^{p}(\Omega))}\right)\vspace{0.1cm}\\
=C(\tilde{K}_{1}+\tilde{K}_{2}+\tilde{K}_{3}).
 \end{array}
\end{equation}
Since,
\begin{equation*}
    \begin{array}{c}
    \tilde{K}_{1}\leq C\left(\Vert\Vert\nabla (y-y_{m})\Vert_{L^{p}} \Vert\nabla y \Vert_{L^{p}}\Delta y^{\prime}\Vert_{L^{q}(\rho_{3}^{q}(0,T);L^{p}(\Omega))}\right.\\ 
   \left. +
     \Vert\Vert\nabla (y-y_{m})\Vert_{L^{p}}\Vert\nabla y_{m}\Vert_{L^{p}}\Delta y^{\prime}\Vert_{L^{q}(\rho_{3}^{q}(0,T);L^{p}(\Omega))}\right)
\end{array}
\end{equation*}
then, using the same arguments as \eqref{gradiente limitado em Lp}, we conclude that
\begin{equation*}
    \begin{array}{c}
  \tilde{K}_{1}\leq \chi_{1,m}\Vert(y^{\prime},P^{\prime},\theta^{\prime},v^{\prime},v_{0}^{\prime})\Vert_{ \mathcal{U}_{N}}     
    \end{array}
\end{equation*}
where
\begin{equation*}
\begin{array}{c}
\chi_{1,m}= C\Vert(y_m, P_m,\theta_m, v_m,{v_{0}}_{m})-(y, P,\theta, v,{v_{0}})\Vert_{\mathcal{U}_{N}}\left( \Vert(y, P,\theta, v,{v_{0}})\Vert_{\mathcal{U}_{N}}\right.\vspace{0.1cm}\\
\left.+\Vert(y_m, P_m,\theta_m, v_m,{v_{0}}_{m})\Vert_{\mathcal{U}_{N}}\right).
\end{array}
\end{equation*}

Now, adding and subtracing $\left(2\nu_{1}\Vert\nabla y\Vert^{2-p}_{L^{p}}\int_{\Omega}\vert\nabla y\vert^{p-2}\nabla y\nabla y^{\prime}dx\right)\Delta y_{m}$ in $\tilde{K}_{2}$, we have 
\begin{equation*}
    \begin{array}{ll}
\tilde{K}_{2} \leq \Vert\left(2\nu_{1}\Vert\nabla y\Vert^{2-p}_{L^{p}}\displaystyle\int_{\Omega}\vert\nabla y\vert^{p-2}\nabla y\nabla y^{\prime}dx\right)\Delta (y_{m}-y)\Vert\vspace{0.1cm}\\

+ \Vert 2\nu_{1}\left(\Vert\nabla y_{m}\Vert^{2-p}_{L^{p}}\displaystyle\int_{\Omega}\vert\nabla y_{m}\vert^{p-2}\nabla y_{m}\nabla y^{\prime}dx - \Vert\nabla y\Vert^{2-p}_{L^{p}}\displaystyle\int_{\Omega}\vert\nabla y\vert^{p-2}\nabla y\nabla y^{\prime} dx\right)\Delta y_{m}\Vert\\

= \tilde{K}_{2,1} + \tilde{K}_{2,2}.
    \end{array}
\end{equation*}
Using arguments similar to those applied in \eqref{I6 tilde}, we have
\begin{equation}\label{K2,1}
    \begin{array}{lll}
(\tilde{K}_{2,1})^{q} &\leq & C\left(\displaystyle\sup_{[0,T]}\Vert \nabla y\Vert^{q}_{L^{p}(\Omega)}\right)^{\hspace{-0.1cm}(2-p)} \hspace{-0.3cm}\displaystyle\int_{0}^{T}\hspace{-0.1cm}\left(\Vert \nabla y\Vert_{L^{p}(\Omega)}^{q(p-1)}\Vert\nabla y^{\prime}\Vert_{L^{p}(\Omega)}^{q}\right)\left(\displaystyle\int_{\Omega}\rho_{3}^{p}\vert\Delta(y_{m} -  y)\vert^{p}dx\right)^{q/p}\hspace{-0.4cm}dt 
\vspace{0.1cm}\\
& \leq & C \Vert \kappa y \Vert_{L^{\infty}(0,T;W^{1,p}(\Omega))}^{q}\Vert \kappa y^{\prime}\Vert^{q}_{L^{\infty}(0,T;W^{1,p}(\Omega))}\Vert \kappa(y_{m}-y)\Vert^{q}_{L^{q}(0,T;W^{2,p}(\Omega))}\vspace{0.1cm}\\
 &\leq & C \left(\Vert(y_m, P_m,\theta_m, v_m,{v_{0}}_{m})-(y, P,\theta, v,{v_{0}})\Vert^{q}_{\mathcal{U}_{N}} \Vert(y^{\prime}, P^{\prime},\theta^{\prime}, v^{\prime},{v_{0}^{\prime}})\Vert^{q}_{\mathcal{U}_{N}}\right.\vspace{0.1cm}\\
&\quad{}&\left.\Vert(y, P,\theta, v,{v_{0}})\Vert^{q}_{\mathcal{U}_{N}}\right).
    \end{array}
\end{equation}
And, adding and subtracing $2\nu_{1}\int_{\Omega}\frac{\vert\nabla y\vert^{p-2}}{\Vert\nabla y\Vert^{p-2}_{L^{p}}}\nabla y_{m}\nabla y^{\prime}dx\Delta y_{m}$ in $\tilde{K}_{2,2}$, we get
\begin{equation}\label{K2,2 tilde}
    \begin{array}{lll}
\tilde{K}_{2,2} &=& \Vert 2\nu_{1}\left[\displaystyle\int_{\Omega}\left(\dfrac{\vert\nabla y_{m}\vert^{p-2}}{\Vert\nabla y_{m}\Vert^{p-2}_{L^{p}}}-\dfrac{\vert\nabla y\vert^{p-2}}{\Vert\nabla y\Vert^{p-2}_{L^{p}}}\right)\nabla y_{m}\nabla y^{\prime}dx\right]\Delta y_{m}\Vert\vspace{0.1cm}\\

&\quad{} & +\,\Vert 2\nu_{1}\left(\displaystyle\int_{\Omega}\dfrac{\vert\nabla y\vert^{p-2}}{\Vert\nabla y\Vert^{p-2}_{L^{p}}}\nabla(y_{m}-y)\nabla y^{\prime}dx\right)\Delta y_{m} \Vert\vspace{0.1cm}\\

& = & \tilde{K}_{2,2}^{1} + \tilde{K}_{2,2}^{2}.
    \end{array}
\end{equation}
Let's analyze the integrals of \eqref{K2,2 tilde} separately. First, denote by $z_{m}=\frac{\vert\nabla y_{m}\vert}{\Vert\nabla y_{m}\Vert_{L^{p}}}$ and $z=\frac{\vert\nabla y\vert}{\Vert\nabla y\Vert_{L^{p}}}$. Applying in order Holder's inequality for $\frac{p-2}{p}+\frac{1}{p}+\frac{1}{p}=1$, the Mean Value Theorem and again Hölder for $\frac{1}{p-2}+\frac{p-3}{p-2}=1$, we obtain
\begin{equation*}
    \begin{array}{l}
\bullet \displaystyle\int_{\Omega}(z_{m}^{p-2}-z^{p-2})\nabla y_{m}\nabla y^{\prime}dx\vspace{0.1cm}\\

\leq \left(\displaystyle\int_{\Omega}(z_{m}^{p-2}-z^{p-2})^{p/(p-2)}dx\right)^{(p-2)/p}\Vert\nabla y_{m}\Vert_{L^{p}}\Vert\nabla y^{\prime}\Vert_{L^{p}}\vspace{0.1cm}\\

\leq \left[\displaystyle\int_{\Omega}\left((p-2)(\vert z\vert + \vert z_{m}\vert)^{p-3}\vert z_{m}-z\vert\right)^{p/(p-2)}dx\right]^{(p-2)/p}\Vert\nabla y_{m}\Vert_{L^{p}}\Vert\nabla y^{\prime}\Vert_{L^{p}}\vspace{0.1cm}\\

\leq C \left(\displaystyle\int_{\Omega}(\vert z\vert + \vert z_{m}\vert)^{\frac{(p-3)p}{p-2}}\vert z_{m}-z\vert^{p/(p-2)}dx\right)^{(p-2)/p}\Vert\nabla y_{m}\Vert_{L^{p}}\Vert\nabla y^{\prime}\Vert_{L^{p}}\vspace{0.1cm}\\

\leq C \left[ \left(\displaystyle\int_{\Omega}(\vert z\vert + \vert z_{m}\vert)^{p}dx\right)^{\frac{(p-3)}{p-2}}\left(\displaystyle\int_{\Omega}\vert z_{m}-z\vert^{p}dx\right)^{1/(p-2)}\right]^{(p-2)/p} \Vert\nabla y_{m}\Vert_{L^{p}}\Vert\nabla y^{\prime}\Vert_{L^{p}}\vspace{0.1cm}\\

\leq C \left(\displaystyle\int_{\Omega}(\vert z\vert + \vert z_{m}\vert)^{p}dx\right)^{\frac{(p-3)}{p}}\left(\displaystyle\int_{\Omega}\vert z_{m}-z\vert^{p}dx\right)^{1/p}\Vert\nabla y_{m}\Vert_{L^{p}}\Vert\nabla y^{\prime}\Vert_{L^{p}}\vspace{0.1cm}\\

\leq C (\Vert z\Vert_{L^{p}}+\Vert z_{m}\Vert_{L{p}})^{p-3}\Vert z_{m}-z \Vert_{L^{p}}\Vert\nabla y_{m}\Vert_{L^{p}}\Vert\nabla y^{\prime}\Vert_{L^{p}}\vspace{0.1cm}\\

\leq C \, 2^{p-3} \Vert z_{m}-z \Vert_{L^{p}}\Vert\nabla y_{m}\Vert_{L^{p}}\Vert\nabla y^{\prime}\Vert_{L^{p}}\vspace{0.1cm}\\

\leq C\left(\displaystyle\int_{\Omega}\left\vert \dfrac{\vert\nabla y_{m}\vert}{\Vert\nabla y_{m}\Vert_{L^{p}}} - \dfrac{\vert\nabla y\vert}{\Vert\nabla y_{m}\Vert_{L^{p}}} + \dfrac{\vert\nabla y\vert}{\Vert\nabla y_{m}\Vert_{L^{p}}} - \dfrac{\vert\nabla y\vert}{\Vert\nabla y\Vert_{L^{p}}} \right\vert^{p}dx\right)^{1/p} \Vert\nabla y_{m}\Vert_{L^{p}}\Vert\nabla y^{\prime}\Vert_{L^{p}}\vspace{0.1cm}\\

\leq C \left[\displaystyle\int_{\Omega}\left(\left\vert\dfrac{\nabla (y_{m}-y)}{\Vert\nabla y_{m}\Vert_{L^{p}}}\right\vert + \left\vert\dfrac{\vert\nabla y\vert(\Vert\nabla y\Vert_{L^{p}}-\Vert\nabla y_{m}\Vert_{L^{p}})}{\Vert\nabla y_{m}\Vert_{L^{p}}\Vert\nabla y\Vert_{L^{p}}}\right\vert\right)^{p}dx \right]^{1/p}\Vert\nabla y_{m}\Vert_{L^{p}}\Vert\nabla y^{\prime}\Vert_{L^{p}}\vspace{0.1cm}\\

\leq C \left(\dfrac{\Vert\nabla (y_{m}-y)\Vert_{L^{p}}}{\Vert\nabla y_{m}\Vert_{L^{p}}} + \dfrac{\Vert\nabla y\Vert_{L^{p}}\Vert\nabla (y_{m}-y)\Vert_{L^{p}}}{\Vert\nabla y_{m}\Vert_{L^{p}}\Vert\nabla y\Vert_{L^{p}}}\right)\Vert\nabla y_{m}\Vert_{L^{p}}\Vert\nabla y^{\prime}\Vert_{L^{p}}\vspace{0.1cm}\\

\leq C \Vert\nabla (y_{m}-y)\Vert_{L^{p}}\Vert\nabla y^{\prime}\Vert_{L^{p}}.
    \end{array}
\end{equation*}
Therefore,
\begin{equation}\label{primeiro de K2,2 tilde}
    \begin{array}{l}
         (\tilde{K}_{2,2}^{1})^{q}\leq C \displaystyle\int_{0}^{T}\Vert\nabla(y_{m}-y)\Vert^{q}_{L^{p}}\Vert\nabla y^{\prime}\Vert^{q}_{L^{p}}\left(\displaystyle\int_{\Omega}\rho_{3}^{p}\vert\Delta y_{m}\vert^{p}dx\right)^{q/p}\vspace{0.1cm}\\
         \leq C \Vert\kappa(y_{m}-y)\Vert^{q}_{L^{\infty}(0,T;W^{1,p}(\Omega))}\Vert\kappa y^{\prime}\Vert^{q}_{L^{\infty}(0,T;W^{1,p}(\Omega))}\Vert\kappa y_{m}\Vert^{q}_{L^{q}(0,T;W^{2,p}(\Omega))}\vspace{0.1cm}\\
         \leq C \left(\Vert(y_m, P_m,\theta_m, v_m,{v_{0}}_{m})-(y, P,\theta, v,{v_{0}})\Vert^{q}_{\mathcal{U}_{N}} \Vert(y^{\prime}, P^{\prime},\theta^{\prime}, v^{\prime},{v_{0}^{\prime}})\Vert^{q}_{\mathcal{U}_{N}}\right.\vspace{0.1cm}\\
\left.\Vert(y_{m}, P_{m},\theta_{m}, v_{m},{v_{0m}})\Vert^{q}_{\mathcal{U}_{N}}\right).
    \end{array}
\end{equation}

And, again using Holder's inequality for $\frac{p-2}{p}+\frac{1}{p}+\frac{1}{p}=1$
\begin{equation*}
    \begin{array}{lll}
      \bullet \displaystyle\int_{\Omega}\dfrac{\vert\nabla y\vert^{p-2}}{\Vert\nabla y\Vert^{p-2}_{L^{p}}}\nabla(y_{m}-y)\nabla y^{\prime}dx
      &\leq& C \left(\displaystyle\int_{\Omega}\dfrac{\vert\nabla y\vert_{p}}{\Vert\nabla y\Vert^{p}_{L^{p}}}dx\right)^{(p-2)/p}\Vert\nabla (y_{m}-y)\Vert_{L^{p}}\Vert\nabla y^{\prime}\Vert_{L^{p}}\vspace{0.1cm}\\
      &\leq &C \Vert\nabla (y_{m}-y)\Vert_{L^{p}}\Vert\nabla y^{\prime}\Vert_{L^{p}}.
    \end{array}
\end{equation*}
Then,
\begin{equation}\label{segundo de K2,2 tilde}
    \begin{array}{l}
         (\tilde{K}_{2,2}^{2})^{q}
         \leq C \left(\Vert(y_m, P_m,\theta_m, v_m,{v_{0}}_{m})-(y, P,\theta, v,{v_{0}})\Vert^{q}_{\mathcal{U}_{N}} \Vert(y^{\prime}, P^{\prime},\theta^{\prime}, v^{\prime},{v_{0}^{\prime}})\Vert^{q}_{\mathcal{U}_{N}}\right.\vspace{0.1cm}\\
\left.\Vert(y, P,\theta, v,{v_{0}})\Vert^{q}_{\mathcal{U}_{N}}\right).
    \end{array}
\end{equation}

From \eqref{primeiro de K2,2 tilde} and \eqref{segundo de K2,2 tilde} in \eqref{K2,2 tilde}, we conclude that
\begin{equation}\label{K2,2 tilde 2}
\begin{array}{l}
    \tilde{K}_{2,2}\leq C \left(\Vert(y_m, P_m,\theta_m, v_m,{v_{0}}_{m})-(y, P,\theta, v,{v_{0}})\Vert_{\mathcal{U}_{N}} \Vert(y^{\prime}, P^{\prime},\theta^{\prime}, v^{\prime},{v_{0}^{\prime}})\Vert_{\mathcal{U}_{N}}\right.\vspace{0.1cm}\\
\left.\Vert(y_{m}, P_{m},\theta_{m}, v_{m},{v_{0m}})\Vert_{\mathcal{U}_{N}}\right).
\end{array}
\end{equation}
And as a consequence of \eqref{K2,1} and \eqref{K2,2 tilde 2}
\begin{equation*}
    \begin{array}{l}
       \tilde{K}_{2}\leq  \chi_{2,m} \Vert(y^{\prime}, P^{\prime},\theta^{\prime}, v^{\prime},{v_{0}^{\prime}})\Vert_{\mathcal{U}_{N}},
    \end{array}
\end{equation*}
with
$$
\chi_{2,m}  = C\Vert(y_m, P_m,\theta_m, v_m,{v_{0}}_{m})-(y, P,\theta, v,{v_{0}})\Vert_{\mathcal{U}_{N}}\Vert(y_{m}, P_{m},\theta_{m}, v_{m},{v_{0m}})\Vert_{\mathcal{U}_{N}}.
$$
Moreover, by the same reasoning as \eqref{y grad y limitado em Lp},
\begin{equation*}
    \begin{array}{lll}
      \tilde{K}_{3}& \leq& C\left(\Vert(y^{\prime}\cdot\nabla)(y_{m}-y)\Vert_{L^{q}(\rho_{3}^{q}(0,T);L^{p}(\Omega))}  + \Vert((y_{m}-y)\cdot\nabla)
y^{\prime}\Vert_{L^{q}(\rho_{3}^{q}(0,T);L^{p}(\Omega))}\right)   \vspace{0.1cm}\\

&\leq &\chi_{3,m}\Vert(y^{\prime},P^{\prime},\theta^{\prime},v^{\prime},v_{0}^{\prime})\Vert_{\mathcal{U}_{N}},
    \end{array}
\end{equation*}
with
\begin{equation*}
\chi_{3,m}=C\Vert(y_m, P_m,\theta_m, v_m,{v_{0}}_{m})-(y, P,\theta, v,{v_{0}})\Vert_{\mathcal{U}_{N}}.
\end{equation*}

It is easy to check that $\mathbb{D}_{j,m}$ for $j=2$ and $j=4$ satisfy similar inequalities.

Again, all inequality norms below are norms in $L^{q}(\rho_{3}^{q}(0,T);L^{p}(\Omega))$, we will omit them for simplicity. For $\mathbb{D}_{3,m}$ after some manipulations we get the following:
\begin{equation*}
    \begin{array}{l}
\bullet\,\,\, \Vert\mathbb{D}_{3,m}(y^{\prime},P^{\prime},\theta^{\prime},v^{\prime},v_{0}^{\prime})\Vert\leq C\left[ \Vert\, \Vert\nabla(y_{m}-y)\Vert_{L^{p}}\,\Vert\nabla y\Vert_{L^{p}}\, \Delta\theta^{\prime}\,\Vert\right.\vspace{0.1cm}\\ 
+ \Vert\,\Vert\nabla(y_{m}-y)\Vert_{L^{p}}\,\Vert\nabla y_{m}\Vert_{L^{p}}\Delta\theta^{\prime}\,\Vert + \Vert\left(2\nu_{1}\Vert\nabla \theta\Vert^{2-p}_{L^{p}}\int_{\Omega}\vert\nabla \theta\vert^{p-2}\nabla \theta\nabla \theta^{\prime}dx\right)\Delta(\theta_{m}-\theta)\Vert\vspace{0.1cm}\\
   +\Vert 2\nu_{1}\left(\Vert\nabla \theta_{m}\Vert^{2-p}_{L^{p}}\int_{\Omega}\vert\nabla \theta_{m}\vert^{p-2}\nabla \theta_{m}\nabla\theta^{\prime}dx- \Vert\nabla\theta\Vert^{2-p}_{L^{p}}\int_{\Omega}\vert\nabla \theta\vert^{p-2}\nabla\theta\nabla \theta^{\prime}dx\right)\Delta\theta_{m}\Vert \vspace{0.1cm}\\
       + \Vert(\bar{\nu}(\nabla y_{m})D(y_{m}-y):\nabla y^{\prime}\Vert + 
    \Vert\nu_{1}(\Vert\nabla y_{m}\Vert_{L^{p}}^{2}-\Vert\nabla y\Vert_{L^{p}}^{2})Dy:\nabla y^{\prime}\Vert\vspace{0.1cm}\\
    
    + \Vert\bar{\nu}(\nabla y_{m})D y^{\prime}:\nabla (y_{m}-y)\Vert
       \,+ \Vert \nu_{1}(\Vert\nabla y_{m}\Vert_{L^{p}}^{2}-\Vert\nabla y\Vert_{L^{p}}^{2})D y^{\prime}:\nabla y \Vert \vspace{0.1cm}\\
       
       +\Vert \left(2\nu_{1}\Vert\nabla y_{m}\Vert_{L^{p}}^{2-p}\int_{\Omega}\vert\nabla y_{m}\vert^{p-2}\nabla y_{m}\nabla y^{\prime}dx\right)Dy_{m}:\nabla(y_{m}-y)\Vert \vspace{0.1cm}\\

       +\Vert \left(2\nu_{1}\Vert\nabla y_{m}\Vert_{L^{p}}^{2-p}\int_{\Omega}\vert\nabla y_{m}\vert^{p-2}\nabla y_{m}\nabla y^{\prime}dx\right)D(y_{m}-y):\nabla y\Vert \vspace{0.1cm}\\

+ \Vert2\nu_{1}\left(\Vert\nabla y_{m}\Vert_{L^{p}}^{2-p}\int_{\Omega}\vert\nabla y_{m}\vert^{p-2}\nabla y_{m}\nabla y^{\prime}dx - \Vert\nabla y\Vert^{2-p}_{L^{p}}\int_{\Omega}\vert\nabla y\vert^{p-2}\nabla y\nabla y^{\prime}dx \right)Dy:\nabla y\Vert\vspace{0.1cm}\\
       
      \,\left. +\Vert y^{\prime}\cdot\nabla (\theta_{m}-\theta)\Vert+\Vert(y_{m}-y)\cdot\nabla\theta^{\prime}\Vert\right]  = C\displaystyle\sum_{s=4}^{16}\tilde{K}_{s}. 
       \end{array}
\end{equation*}
Applying arguments similar to those used in Lemma \ref{F bem definido extra} and in \eqref{D1,m em Lq} we can conclude that $\tilde{K}_{s}\leq \chi_{s,m}$ for $k = \lbrace 4, 5, \ldots, 16\rbrace$. Indeed, let us evaluate $\tilde{K}_{14}$, from the calculations performed for \eqref{K2,2 tilde} we have
\begin{equation*}
    \begin{array}{lll}
(\tilde{K}_{14})^{q} \leq C \displaystyle\int_{0}^{T}\Vert\nabla(y_{m}-y)\Vert^{q}_{L^{p}}\Vert\nabla y^{\prime}\Vert^{q}_{L^{p}}\left(\displaystyle\int_{\Omega}\rho_{3}^{p}\vert\nabla y\vert^{2p}dx\right)^{q/p}\vspace{0.1cm}\\
         \leq C \Vert\kappa(y_{m}-y)\Vert^{q}_{L^{\infty}(0,T;W^{1,p}(\Omega))}\Vert\kappa y^{\prime}\Vert^{q}_{L^{\infty}(0,T;W^{1,p}(\Omega))}\Vert\kappa y\Vert^{q}_{L^{\infty}(0,T;W^{1,p}(\Omega))}\Vert\kappa y\Vert^{q}_{L^{q}(0,T;W^{2,p}(\Omega))}\vspace{0.1cm}\\
         \leq C \left(\Vert(y_m, P_m,\theta_m, v_m,{v_{0}}_{m})-(y, P,\theta, v,{v_{0}})\Vert^{q}_{\mathcal{U}_{N}} \Vert(y^{\prime}, P^{\prime},\theta^{\prime}, v^{\prime},{v_{0}^{\prime}})\Vert^{q}_{\mathcal{U}_{N}}\right.
\left.\Vert(y, P,\theta, v,{v_{0}})\Vert^{2q}_{\mathcal{U}_{N}}\right).
    \end{array}
\end{equation*}
Thus 
\begin{equation*}
    \tilde{K}_{14}\leq \chi_{14,m}\Vert(y^{\prime}, P^{\prime},\theta^{\prime}, v^{\prime},{v_{0}^{\prime}})\Vert_{\mathcal{U}_{N}},
\end{equation*}
with 
$$
\chi_{14,m} = C \Vert(y_m, P_m,\theta_m, v_m,{v_{0}}_{m})-(y, P,\theta, v,{v_{0}})\Vert_{\mathcal{U}_{N}} \Vert(y, P,\theta, v,{v_{0}})\Vert^{2}_{\mathcal{U}_{N}}.
$$

Thus, we have $\lim\limits_{m\to\infty}\chi_{s,m}=0$ for all $s\in\lbrace 1,...,16\rbrace$ and consequently \eqref{convergencia da derivada de I} is obtained. This ends the proof.
\end{proof}

\begin{lem}\label{Mapa sobrejetivo extra}
Let ${\mathcal{I}}$ be the mapping in (\ref{Mapa F extra})-(\ref{Mapas Fi extras}). Then, ${\mathcal{I}}^{\prime}(0,0,0,0,0)$ is onto.
\end{lem}
\begin{proof}
 Let \((F_{1}, y^{0}, F_{2}, \theta^{0}) \in \mathcal{R}_{N}\). Then, the result follows directly from Proposition \ref{proposição controle nulo do sistema linear para W{1,p}}.

\end{proof}

\noindent\textbf{Proof of Theorem $\ref{Teo extra control nulo}$} According to Lemmas \ref{F bem definido extra}--\ref{Mapa sobrejetivo extra}, we can apply the Inverse Mapping Theorem (Theorem \ref{Liusternik}), then, there exists $\delta > 0$ and a mapping $W:B_{\delta}(0)\subset {\mathcal{R}}_{N}\rightarrow {\mathcal{U}}_{N}$ such that
\begin{equation*}
    W(z)\in B_{r}(0)\,\,\, \text{and}\,\,\, {\mathcal{I}}(W(z))=z, \,\,\, \forall z\in B_{\delta}(0).
\end{equation*}
Taking $(0,y^{0},0,\theta^{0})\in B_{\delta}(0)$ and $(y,P,\theta,v,v_{0})=W(0,y^{0},0,\theta^{0})\in {\mathcal{U}}_{N}$, we have
\begin{equation*}
    {\mathcal{I}}(y,P,\theta,v,v_{0})=(0,y^{0},0,\theta^{0}).
\end{equation*}
Thus, we conclude that \eqref{lad.boussinesq caso extra} is locally null controllable at time $T > 0$.

\section{Large time null-controllability}\label{Sec.4}

This section is dedicated to the proof of Theorem \ref{Teo large time}. Following the same approach as \cite{CarvalhoDemarqueLimaco, Kevin}, we first analyze the uncontrolled evolution of the system \eqref{lad.boussinesq} and establish the asymptotic behavior of its solutions as $t \to \infty$, specifically in the case $N = 2$. In particular, we examine the energy decay associated with the solutions of the complete Ladyzhenskaya-Boussinesq system. Based on this long-time analysis, we identify a time $T^* > 0$ such that the states $y(\cdot,T^*)$ and $\theta(\cdot,T^*)$ can be employed as initial data in the study of the null local controllability of \eqref{lad.boussinesq} (as stated in Theorem \ref{Teo principal control nulo}). Consequently, using Theorem \ref{Teo principal control nulo}, we obtain control functions $v$ and $v_0$ that steer the system’s solutions to zero within a sufficiently large time interval.

Accordingly we state the following lemma, which will be fundamental for the demonstration of Theorem \ref{Teo large time}.
\begin{lem}\label{energy decay}
For $N=2$, any $T>0$ and $(y^{0},\theta^{0})\in V\times H^{1}_{0}(\Omega)$, if there is positive constant $r>0$ such that \begin{equation*}
    \Vert (y^{0},\theta^{0})\Vert_{V\times H_{0}^{1}(\Omega)} < r
\end{equation*}
and $(y,p,\theta)$ is a solution of \eqref{lad.boussinesq} with $v\equiv v_{0}\equiv 0$, so this solution  has asymptotic behavior as $t\to\infty$. More precisely, for $${E}(t):=\Vert\nabla y(.,t)\Vert^{2} + \Vert\theta(.,t)\Vert^{2} + \Vert\nabla\theta(.,t)\Vert^{2} $$ there are positive constants $\mathrm{C}_{1}, \mathrm{C}_{2}$ such that
\begin{equation}\label{decaimento}
\begin{array}{c}
E(t)\leq  \mathrm{C}_{2}\, e^{-\mathrm{C}_{1}t}E(0) \,\text{a.e in}\, (0,T).
\end{array}
\end{equation}
\end{lem}

For the convenience of readers, we will give the proof for inequality \eqref{decaimento} in Lemma \ref{energy decay} in Appendix \ref{Appendix}.

\textbf{Proof of Theorem \ref{Teo large time}.}
First, let's fix $T_{0}>0$. Applying the Theorem \ref{Teo principal control nulo} there exists $\delta >0$ such that the system \eqref{lad.boussinesq}, with any initial data $(\bar{y}^{0},\bar{\theta}^{0})\in V\times W_{0}^{1,3/2}(\Omega)$ satisfying $\Vert(\bar{y}^{0},\bar{\theta}^{0})\Vert_{V\times W^{1,3/2}_{0}(\Omega)}<\delta$, is locally null controllable at $T_{0}$.

Determine  $(y^{0},\theta^{0})\in V\times H^{1}_{0}(\Omega)$ and consider $r>0$ as defined in the statement of Lemma \ref{energy decay}. Let then $T^{\ast}$ be a positive time satisfying 
\begin{equation}\label{T ast}
T^{\ast}>\dfrac{-1}{C_{1}}\ln\left(\dfrac{\delta}{\mathrm{C}_{2}(\Vert\nabla y^{0}\Vert^{2} + \Vert\theta^{0}\Vert^{2} + \Vert\nabla\theta^{0}\Vert^{2})}\right)
\end{equation}
 and consider a solution $(y,p,\theta)$ of the system \eqref{lad.boussinesq}, with $T=T^{\ast} + T_{0}$, $v\equiv v_{0}\equiv 0$ and $(y^{0},\theta^{0})$ as the initial data.

From $\eqref{decaimento}$ and \eqref{T ast}, $y(.,T^{\ast}),\, \theta(.,T^{\ast})$ are such that
\[
\Vert\left(y(.,T^{\ast}),\theta(.,T^{\ast})\right)\Vert_{V\times W^{1,3/2}_{0}(\Omega)}\leq \mathrm{C}_{2}\, e^{-\mathrm{C}_{1}T^{\ast}}(\Vert\nabla y^{0}\Vert^{2} + \Vert\theta^{0}\Vert^{2} + \Vert\nabla\theta^{0}\Vert^{2}) <\delta.
\]
Consequently, by Theorem \ref{Teo principal control nulo},  \eqref{lad.boussinesq} is null controllable at $T^{\ast}+T_{0}$.

\section{Some additional comments and  open problems}\label{Comentarios Adicionais}
In this section, we provide some remarks on the systems studied and also highlight several problems within the addressed context that, to the best of our knowledge, remain open.

\begin{itemize}
    \item Initially, note that \eqref{lad.boussinesq caso extra} can be solved with the same techniques by taking $\bar{\nu}(\nabla\varsigma):= \nu_{0}+\nu_{1}\Vert\nabla\varsigma\Vert^{2}_{L^{2}}$.
    \item Furthermore, for our systems \eqref{lad.boussinesq} and \eqref{lad.boussinesq caso extra} it is also possible to obtain the local null controllability with control at the border $\Gamma_{0}\times (0,T)$, where $\Gamma_{0}\subset\partial\Omega$. Indeed, just construct a domain $\hat{\Omega}$ with boundary $\partial\hat{\Omega}$ sufficiently regular via a subset $\mathrm{U}$ of $\mathbb{R}^{N}$ such that $\hat{\Omega}=\Omega\cup \mathrm{U}$ and $\bar{\mathrm{U}}\cap (\overline{\partial\Omega-\Gamma_{0}})=\emptyset$. So, taking $\omega\subset\hat{\Omega}-\bar{\Omega}$ and keeping in mind the controllability result for distributed controls, the control at the boundary is obtained by considering the constraint trace in $\Omega\times (0,T)$ of the state of the distributed control system. That is, since $z(x,t)$ is the solution in $\hat{\Omega}\times (0,T)$ of the distributed control system then 
\begin{equation*}
u=\gamma(z\mid_{\Omega\times (0,T)}) = \left\{\begin{array}{lll}
     \gamma(z) &\text{in}& \Gamma_{0}, \\
     0 &\text{in}& \partial\Omega - \Gamma_{0} 
\end{array}\right.
\end{equation*}
is the control on the desired boundary, where $\gamma: H^{1}(\Omega)\longrightarrow H^{1/2}(\partial\Omega)$.
\end{itemize}

Now, we comment on some open questions that arise naturally in the context of our results.
\begin{enumerate}
    \item Is it possible the local exact controllability to the trajectories for the systems \eqref{lad.boussinesq} and \eqref{lad.boussinesq caso extra}? The main difficulty for this problem is finding a suitable Carleman estimate.
    
    \item Is it possible the local null controllability to \eqref{lad.boussinesq} when $N\geq 4$? This is a very difficult question, because in the proof of Lemma \ref{F bem definido} we use the immersion $H^{2}(\Omega)\hookrightarrow L^{\infty}(\Omega)$ and this is only valid when $N\leq3$.
    \item  Finally, can we deduce the null controllability of \eqref{lad.boussinesq} and \eqref{lad.boussinesq caso extra} in $N$ dimensions, with
$N - 1$ controls? 
\end{enumerate}

\vspace{0.1cm}

\textbf{Acknowledgments}
\noindent This study was financed in part by the Coordena\c{c}\~ao de Aperfei\c{c}oamento de Pessoal de N\'ivel Superior-Brasil (CAPES) and the second author also has support from CNPQ-Brasil.

\section*{Declarations}
\textbf{Conflict of Interest.} The corresponding author hereby affirms, on behalf of all authors involved in this work, that there are no conflicts of interest to disclose.\\

\noindent\textbf{Data Availability Statement (DAS).} No datasets were generated or analysed during the current study. 


\begin{appendices}
\renewcommand{\theequation}{\thesection.\arabic{equation}}
\section{Existence and uniqueness of solution for (\ref{lad.boussinesq})}\label{Appendix}

The following theorem will show the existence and uniqueness of strong solutions for (\ref{lad.boussinesq}).
\begin{theorem}\label{existence and uniq.}
    There exists $R>0$ such that if, 
\begin{equation*}
    \Vert v\Vert^{2}_{L^{2}(\omega\times (0,T))^{N}}+\Vert v_{0}\Vert^{2}_{L^{2}(\omega\times(0,T))}+\Vert y^{0}\Vert_{V}+\Vert\theta^{0}\Vert_{W_{0}^{1,3/2}(\Omega)} < R.
\end{equation*}
then there exists a unique $(y,p,\theta)$ strong solution of (\ref{lad.boussinesq}) in the class
\begin{equation*}
    \left\{
    \begin{array}{l}
     y\in L^{2}(0,T;H^{2}(\Omega)^{N}\cap V)\cap C^{0}([0,T];V),\, \, y_{t}\in L^{2}(0,T;H)    \vspace{0.1cm}  \\
     \theta\in L^{2}(0,T;W^{2,3/2}(\Omega)),\,\, \theta_{t}\in L^{2}(0,T;L^{3/2}(\Omega)).     
    \end{array}\right.
\end{equation*}
\end{theorem}
\begin{proof}
\textit{\textbf{Existence.}}
    We will apply Faedo-Galerkin method to obtain the proof, be orthonormal eigenfunctions of the Stokes operator, i.e, the solutions to 
    \begin{equation*}
    \left\{\begin{array}{lll}
         -\Delta u_{m}+\nabla P_{m} = \lambda_{m}u_{m}& \text{in} & \Omega,  \\
        u_{m}=0 & \text{on} & \partial\Omega,  
    \end{array}\right.    
    \end{equation*}
and $\lbrace w_{m}\rbrace_{m\in\mathrm{N}}$ the basis formed by the eigenfunctions of the Dirichlet Laplacian in $\Omega$. Consider, for $m\in\mathbb{N}$, $U_{m}= span\lbrace u_{1},u_{2},\ldots,u_{m}\rbrace$ and $V_{m}=span\lbrace w_{1},w_{2},\ldots,w_{m}\rbrace$. Let us introduce the finite dimensional Galerkin approximations as follows: find $y_{m}$, $\theta_{m}$, with $y_{m}(t)\in U_{m}$ and $\theta_{m}(t)\in V_{m}$ for all $t$, associated with the initial data $(y^{0},\theta^{0})$, such that
\begin{equation}\label{prob lad.smag. aproximado}
   \left\{ \begin{array}{l}
      (y^{\prime}_{m},u)+((\nu_{0}+\nu_{1}\Vert\nabla y_{m}\Vert^{2})\nabla y_{m},\nabla u)+((y_{m}\cdot\nabla)y_{m},u)=(\nu_{0}\theta_{m}e_{N},u)\vspace{0.1cm}\\
      +(v\tilde{1}_{\omega},u),\forall\, u\in U_{m},\vspace{0.1cm}    \\
       (\theta_{m}^{\prime},w)+((\nu_{0}+\nu_{1}\Vert\nabla y_{m}\Vert^{2})\nabla \theta_{m},\nabla w) + (y_{m}\cdot\nabla \theta_{m},w)=(v_{0}\tilde{1}_{\omega},w)\vspace{0.1cm}\\
       +((\nu_{0}+\nu_{1}\Vert\nabla y_{m}\Vert^{2})D y_{m}:\nabla y_{m},w), \forall\, w\in V_{m}, \vspace{0.1cm}\\
       y_{m}(0)=y^{0}_{m}\rightarrow y^{0}\,\,\, \text{in}\,\,\, V,\,\,\,\theta_{m}(0)=\theta^{0}_{m}\rightarrow\theta^{0}\,\,\, \text{in}\,\,\, L^{2}(\Omega). 
    \end{array}\right.
\end{equation}
The classical ODE theory gives us the existence and uniqueness of a solution for (\ref{prob lad.smag. aproximado}), in local time. By means of the uniform estimates that we will obtain below, we will be able to define such solutions for all time $t$.

\noindent\textbf{Estimate I:} Multiplying the first row of \eqref{prob lad.smag. aproximado} by $\lambda_{1}$, taking $u=-\Delta y_{m}(t)$ and $w=\theta_{m}(t)$ in the first and second equation of \eqref{prob lad.smag. aproximado} and knowing that $\Vert .\Vert_{L^{3}(\Omega}\leq C\Vert .\Vert^{1/2}\Vert .\Vert_{H^{1}(\Omega)}^{1/2}$, we have
\begin{equation}\label{estimate 1 N=2,3}
    \begin{array}{l}
   \dfrac{1}{2}\dfrac{d}{dt}(\lambda_{1}\Vert\nabla y_{m}\Vert^{2} + \Vert\theta_{m}\Vert^{2}) + \dfrac{\nu_{0}}{4}\Vert\nabla\theta_{m}\Vert^{2}   +\nu_{1}\Vert\nabla y_{m}\Vert^{2}\Vert\nabla\theta_{m}\Vert^{2} 
   \vspace{0.1cm}\\
   + \lambda_{1}\nu_{1}\Vert\nabla y_{m}\Vert^{2}\Vert\Delta y_{m}\Vert^{2} 
   + \dfrac{\lambda_{1}\nu_{0}}{8} \Vert\Delta y_{m}\Vert^{2} + \left[\dfrac{\lambda_{1}\nu_{0}}{8} - \hat{C}_{3}(\Omega,\lambda_{1})\Vert\nabla y_{m}\Vert\right.\vspace{0.1cm}\\
   \left. - \hat{C}_{4}(\Omega,\nu_{0},\lambda_{1})\Vert\nabla y_{m}\Vert^{2} -\hat{C}_{5}(\Omega,\nu_{0},\nu_{1},\lambda_{1})\Vert\nabla y_{m}\Vert^{6}\right] \Vert\Delta y_{m}\Vert^{2}\vspace{0.1cm}\\
   \leq   \hat{C}_{2}(\Omega,\nu_{0})\Vert v_{0}\Vert^{2}_{L^{2}(\omega)}
  + \hat{C}_{1}(\nu_{0},\lambda_{1})\Vert v\Vert^{2}_{L^{2}(\omega)^{N}},
   \end{array}
\end{equation}
For simplicity of notation, we will omit the dependencies of the constants already known. Hence, the following statement is valid:
\begin{assumption}
\begin{equation}\label{barA(t)}
\begin{array}{l}
    \bar{A}(t)=   \hat{C}_{3}\Vert\nabla y_{m}(t)\Vert +\hat{C}_{4}\Vert\nabla y_{m}(t)\Vert^{2} + \hat{C}_{5}\Vert \nabla y_{m}(t)\Vert^{6}     <\dfrac{\lambda_{1}\nu_{0}}{8},\,\,
    \forall\, t\in[0,T_{m}].
\end{array}
\end{equation}
\end{assumption}
Indeed, assuming by contradiction that (\ref{barA(t)}) is false then there exist ${t}_{1m}$ such that
\begin{equation*}
    \bar{A}(t)<\dfrac{\lambda_{1}\nu_{0}}{8},\,\, \forall\,\, 0\leq t< {t}_{1m}
\end{equation*}
and
\begin{equation}\label{barA(tm)}
    \bar{A}({t}_{1m})=\dfrac{\lambda_{1}\nu_{0}}{8}.
\end{equation}
By hypothesis, there is $\tilde{\rho}_{0}>0$ such that
\begin{equation*}
    \Vert v\Vert^{2}_{L^{2}(\omega\times (0,T))^{N}}+\Vert v_{0}\Vert^{2}_{L^{2}(\omega\times(0,T))}+\Vert y^{0}\Vert_{V}+\Vert\theta^{0}\Vert_{W_{0}^{1,3/2}(\Omega)} < \tilde{\rho}_{0}.
\end{equation*}
Then, we have
\begin{equation}\label{estimate dado inicial N=2,3}
 \left\{\begin{array}{l}
    \hat{C}_{3}\Vert\nabla y^{0}\Vert +\hat{C}_{4}\Vert\nabla y^{0}\Vert^{2} + \hat{C}_{5}\Vert \nabla y^{0}\Vert^{6}     <\dfrac{\lambda_{1}\nu_{0}}{8},\vspace{0.1cm}\\
        \hat{C}_{1}\Vert v\Vert^{2}_{L^{2}(\omega\times (0,T))^{N}} + \hat{C}_{2}\Vert v_{0}\Vert^{2}_{L^{2}(\omega\times (0,T))} + \dfrac{1}{2}(\lambda_{1}\Vert\nabla y^{0}\Vert^{2}+\Vert\theta^{0}\Vert^{2}) \vspace{0.1cm}\\
       <\min\left\lbrace\dfrac{{\lambda_{1}}^{2}\nu_{0}}{48\hat{C}_{4}}, {\lambda_{1}}^{3}\left(\dfrac{\nu_{0}}{24\sqrt{2}\hat{C}_{3}}\right)^{2},{\lambda_{1}}^{4/3}\left(\dfrac{\nu_{0}}{192\hat{C}_{5}}\right)^{1/3}\right\rbrace. 
 \end{array}\right.   
\end{equation}

Integrating (\ref{estimate 1 N=2,3}) from $0$ to ${t}_{1m}$, we obtain
\begin{equation}\label{apos integrar N=2,3}
    \begin{array}{l}
         \dfrac{1}{2}\left(\lambda_{1}\Vert\nabla y_{m}({t}_{1m})\Vert^{2} + \Vert\theta_{m}({t}_{1m})\Vert^{2}\right) + \dfrac{\nu_{0}}{4}\displaystyle\int^{{t}_{1m}}_{0}\Vert\nabla \theta_{m}\Vert^{2}dt + \dfrac{\lambda_{1}\nu_{0}}{8}\displaystyle\int^{{t}_{1m}}_{0}\Vert\Delta y_{m}\Vert^{2}dt\vspace{0.1cm}\\
         \leq \hat{C}_{1}\Vert v\Vert^{2}_{L^{2}(\omega\times (0,T))^{N}} + \hat{C}_{2}\Vert v_{0}\Vert^{2}_{L^{2}(\omega\times (0,T))} + \dfrac{1}{2}(\lambda_{1}\Vert\nabla y^{0}\Vert^{2}+\Vert\theta^{0}\Vert^{2}). 
    \end{array}
\end{equation}
Then from \eqref{estimate dado inicial N=2,3} and \eqref{apos integrar N=2,3} we arrive at $\bar{A}(t_{1m})<\lambda_{1}\nu_{0}/8$, which contradicts \eqref{barA(tm)}. Therefore, \eqref{barA(t)} holds an we obtain that
\begin{equation}\label{Inequality Estimate I}
    \begin{array}{l}
         \dfrac{1}{2}\left(\lambda_{1}\Vert\nabla y_{m}(t)\Vert^{2} + \Vert\theta_{m}(t)\Vert^{2}\right) + \dfrac{\nu_{0}}{4}\displaystyle\int^{t}_{0}\Vert\nabla \theta_{m}(s)\Vert^{2}ds + \dfrac{\lambda_{1}\nu_{0}}{8}\displaystyle\int^{t}_{0}\Vert\Delta y_{m}(s)\Vert^{2}ds\vspace{0.1cm}\\
         \leq \hat{C}_{1}\Vert v\Vert^{2}_{L^{2}(\omega\times (0,T))^{N}} + \hat{C}_{2}\Vert v_{0}\Vert^{2}_{L^{2}(\omega\times (0,T))} + \dfrac{1}{2}(\lambda_{1}\Vert\nabla y^{0}\Vert^{2}+\Vert\theta^{0}\Vert^{2})\,\forall\, t\in[0,T_{m}]. 
    \end{array}
\end{equation}
As the term on the right side of \eqref{Inequality Estimate I}  is independent of $m$, we can extend the solution $(y_{m},\theta_{m})$ to the entire interval $[0,T]$ and in the same way we can estimate \eqref{Inequality Estimate I} for $t\in[0,T]$. More precisely, 
\begin{equation}\label{estimate I N=3}
    \begin{array}{l}
      \Vert y_{m}\Vert^{2}_{L^{\infty}(0,T;V)} + \Vert \theta_{m}\Vert^{2}_{L^{\infty}(0,T;L^{2}(\Omega))} + \Vert y_{m}\Vert^{2}_{L^{2}(0,T;H^{2}(\Omega)^{N}\cap V)} + \Vert \theta_{m}\Vert^{2}_{L^{2}(0,T;H^{1}_{0}(\Omega))}     \vspace{0.1cm}\\
      \leq C(\Vert y^{0}\Vert^{2}_{V} + \Vert\theta^{0}\Vert^{2} + \Vert v\Vert^{2}_{L^{2}(\omega\times (0,T))^{N}} + \Vert v_{0}\Vert^{2}_{L^{2}(\omega\times (0,T))}). 
    \end{array}
\end{equation}

\noindent\textbf{Estimate II:} Taking $u=y^{\prime}_{m}$ in the first equation of (\ref{prob lad.smag. aproximado}), we obtain after some calculations
\begin{equation*}
   \begin{array}{l}
    \dfrac{1}{2}\displaystyle\int^{t}_{0}\Vert y_{t,m}(s)\Vert^{2}ds + \dfrac{\nu_{0}}{2}\Vert\nabla y_{m}(t)\Vert^{2} + \dfrac{\nu_{1}}{4}\Vert\nabla y_{m}(t)\Vert^{4}\vspace{0.1cm}\\
    \leq C\displaystyle\int^{t}_{0}\Vert\Delta y_{m}(s)\Vert^{2}\Vert\nabla y_{m}(s)\Vert^{2}ds + C\displaystyle\int_{0}^{t}\Vert\theta_{m}(s)\Vert^{2}ds + \Vert v\Vert^{2}_{L^{2}(\omega\times(0,T))^{N}}\vspace{0.1cm}\\ 
   \,\,\,\, + \dfrac{\nu_{0}}{2}\Vert\nabla y^{0}\Vert^{2}
    + \dfrac{\nu_{1}}{4}\Vert\nabla y^{0}\Vert^{4},
    \end{array}
\end{equation*}
therefore, using Estimate I and the Gronwall's Lemma, we arrive at
\begin{equation}\label{estimate II}
    \begin{array}{l}
    \Vert y_{t,m}\Vert^{2}_{L^{2}(0,T;H)} + \Vert y_{m}\Vert^{2}_{L^{\infty}(0,T;V)}\vspace{0.1cm}\\
    \,\,\leq C(\Vert y^{0}\Vert^{2}_{V} + \Vert y^{0}\Vert^{4}_{V} + \Vert\theta^{0}\Vert^{2} + \Vert v\Vert^{2}_{L^{2}(\omega\times (0,T))^{N}} + \Vert v_{0}\Vert^{2}_{L^{2}(\omega\times (0,T))} ). 
    \end{array}
\end{equation}

\noindent\textbf{Estimate III:} Since the $\theta_{m}$ are the eigenfunctions of $-\Delta$ in $H^{1}_{0}(\Omega)$, we have from Estimate I,
\begin{equation}\label{estimate III}
    \Vert\theta_{t,m}\Vert^{2}_{L^{2}(0,T;H^{-1}(\Omega))}\leq C(\Vert y^{0}\Vert^{2}_{V} + \Vert\theta^{0}\Vert^{2} + \Vert v\Vert^{2}_{L^{2}(\omega\times (0,T))^{N}} + \Vert v_{0}\Vert^{2}_{L^{2}(\omega\times (0,T))}).
\end{equation}

From estimates \eqref{estimate I N=3}, (\ref{estimate II}) and (\ref{estimate III}) we can extract subsequences of $\lbrace y_{m}\rbrace$ and $\lbrace \theta_{m}\rbrace$ denoted equal, so that taking the limit $m\rightarrow \infty$ in the equation (\ref{prob lad.smag. aproximado}), $y_{m}$ and $\theta_{m}$ converge to a solution (weak) of \eqref{lad.boussinesq}. Indeed, to obtain the a.e. convergence of nonlocal terms, just use the fact that the sequence $y_m$ is pre-compact in $L^{2}(0,T;V)$. This solution must satisfy 
\begin{equation*}
   \left\{
    \begin{array}{l}
     y\in L^{2}(0,T;H^{2}(\Omega)^{N}\cap V)\cap C^{0}([0,T];V),\, \, y_{t}\in L^{2}(0,T;H)    \vspace{0.1cm}  \\
     \theta\in L^{2}(0,T;H^{1}_{0}(\Omega))\cap L^{\infty}(0,T;L^{2}(\Omega)),\,\, \theta_{t}\in L^{2}(0,T;H^{-1}(\Omega)).     
    \end{array}\right.
\end{equation*}

Furthermore, since $\nu(\nabla y)Dy:\nabla y + v_{0}\tilde{1}_{\omega}\in L^{2}(0,T;L^{3/2}(\Omega))$(see \eqref{X4}) and $\theta^{0}\in W^{1,3/2}_{0}(\Omega)$, from $L^{p}-L^{q}$ regularity for parabolic equation (see, \cite{Robert}), we have $\theta$ solution of
\begin{equation*}
    \hspace{-0.1cm}\left\{\hspace{-0.1cm}
    \begin{array}[c]{lll}
    \theta_{t}-\nu(\nabla y)\Delta\theta + y\cdot\nabla\theta = v_{0}\tilde{1}_{\omega} +   \nu(\nabla y)Dy:\nabla y &\text{in}& Q,\\
   \theta(x,t)=0 &\text{on}& \Sigma,\\
   \theta(x,0)=\theta^{0}(x) &\text{in}& \Omega.
    \end{array}\right.
\end{equation*}
in class $ \theta\in L^{2}(0,T;W^{2,3/2}(\Omega)),\,\, \theta_{t}\in L^{2}(0,T;L^{3/2}(\Omega))$.

This yields \eqref{regularidade das sol.}.

\noindent\textit{\textbf{Uniqueness.}}
Let $(u,q,w)=(y^{1},p^{1},\theta^{1})-(y^{2},p^{2},\theta^{2})$, where $(y^{1},p^{1},\theta^{1})$ and $(y^{2},p^{2},\theta^{2})$ are solutions of problem (\ref{lad.boussinesq}). Then, we got
\begin{equation*}
    \hspace{-0.1cm}\left\{\hspace{-0.1cm}
    \begin{array}[c]{lll}
    u_{t}-\nu_{0}\Delta u - \nu_{1}\Vert\nabla y^{1}\Vert^{2}\Delta y^{1} + \nu_{1}\Vert\nabla y^{2}\Vert^{2}\Delta y^{2} + (u\cdot\nabla)y^{1} + (y^{2}\cdot\nabla)u\vspace{0.1cm}\\
    + \nabla q = \nu_{0}we_{N}, \,  \nabla\cdot u=0    & \text{in} & Q, \\
    w_{t}-\nu_{0}\Delta w -\nu_{1}(\Vert\nabla y^{1}\Vert^{2})\Delta\theta^{1}-\Vert\nabla y^{2}\Vert^{2}\Delta\theta^{2}) + u\cdot\nabla\theta^{1} + y^{2}\cdot\nabla w\vspace{0.1cm}\\
    = \nu_{0}Dy^{1}:\nabla y^{1} - \nu_{0}Dy^{2}:\nabla y^{2} + \nu_{1}\Vert\nabla y^{1}\Vert^{2}Dy^{1}:\nabla y^{1} - \nu_{1}\Vert\nabla y^{2}\Vert^{2}Dy^{2}:\nabla y^{2} &\text{in}& Q,\\
    u(x,t)=0, w(x,t)=0 &\text{on}& \Sigma,\\
    u(x,0)=0,\, w(x,0)=0 &\text{in}& \Omega.
    \end{array}\right.
\end{equation*}
Which we can rewrite as follows
\begin{equation}\label{prob unicidade}
    \hspace{-0.1cm}\left\{\hspace{-0.1cm}
    \begin{array}[c]{lll}
    u_{t}-\nu_{0}\Delta u - \nu_{1}[\Vert\nabla y^{1}\Vert^{2}\Delta u + (\Vert\nabla y^{1}\Vert + \Vert\nabla y^{2}\Vert)(\Vert\nabla y^{1}\Vert - \Vert\nabla y^{2}\Vert)\Delta y^{2}]\vspace{0.1cm}\\
    + (u\cdot\nabla y^{1}) + (y^{2}\cdot\nabla)u + \nabla q = \nu_{0}we_{N}, \,  \nabla\cdot u=0    & \text{in} & Q,\vspace{0.1cm} \\
    w_{t}-\nu_{0}\Delta w - \nu_{1}[\Vert\nabla y^{1}\Vert^{2}\Delta w + (\Vert\nabla y^{1}\Vert + \Vert\nabla y^{2}\Vert)(\Vert\nabla y^{1}\Vert - \Vert\nabla y^{2}\Vert)\Delta \theta^{2}] \vspace{0.1cm}\\
    + u\cdot\nabla\theta^{1} + y^{2}\cdot\nabla w = \nu_{0}(Du:\nabla y^{1} + Dy^{2}:\nabla u) + \nu_{1}[\Vert\nabla y^{1}\Vert^{2} Du:\nabla y^{1}\vspace{0.1cm}\\
    + (\Vert\nabla y^{1}\Vert + \Vert\nabla y^{2}\Vert)(\Vert\nabla y^{1}\Vert - \Vert\nabla y^{2}\Vert)Dy^{2}:\nabla u] &\text{in}& Q,\\
    u(x,t)=0, w(x,t)=0 &\text{on}& \Sigma,\\
    u(x,0)=0,\, w(x,0)=0 &\text{in}& \Omega.
    \end{array}\right.
\end{equation}
Multiplying by $-\Delta u$ and $w$ in first and second line of (\ref{prob unicidade}), respectively, and integrating in $\Omega$, we obtain
\begin{equation*}
    \begin{array}{l}
       \dfrac{1}{2}\dfrac{d}{dt}(\Vert\nabla u\Vert^{2} + \Vert w\Vert^{2}) + (\nu_{0} + \nu_{1}\Vert\nabla y^{1}\Vert^{2})(\Vert\Delta u\Vert^{2} + \Vert\nabla w\Vert^{2}) \vspace{0.1cm}\\
       =\displaystyle\int_{\Omega}(\Vert\nabla y^{1}\Vert + \Vert\nabla y^{2}\Vert)(\Vert\nabla y^{1}\Vert-\Vert\nabla y^{2}\Vert)\Delta y^{2}\Delta u + \displaystyle\int_{\Omega}[(u\cdot\nabla y^{1})\Delta u\vspace{0.1cm}\\
       + (y^{2}\cdot\nabla u)\Delta u] - \displaystyle\int_{\Omega}we_{3}\Delta u + \displaystyle\int_{\Omega}(\Vert\nabla y^{1}\Vert + \Vert\nabla y^{2}\Vert)(\Vert\nabla y^{1}\Vert - \Vert\nabla y^{2}\Vert)\nabla\theta^{2}\nabla w \vspace{0.1cm}\\
       -\displaystyle\int_{\Omega}u\cdot\nabla\theta^{1}w + \displaystyle\int_{\Omega}\nu_{0}[(Du:\nabla y^{1})w + (Dy^{2}:\nabla u)w]\vspace{0.1cm}\\
       +\displaystyle\int_{\Omega}(\Vert\nabla y^{1}\Vert + \Vert\nabla y^{2}\Vert)(\Vert\nabla y^{1}\Vert-\Vert\nabla y^{2}\Vert)Dy^{2}:\nabla u\, w = \displaystyle\sum_{i=1}^{7}L_{i}.
    \end{array}
\end{equation*}
Notice that,
\begin{equation*}
    \hspace{-0.5cm}\begin{array}{l}
        \vert L_{1}\vert\leq \displaystyle\int_{\Omega}C(\Vert\nabla y^{1}\Vert + \Vert\nabla y^{2}\Vert)\Vert\nabla u\Vert \vert\Delta y^{2}\vert\vert\Delta u\Vert\vspace{0.1cm} \\
        \hspace{0.7cm}\leq C(\Vert\nabla y^{1}\Vert + \Vert\nabla y^{2}\Vert)^{2}\Vert\nabla u\Vert^{2}\Vert\Delta y^{2}\Vert^{2} + \dfrac{1}{\epsilon}\Vert\Delta u\Vert^{2};         
    \end{array}
\end{equation*}
Since $N\leq 3$, by continuous embedding $H^{2}(\Omega)\hookrightarrow L^{\infty}(\Omega)$, $H^{1}(\Omega)\hookrightarrow L^{4}(\Omega)$ and inequality $\Vert f\Vert_{H^{2}(\Omega)}\leq C\Vert\Delta f\Vert$ to any $f\in H^{2}(\Omega)\cap H^{1}_{0}(\Omega)$, we achieved
\begin{equation}\label{des unicidade}
    \begin{array}{l}
       \dfrac{1}{2}\dfrac{d}{dt}(\Vert\nabla u\Vert^{2} + \Vert w\Vert^{2}) + (\nu_{0} + \nu_{1}\Vert\nabla y^{1}\Vert^{2})(\Vert\Delta u\Vert^{2} + \Vert\nabla w\Vert^{2}) \vspace{0.2cm}\\
       \leq  \dfrac{6}{\epsilon}\Vert\Delta u\Vert^{2} + C L_{8}\Vert\nabla u\Vert^{2} + C L_{9}\Vert w\Vert^{2},
       \end{array}
\end{equation}
where $$L_{8}= (\Vert\nabla y^{1}\Vert + \Vert\nabla y^{2}\Vert)^{2}(\Vert\Delta y^{2}\Vert^{2}+1) + \Vert\nabla\theta^{2}\Vert^{2} 
     + \Vert\Delta y^{1}\Vert^{2}+ \Vert\Delta y^{2}\Vert^{2} $$ and $$L_{9}= 1+\Vert\nabla\theta^{1}\Vert^{2} + \Vert\Delta y^{1}\Vert^{2} + \Vert\Delta y^{2}\Vert^{2} 
      + (\Vert\nabla y^{1}\Vert + \Vert\nabla y^{2}\Vert)^{2}(\Vert\nabla y^{1}\Vert^{2} + \Vert\nabla y^{2}\Vert^{2})\Vert\Delta y^{2}\Vert^{2}.$$
Taking $\epsilon=12/\nu_{0}$ and integrating (\ref{des unicidade}) from $0$ to $t$,
\begin{equation*}
    \begin{array}{l}
       \Vert\nabla u(t)\Vert^{2} + \Vert w(t)\Vert^{2} + \displaystyle\int_{0}^{t}(\nu_{0} + \nu_{1}\Vert\nabla y^{1}\Vert^{2})(\Vert\Delta u(s)\Vert^{2} + \Vert\nabla w(s)\Vert^{2})ds \vspace{0.1cm} \\
       \leq 0 + \displaystyle\int_{0}^{t}C(L_{8}+L_{9})(\Vert\nabla u(s)\Vert^{2} + \Vert w(s)\Vert^{2})ds.
    \end{array}
\end{equation*}
Hence, applying Gronwall's Lemma, we obtain 
 the uniqueness of the solution.
 \hfill
\end{proof}

\begin{proof}[
\textit{\textbf{Proof of Lemma \ref{energy decay}}}]
Continuing as in the proof of the existence of solution, we will make an additional estimate for the temperature term. More accurately, taking $w=-\Delta\theta_{m}(t)$ in the second equation of \eqref{prob lad.smag. aproximado}, using the inequalities $\Vert .\Vert_{L^{3}(\Omega)}\leq C\Vert.\Vert^{1/2}\Vert .\Vert^{1/2}_{H^{1}(\Omega)}$, $\Vert .\Vert_{L^{4}(\Omega)}\leq C \Vert.\Vert^{1/2}\Vert .\Vert^{1/2}_{H^{1}(\Omega)}$ and \eqref{estimate 1 N=2,3} we deduce
\begin{equation}\label{estimate de energia para decaimento}
    \begin{array}{l}
   \dfrac{1}{2}\dfrac{d}{dt}(\lambda_{1}\Vert\nabla y_{m}\Vert^{2} + \Vert\theta_{m}\Vert^{2} + \Vert\nabla\theta_{m}\Vert^{2}) + \dfrac{\nu_{0}}{4}( \Vert\nabla\theta_{m}\Vert^{2}  +  \Vert\Delta\theta_{m}\Vert^{2} )\vspace{0.1cm}\\

    + \nu_{1}\Vert\nabla y_{m}\Vert^{2}\Vert\Delta\theta_{m}\Vert^{2} + \lambda_{1}\nu_{1}\Vert\nabla y_{m}\Vert^{2}\Vert\Delta y_{m}\Vert^{2} + \dfrac{\lambda_{1}\nu_{0}}{8}\Vert\Delta y_{m}\Vert^{2}\vspace{0.1cm}\\
 +\left(\nu_{1}-\hat{C}_{6}\Vert\nabla y_{m}\Vert^{2}\right)\Vert\nabla y_{m}\Vert^{2}\Vert\nabla\theta_{m}\Vert^{2}

   +\left[\dfrac{\lambda_{1}\nu_{0}}{8} - \hat{C}_{3}\Vert\nabla y_{m}\Vert\right.\vspace{0.1cm}\\
   \left. - \hat{C}_{9}\Vert\nabla y_{m}\Vert^{2} -\hat{C}_{10}\Vert\nabla y_{m}\Vert^{6}\right] \Vert\Delta y_{m}\Vert^{2}
  \leq 0,   
   \end{array}
\end{equation}
where $\hat{C}_{9}=\lbrace \hat{C}_{4},\hat{C}_{7}\rbrace$ and  $\hat{C}_{10}=\lbrace \hat{C}_{5},\hat{C}_{8}\rbrace$. Remembering that $\hat{C}_{i}$, $i=\lbrace 1,\ldots, 10\rbrace$ are constants that may depend on $\Omega, \nu_{0}, \nu_{1}$ and $\lambda_{1}$.

In a similar way to what was done for \eqref{estimate 1 N=2,3} we can obtain that all terms on the left side of \eqref{estimate de energia para decaimento} are positive. Therefore, for any $t\in[0,T]$,
\begin{equation*}
    \begin{array}{c}
\dfrac{1}{2}\dfrac{d}{dt}(\lambda_{1}\Vert\nabla y_{m}(t)\Vert^{2} + \Vert\theta_{m}(t)\Vert^{2} + \Vert\nabla\theta_{m}(t)\Vert^{2}) 
  + \dfrac{\nu_{0}}{4}\Vert\nabla\theta_{m}(t)\Vert^{2}
  + \dfrac{\lambda_{1}\nu_{0}}{8}\Vert\Delta y_{m}(t)\Vert^{2} \leq  0
   \end{array}
\end{equation*}
which we will rewrite in the form
\begin{equation}\label{estimate com v=v0=0}
    \begin{array}{l}
   \dfrac{d}{dt}{\Phi_{m}}(t)
  + \dfrac{\nu_{0}}{2}\Vert\nabla\theta_{m}\Vert^{2}
    
    + \dfrac{\lambda_{1}\nu_{0}}{4}\Vert\Delta y_{m}\Vert^{2}  
  \leq  0,
   \end{array}
\end{equation}
where ${\Phi_{m}}(t)=\lambda_{1}\Vert\nabla y_{m}\Vert^{2} + \Vert\theta_{m}\Vert^{2} + \Vert\nabla\theta_{m}\Vert^{2}$. Note that, 
\begin{equation*}
    \begin{array}{l}
 {\Phi_{m}}(t)\leq \dfrac{\tilde{C}_{1}\nu_{0}}{2}{2}\Vert\nabla\theta_{m}\Vert^{2}
    + \dfrac{\tilde{C}_{2}\lambda_{1}\nu_{0}}{4}\Vert\Delta y_{m}\Vert^{2}\leq \hat{C}\left(\dfrac{\nu_{0}}{2}\Vert\nabla\theta_{m}\Vert^{2}
    + \dfrac{\lambda_{1}\nu_{0}}{4}\Vert\Delta y_{m}\Vert^{2}\right)
    \end{array}
\end{equation*}
where $\tilde{C}_{1}, \tilde{C}_{2} > 0$ and $\hat{C}=\max\lbrace\tilde{C}_{1},\tilde{C}_{2}\rbrace$. 
Thus, from \eqref{estimate com v=v0=0},
\[         \dfrac{d}{dt}{\Phi_{m}}(t)
  + \dfrac{1}{\hat{C}}{\Phi_{m}}(t)\leq 0 
\]
which results in $${\Phi_{m}}(t)\leq e^{(-1/\hat{C})t}{\Phi_{m}}(0)$$
and consequently it can be deduced \eqref{decaimento}.
\end{proof}




\end{appendices}



 \end{document}